\documentclass[11pt,a4paper,leqno,notitlepage]{amsart}
\usepackage[english,frenchb]{babel}
\usepackage{enumerate}
\usepackage{amsmath,todonotes}
\usepackage{amsfonts}
\usepackage{amssymb}
\usepackage{amsthm}
\usepackage{graphicx}
\usepackage{mathrsfs}
\usepackage{datetime}
\usepackage{comment}
\usepackage{fourier-orns}
\usepackage{dsfont}
\usepackage[cyr]{aeguill}
\usepackage{fancyhdr}
\usepackage{mathabx}
\usepackage[latin1]{inputenc}
\usepackage[all,cmtip]{xy}

\usepackage{tikz-cd}
\usetikzlibrary{decorations.pathmorphing}

\usepackage{amsfonts,amsmath,graphicx}

\usepackage{geometry}
\newcommand{\bc}{\begin{center}}
\newcommand{\ec}{\end{center}}
\newcommand{\bq}{\begin{quote}}
\newcommand{\eq}{\end{quote}}
\newcommand{\bqtn}{\begin{quotation}}
\newcommand{\eqtn}{\end{quotation}}
\newcommand{\beq}{\begin{equation}}
\newcommand{\eeq}{\end{equation}}
\newcommand{\bearr}{\begin{eqnarray}}
\newcommand{\eearr}{\end{eqnarray}}
\newcommand{\bearrn}{\begin{eqnarray*}}
\newcommand{\eearrn}{\end{eqnarray*}}
\newcommand{\bi}{\begin{itemize}}
\newcommand{\ei}{\end{itemize}}
\newcommand{\be}{\begin{enumerate}}
\newcommand{\ee}{\end{enumerate}}
\newcommand{\bthe}{\begin{theorem}}
\newcommand{\ethe}{\end{theorem}}
\newcommand{\blem}{\begin{lemme}}
\newcommand{\elem}{\end{lemme}}
\newcommand{\bsolu}{\begin{solution}}
\newcommand{\esolu}{\end{solution}}
\newcommand{\bexer}{\begin{exercise}}
\newcommand{\eexer}{\end{exercise}}

\newcommand{\ba}{\begin{array}}
\newcommand{\ea}{\end{array}}

\usepackage{amsfonts,amsmath}
\newtheorem{theoreme}{Theorem}[section]
\newtheorem{theorem}[theoreme]{Theorem}
\newtheorem{lemme}[theoreme]{Lemma}
\newtheorem{lemma}[theoreme]{Lemma}
\newtheorem{proposition}[theoreme]{Proposition}
\newtheorem{definition}[theoreme]{Definition}

\newtheorem{corollaire}[theoreme]{Corollary}

\newtheorem{solution}[theoreme]{Solution}
\newtheorem{exercise}[theoreme]{Exercise}
\newcommand{\bdefi}{\begin{definition}}
\newcommand{\edefi}{\end{definition}}
\newcommand{\brk}{\begin{remarque}}
\newcommand{\erk}{\end{remarque}}
\newcommand{\bpp}{\begin{proposition}}
\newcommand{\epp}{\end{proposition}}
\newcommand{\bpf}{\begin{proof}}
\newcommand{\epf}{\end{proof}}
\newcommand{\bcor}{\begin{corollaire}}
\newcommand{\ecor}{\end{corollaire}}
\newcommand{\bsol}{\begin{solution}}
\newcommand{\esol}{\end{solution}}
\theoremstyle{definition}

\newtheorem{remarque}[theoreme]{Remark}
\allowdisplaybreaks
\title{Cubic Fields: A Primer}
\author{Sophie Marques and Kenneth Ward}

\begin{document}
\large
\selectlanguage{english}
\maketitle
\begin{abstract}
We classify all cubic extensions of any field of arbitrary characteristic, up to isomorphism, via an explicit construction involving three fundamental types of cubic forms. We deduce a classification of any Galois cubic extension of a field. The splitting and ramification of places in a separable cubic extension of any global function field are completely determined, and precise Riemann-Hurwitz formulae are given. In doing so, we determine the decomposition of any cubic polynomial over a finite field. \end{abstract}

\noindent \quad {\footnotesize MSC Code (primary): 11T22}

\noindent \quad {\footnotesize MSC Codes (secondary):  11R32, 11R16, 11T55, 11R58}

\tableofcontents
\section*{Introduction}
In this paper, we give a complete classification of cubic field extensions up to isomorphism over an arbitrary field of any characteristic, which we had begun in  \cite{MWcubic}.  More precisely, in Corollary \ref{classification}, we prove that any cubic extension of an \emph{arbitrary} field admits a generator $y$, explicitly determined in terms of an arbitrary initial generating equation, such that
\begin{enumerate}
\item $y^3 =a$, with $a \in F$, or
\item
\begin{enumerate}
\item $y^3 -3y=a$, with $a \in F$, when $p\neq 3$, or
\item $y^3 +ay+a^2 = 0$, with $a \in F$, when $p=3$.
\end{enumerate}
\end{enumerate}
In \S $1.3$, we devise a procedure to compare, up to isomorphism, any two separable cubic extensions of a given field. When the characteristic is not equal to $3$, this employs the purely cubic closure which we determine in Theorem \ref{purelycubic}, whereas the Galois closure is needed in characteristic 3, eliminating the need for the brute force computations given in \cite{MWcubic}. As we demonstrate, the purely cubic closure is essential for the study of cubic extensions which are not pure in characteristic distinct from $3$. Furthermore, when the characteristic is equal to 3, the form (2)(a) we obtain can be viewed as a generalised Artin-Schreier form for separable extensions, where the techniques from Artin-Schreier theory can be adapted. We emphasize that all of our classifications (in any characteristic) are valid for any separable cubic extension. 

For Galois cubics extensions $L/F$, we prove in \S 1.4 that one of the following cases occurs, where each generator is again explicitly determined:
 \begin{enumerate}
 \item If $p=3$, then $L/F$ is an {\sf Artin-Schreier} extension; that is, there is a generator $y$ of $L/F$ such that $y^3 -y= a$ with $a \in F$. 
 \item If $p\neq 3$ and $F$ contains a primitive $3^{rd}$ root of unity, then $L/F$ is an {\sf Kummer} extension; that is, there is a generator $y$ of $L/F$ such that $y^3 =a$ with $a \in F$.
  \item If $p\neq 3$ and $F$ does not contain a primitive $3^{rd}$ root of unity, then $L/F$ is an extension with a generator $y$ of $L/F$ such that $y^3 -3y= \frac{2a^2+2a-1}{a^2+  a+1}$ with $a \in F$. 
 \end{enumerate}
Cases (1) and (2) are of course well known; we are able to deduce these again in a completely explicit and elementary way. Among Galois extensions, our primary focus is the third case, where Artin-Schreier and Kummer theory do not apply. We employ a generalised form of a Shank cubic polynomial in order to describe the Galois cubic extensions of the form (3) by way of an (explicit) one-to-one correspondence (Theorem \ref{shanksconversion}). We also describe the Galois action on generators of the form (3) (Corollary \ref{rootour}); this is a consequence of Theorem \ref{extensionsthesame1}, which determines when generators of the form (3) result in isomorphic extensions.

In \S 3, we describe the splitting and ramification for any place in a separable cubic extension of a global function field. To accomplish this, we use Kummer's theorem to determine splitting of cubic polynomials over a finite field, as all residue fields are finite. Thus, in \S 2, we characterise completely the decomposition of any cubic polynomial over a finite field, using earlier results of Dickson, Pommerening, and Williams \cite{Dickson,Pomm,Williams}. This study of decomposition permits us to obtain in \S 3.3 a Riemann-Hurwitz formula for any separable extension of a cubic global function field (Theorems \ref{RHPC}, \ref{RH}, \ref{char3RH}).

While the classification we give for ramification, splitting, and genera is presented here for global function fields, many of the methods employed remain valid over any global field. As an example, we present Proposition \ref{denom}, which characterises certain Galois cubic extensions and is proven separately over $\mathbb{Q}$ and $\mathbb{F}_q(x)$.  We intend to produce a second study which addresses number fields, as there are substantive differences between the two types of global fields.

\section{Classification of cubics over any field}

In this section, we wish to obtain a complete and concise description of generating equations for cubic fields. The classification we obtain yields a family of three types of generating equations, all of which require only one parameter and are depressed, i.e., possess no quadratic term ($\S 1.1$). This classification is valid over any field. We note that of the families we obtain when $p \neq 3$ have generating equation $X^3 - 3X - a$; in this case, when $-3$ is not a square in $F$, then the linear coefficient cannot be removed. We also provide a procedure permitting to determine whether any two cubic extensions are isomorphic (\S 1.3). When $p\neq 3$, in order to determine when two cubics are isomorphic, we make the use of the purely cubic closure (see Definition \ref{pureclosuredef}), which we determine precisely in $\S 1.2$. The purely cubic closure is important for studying impure cubic extensions (as we will see again in later sections, for instance, when we study ramification and splitting in \S 3). When $p=3$, the Galois closure is used in the same way to find the analogous requisite criteria for isomorphism of two cubic extensions. 

We then classify Galois cubics in any characteristic $p \neq 3$, which gives an analogue for all Galois cubic extensions which are not addressed in Artin-Schreier and Kummer theory. We show that the form we obtain in this case can be viewed as equivalent to a weaker form of a Shanks polynomial \cite{Shanks}. Curiously, we note that in the form $$X^3 - 3X - \frac{2a^2 +2a-1}{a^2 +a^2+1},$$ which we obtain when $p \neq 3$ and the extension is impurely cubic, the denominator $a^2+a+1$ of the constant coefficient is equal to an evaluation $X^2 + X + 1|_{X=a}$ of the cyclotomic polynomial for primitive $3^{rd}$ roots of unity.

Henceforth, we let $F$ denote a field and $p = \text{char}(F)$ the characteristic of this field, where we admit the possibility $p = 0$ unless stated otherwise. We let $\overline{F}$ denote the algebraic closure of $F$. 

\begin{definition} \label{pureclosuredef} \begin{itemize}
\item [$\bullet$] If $p \neq 3$, a generator $y$ of a cubic extension $L/F$ with minimal polynomial of the form $X^3 -a$ $(a \in F)$ is called a {\sf purely} cubic generator, and $L/F$ is called a {\sf purely} cubic extension. If $p = 3$, such an extension is simply called {\sf purely inseparable}.
\item [$\bullet$] If $p\neq 3$ and a cubic extension $L/F$ does not possess a generator with minimal polynomial of this form, then $L/F$ is called {\sf impurely} cubic. 
\item [$\bullet$] For any cubic extension $L/F$, we define the {\sf purely cubic closure} of $L/F$ to be the smallest extension $F'$ of $F$ such that $LF'/F'$ is purely cubic.	
\end{itemize}
\end{definition}

\subsection{Generating polynomials} 
In the next theorem, we prove that the irreducibility of any cubic polynomial over any field $F$ depends upon that of polynomials of the form
\begin{enumerate}
\item $X^3 -a$, with $a \in F$,
\item
\begin{enumerate}
\item $X^3 -3X-a$, with $a \in F$, if $p\neq 3$, or
\item $X^3 +aX+a^2$, with $a \in F$, if $p=3$.
\end{enumerate}
\end{enumerate}
The criteria we give therefore depend on the field characteristic, particularly whether or not the characteristic is 3.

\begin{theoreme}\label{irred} 
Let $T(X)= X^3 + e X^2 + f X +g$ be any cubic polynomial with coefficients in $F$. Then one of the following is satisfied:
\begin{enumerate} 
\item $g=0$ or $-27g^2-2f^3+9egf=0$ and $T(X)$ is reducible. 
\item $g \neq 0$, $ -27g^2-2f^3+9egf\neq 0$, $3eg = f^2$ and $T(X)$ is reducible if, and only if, $R(X) = X^3 -a $ is reducible, where $a =  \frac{27g^3}{-27g^2+f^3}$. 
\item $p\neq 3$, $g \neq 0$, $ -27g^2-2f^3+9egf\neq 0$, and $3eg \neq f^2$, and $T(X)$ is reducible if, and only if,  $S(X)= X^3 -3X-a$ is reducible where $a= -2-\frac{ ( 27 g^2 -9efg +2f^3)^2}{(3ge-f^2)^3}$.
\item $p=3$, $f\neq 0$, and $g \neq 0$, and $T(X)$ is reducible if, and only if, $-f^2e^2+ge^3+f^3=0$ or  $G(X)= X^3 + a X+ a^2 $ is reducible, where 
\begin{itemize} 
\item[$\cdot$] $a =  \frac{g^2}{f^3}$ when $e=0$, 
\item[$\cdot$] $a =  \frac{-f^2e^2+ge^3+f^3}{e^6}$ when $e\neq 0$ and $-f^2e^2+ge^3+f^3\neq 0$.
\end{itemize}
\end{enumerate}
\end{theoreme} 
\begin{proof} 
Let $x \in \overline{F}$.
\begin{enumerate} 
\item Suppose that $g=0$. Then $T(X)= X^3 + e X^2 + f X= X(X^2 + e X + f)$, whence $T(X)$ is reducible.  Now, suppose that $g\neq 0$ and $-27g^2-2f^3+9egf=0$, then 
$$0 = \left(\frac{f^3}{g}\right)\left[\left(-\frac{3g}{f}\right)^3 + e\left(-\frac{3g}{f}\right)^2 + f\left(-\frac{3g}{f}\right) + g \right]= \left(\frac{f^3}{g}\right)T\left(-\frac{3g}{f}\right),$$ whence $T(X)$ is reducible with $-\frac{3g}{f}\in F$ as a root. 
\item Suppose $g \neq 0$, $ -27g^2-2f^3+9egf\neq 0$, and $3eg=f^2$. Then $T(x)=0$ if, and only if, $y^3 = \frac{27g^3}{-27g^2+f^3} \in F$, where $$y = \frac{3g x}{fx + 3g}.$$ We note that $ x = \frac{3g y}{3g-fy}$. We have $x \neq \frac{-3g}{f}$ as $-27g^2-2f^3+9egf\neq 0$. 
Moreover, $x\in F$ if, and only if, $y \in F$. It follows that $T(X)$ is irreducible if, and only if, $-27g^2+f^3$  is not a cube in $F$. 
\item Suppose $p \neq 3$, $g \neq 0$, $ -27g^2-2f^3+9egf\neq 0$, and $3eg -f^2\neq 0$,  $T(x)=0$ if, and only if, $y^3 -3y = a$ where $y= \frac{-(6efg-f^3-27g^2)x+3g(3eg-f^2)}{(3eg-f^2)(fx+3g)}$ and $a=-2-\frac{ ( 27 g^2 -9efg +2f^3)^2}{(3ge-f^2)^3}$. 

Note that 
\begin{itemize}
\item[$\cdot$] $x=\frac{-3g(y-1)(3eg-f^2)}{f(3eg-f^2)y+6egf-f^3-27g^2}$, 
\item[$\cdot$] $x \neq \frac{-3g}{f}$ since $-27g^2-2f^3+9egf\neq 0$ and 
\item[$\cdot$] $y \neq -\frac{6egf-f^3-27g^2}{f(3eg-f^2)}$, since $g \neq 0$ and $-27g^2-2f^3+9egf \neq 0$.
\end{itemize} 
Thus, $T(X)$ is irreducible if, and only if, $S(X)= X^3 -3X-a$ is irreducible where $a= -2-\frac{ ( 27 g^2 -9efg +2f^3)^2}{(3ge-f^2)^3}$.
\item Suppose $p=3$, $3eg\neq f^2$, that is $f\neq 0$ since $p=3$ and $g\neq 0$.
\begin{enumerate} 
\item[$\cdot$] If $e =0$, then $T(x)=0$ if, and only if, $y^3+a y + a^2=0$ where $y = \frac{g}{f^2} x$ and $ a= \frac{g^2}{f^3}$. Note that $ x = \frac{f^2}{g} y$. 
\item[$\cdot$] Suppose $e \neq 0$ and $T(x)=0$. Note that  when $x= \frac{f}{e}$, then $-f^2e^2+ge^3+f^3=0$ and $T(X)$ is reducible. Suppose that $-f^2e^2+ge^3+f^3\neq 0$ then $T(x) =0$ if, and only if, $y^3+a y + a^2=0$ where $y= \frac{-f^2e^2+ge^3+f^3}{e^4(ex-f)}$ and $a=\frac{-f^2e^2+ge^3+f^3}{e^6}$. Note that $x =\frac{fe^4y-f^2e^2+ge^3+f^3}{e^5y}$.
\end{enumerate} 
\end{enumerate} 
\end{proof}
From this, we can deduce the following corollary without difficulty, which reduces the study of any cubic extension into exactly three one-parameter forms.
\begin{corollaire} \label{classification} Let $L/F$ be any cubic extension with generator $y$ with minimal polynomial $X^3 + e X^2 + f X + g$, where $e, f, g \in F$. Then there exists a generator $z$ of $L/F$ with minimal polynomial of the form 
\begin{enumerate}
\item  $X^3 -a$ with $a =  \frac{27g^3}{-27g^2+f^3}$ and $z=  \frac{3g y}{fy + 3g}$, when $3eg= f^2$.
\item \begin{enumerate} 
\item $X^3 -3X -a$ with $a=-2-\frac{ ( 27 g^2 -9efg +2f^3)^2}{(3ge-f^2)^3}$ and $z= \frac{-(6efg-f^3-27g^2)y+3g(3eg-f^2)}{(3eg-f^2)(fy+3g)}$, when $3eg \neq f^2$ and $p\neq 3$.
\item $X^3 + a X + a^2$ with 
\begin{itemize} 
\item[$\cdot$] $a =  \frac{g^2}{f^3}$ and $z= \frac{g}{f^2} y$, when $e=0$,
\item[$\cdot$] $a =  \frac{-f^2e^2+ge^3+f^3}{e^6}$ and $z= \frac{-f^2e^2+ge^3+f^3}{e^4(ey -f)}$, when $e\neq 0$ and $-f^2e^2+ge^3+f^3\neq 0$,
\end{itemize}
and $p=3$.
\end{enumerate} 
\end{enumerate} 
\end{corollaire} 
By the previous theorem, we know that when $3eg=f^2$, then $L/F$ is purely cubic when $p\neq 3$, and that $L/F$ is purely inseparable when $p=3$. 
\subsection{Purely cubic extensions} 
In order to obtain a complete classification of cubic extensions (up to isomorphism), we still need a criterion deciding when two cubic extensions are isomorphic. When $p\neq 3$, we have two possible types of generating equations: $y^3 - 3y -a$ and $y^3 -a$. So we first need to determine when the minimal equation $y^3 -3y -a=0$ represents a purely cubic extension. In this section, we obtain a criterion for the coefficients of the minimal polynomial of a cubic extension $L/F$ which determines whether or not the extension is purely cubic.
\begin{theoreme} \label{purelycubic} Suppose $p\neq 3$. Let $L/F$ be a cubic extension and $y$ a primitive element with minimal polynomial 
$$T(X)= X^3 -3X-a $$ where $a \in F$.
Then, $L/F$ is purely cubic if, and only if,  the polynomial $S(X)=X^2 +aX+1$ has a root in $F$. Let $c$ be a root of $S(X)$ in $\overline{F}$. In other words, $F(c)$ is the purely cubic closure for $L/F$. More precisely, $$u= \frac{cy-1}{y-c} $$ is a purely cubic generator for $L(c)/ F(c)$ such that $u^3 =  c$ if, and only if, $$y = \frac{cu-1}{u-c}$$ is a generator of $L/F$ with minimal polynomial $T(X)$.
\end{theoreme}
\begin{proof}
 Suppose that $p\neq 3$. Let $L/F$ be a cubic extension and $y$ a primitive element such that $y^3-3y-a=0$.
Any primitive element of $L/F$ is of the form $ u= \frac{ey + f}{gy+h}$ for $e,f, g, h  \in F$, since the elements of the set $\{y, uy, u ,1\}$ are linearly dependent over $F$. Moreover, $e, g$ are not both equal to $0$, since $u$ is a primitive element. We wish to determine which extensions admit a primitive element $u$ such that $u^3 =c$, for some $c \in F$. Note that $c$ cannot be a cube, since $L/F$ is a cubic extension. We have that $u^3 = c$ and $ u= \frac{ey + f}{gy+h}$ is equivalent to
$$(-e^3+cg^3)y^3+(-3fe^2+3cg^2h)y^2+(-3f^2e+3cgh^2)y+ch^3-f^3=0.$$
Equivalently, since $y^3=3y+a$, 
$$(-3fe^2+3cg^2h)y^2+(-3f^2e+3cgh^2-3e^3+3cg^3)y-ae^3+acg^3-f^3+ch^3=0.$$
As $\{1,y, y^2\}$ forms a basis of $L/F$, we then have that
$$\left\{ \begin{array}{llll} 
-3fe^2+3cg^2h &=& 0 & (1) \\ 
-3f^2e+3cgh^2-3e^3+3cg^3&=& 0 & (2)  \\ 
-ae^3+acg^3-f^3+ch^3 &=& 0 & (3) \\ 
\end{array} \right.$$ 
Note that $f$ and $g $ are nonzero. Indeed, if $f =0$, then by $(1)$ either $g=0$  or $h=0$ since $c \neq 0$. When $g=0$, by $(2)$, $e=0$, but $g$ and $e$ cannot be both zero, hence this case is impossible. When $h=0$, $(2)$ implies that $c$ is a cube, which is also impossible. Similarly, if $g=0$, since $e$ and $g$ are not both zero, $(1)$ implies that $f$ is zero, and by the previous argument this is impossible. 

Since $f$ and $g $ are nonzero, without loss of generality, one can suppose that $f=-1$ and $g=1$, as $u$ is a purely cubic generator, then $l u$ is also a purely cubic generator, for any nonzero $l \in F$. Replacing these values in the previous system yields the new system
$$\left\{ \begin{array}{llll} 
e^2+ch &=& 0 & (1) \\ 
-e+ch^2-e^3+c&=& 0 & (2)  \\ 
1+ch^3-ae^3+ac &=& 0 & (3) \\ 
\end{array} \right.$$ 
From $(1)$, $h = -\frac{e^2}{c}$ since $c \neq0$, and substitution in $(2)$ and $(3)$ lead to 
$$\left\{ \begin{array}{llll} 
(c-e)(-e^3+c)&=& 0 & (1) \\ 
(-e^3+c)(ac^2+c+e^3)&=& 0 & (2)  \\ 
\end{array} \right.$$ 
From $(1)$ and the fact that $c$ is a noncube,  we find $e=c$, and substituting in $(2)$, we obtain $c^2 + ac +1=0$. Moreover, $h=-\frac{e^2}{c}=-c$. Thus, in order for $L/F$ to be purely cubic, the polynomial $X^2 + a X+1$ needs to have a root $c$ in $F$, and if this is the case, then
$$u=\frac{cy-1}{y-c} $$  
is such that $$u^3 = c.$$
Thus the theorem.
\end{proof} 
We note that in the Galois case, this purely cubic closure is simply an extension by a primitive third root of unity. 
\begin{corollaire} \label{purelycubicclosurecorollary}
Suppose that $F$ does not contain a primitive $3^{rd}$ root of unity. Let $L/F$ be a Galois geometric extension and $y$ a primitive element such that $f(y) = y^3-3y-a=0$. Then $F(\xi)$ is purely cubic closure for $L/F$. 
\end{corollaire}
\begin{proof} 
Suppose $F$ does not contain a primitive $3^{rd}$ root of unity. Let $L/F$ be a Galois geometric extension and $y$ a primitive element such that $f(y) = y^3-3y-a=0$. Let $c_\pm$ denote the two roots of the quadratic polynomial $S(X)=X^2+ aX+1$. Let $r$ be a root of the quadratic resolvent  $$R(X)= X^2 +3aX + ( -27 + 9 a^2)$$ of the cubic polynomial $X^3 -3X-a$. The element $r$ lies in $F$ by \cite[Theorem 2.3]{Con}, as $L/F$ is by supposition Galois. 

We will prove that one can find a root of the polynomial $S(X) $ which is of the form $ur  +v$ with $u$ and $v \in F(\xi)$. We note that
\begin{align*} S(u r +v)& =  (u r +v)^2 +a (u r +v)+1 \\
&= u^2 r^2 +2uvr + v^2 +aur + av +1 \\
&= u^2 (-3ar - ( -27 + 9 a^2)) +2uvr + v^2 +aur + av +1\\ 
&= ru(-3au+2v+a)+ u^2 ( 27 - 9 a^2)+v^2 +av +1 =0\end{align*}
We then determine whether there exists a solution to the system
\begin{align*}
-3au+2v+a &=0 \\ 
 u^2 ( 27 - 9 a^2)+v^2 +av +1&= 0
\end{align*}
Thus,
$$u= \frac{ 2v+a}{3a}$$
and 
\begin{align*}0&= \left( \frac{ 2v+a}{3a} \right)^2 9 ( 3 -  a^2)+v^2 +av +1\\ 
&= \frac{(4v^2 + 4va + a^2 )}{a^2}  ( 3 -  a^2)+v^2 +av +1\\ 
&= 3\left(\frac{4}{a^2}-1\right) v^2 + 3a \left(\frac{4}{a^2}-1\right)v+ a^2 \left(\frac{4}{a^2}-1\right).
\end{align*}
Therefore,
$$3v^2 + 3av+ a^2 =0$$
Note that if $\xi$ is a primitive $3^{rd}$ root of unity, then $v_+=\frac{a ( \xi -1) }{3}$ and  $v_-=\frac{a (- \xi -2) }{3}$ are the two roots of the previous equation, which yields
$u_+ = \frac{2 \xi +1 }{9}$ and $u_- = \frac{-2 \xi -1 }{9}$. Thus, the two roots of the polynomial $X^2+aX+1$ are given by 
$$c_+ = \left(\frac{2}{9}r+\frac{1}{3}a\right)\xi+\frac{1}{9}r-\frac{1}{3}a \quad \text{and} \quad c_-= \sigma (c_+) = \left(-\frac{2}{9}r-\frac{1}{3}a\right)\xi-\frac{1}{9}r-\frac{2}{3}a.$$
As a consequence,  $F(c_\pm )= F(\xi )$. 
\end{proof}

\subsection{Isomorphic cubics} 
Now, we consider two cubic extensions $L_1$ and $L_2$ of $F$, and we wish to give criteria which determine when they are isomorphic (we write $L_1\simeq L_2$ when this is so). 
\subsubsection{$p \neq 3$}
When $p\neq 3$, if $L_1\simeq L_2$,  then $L_1$ is purely cubic if, and only if, $L_2$ is purely cubic. In order to determine when $L_1$ or $L_2$ are purely cubic, we use Corollary \ref{classification} and Theorem \ref{purelycubic}. When both $L_1$ and $L_2$ are both purely cubic, we use the following result to determine if $L_1\simeq L_2$. Note that we do not assuming in the following lemma that any of the cubic extensions are Galois.
\begin{theoreme} \label{extensionsthesame}
Suppose that $L_1/F$ and $L_2/F$ are two cubic extensions, with $y_i$ a generator of $L_i/F$ and $y_i ^3 = a_i$ with $a_i\in F$, for each $i=1,2$. Then the following assertions are equivalent:
\begin{enumerate} 
\item $L_1\simeq L_2$,
\item $y_1 = cy_2^{j}$ where $j=1,2$ and $c\in F$,
\item $a_1 = c^3 a_2^j$ where $j=1,2$ and $c\in F$.
\end{enumerate}
\end{theoreme} 
\begin{proof} 
$(2)$ and $(3)$ are clearly equivalent. We now prove that $(1)$ and $(2)$ are equivalent.
If $y_1 = cy_2^{j}$ where $j=1,2$ and $c\in F$, then clearly, $L_1\simeq L_2$. For the converse, suppose then that $L_1\simeq L_2$. Thus $y_1 \in L(y_2)$, so that there are $e, f , g\in F$ such that $y_1= e y_2^2+fy_2 + g$. Thus
\begin{align*} y_1^3 &= (e y_2^2+fy_2 + g)^3\\ 
&= e^3y_2^6+3fe^2y_2^5+3e(ge+f^2)y_2^4+f(f^2+6ge)y_2^3+3g(ge+f^2)y_2^2+3g^2fy_2+g^3\\
&= (3eg^2+3f^2g+3fe^2a)y_2^2+(3g^2f+3e^2ag+3eaf^2)y_2+f^3a+6fage+e^3a^2+g^3,
\end{align*}
 which lies in $F$. As $\{ y_2^2 , y_2 , 1 \}$ is a basis of $L_2/F$, we therefore have the following system:
\begin{align*}
 3eg^2+3f^2g+3fe^2a_2 &= 0 \;\qquad (1) \\ 
 3g^2f+3e^2a_2g+3ea_2f^2&= 0. \qquad (2)
\end{align*}
If $g= 0$, then $3fe^2a_2=0$ in which case either $e=0$ or $f=0$. In both cases, the theorem is proven.  If $g\neq 0$ and $f=0$, then $3e g^2=0$ implies $e=0$, which is impossible, as $y_1 $  is a generator of $L_1/K$. Thus, if $g\neq 0$, then $f\neq 0$. If, on the other hand, $g$ and $f$ are not $0$, then we may compute $f \cdot (1)- g \cdot (2)$, which yields
$$ -3e(-g^3+f^3a_2)=0.$$
As $y_2$ defines a cubic extension, $a_2$ is not a cube; thus $e=0$, and by $(1)$, $3f^2 g=0$, which implies that either $f=0$ or $g=0$. If $e=0$ and $f=0$, then $y_1$ does not generate $L_1$, a contradiction. We are thus left with the case  $e=0$ and $g=0$, whence $y_1$ satisfies the theorem once again. \end{proof}
We easily obtain from this the following corollary, which describes the Galois action on a purely cubic generator in the Galois closure of a purely cubic extension.
\begin{corollaire}\label{GApure}
Let $y_1$ and $y_2$ be two distinct roots of an irreducible polynomial $X^3-3X-a$ in $\overline{F}$. Then $y_1 = \xi y_2$ where $\xi$ is a primitive $3^{rd}$ root of unity.
\end{corollaire}
Now, we suppose that $L_1$ is impurely cubic, and that $L_1\simeq L_2$, whence $L_2$ is also impurely cubic. Moreover, since $p\neq 3$, $L_i$ has a generator $y_i$ such that $y_i^3 -3 y_i =a_i$ where $a_i \in F$ by Corollary \ref{classification}, and the polynomial $X^2+a_i X+1$ is irreducible over $F$ by Theorem \ref{purelycubic}. The following result describes when the two extensions $L_1$ and $L_2$ are isomorphic, which concludes the classification of cubics up to isomorphism when $p\neq 3$. 
\begin{theoreme} \label{extensionsthesame1}
Suppose that $L_1/F$ and $L_2/F$ are two cubic extensions, with $y_i$ a generator of $L_i/F$ and $y_i ^3-3y_i = a_i$ with $a_i\in F$ such that the polynomial $X^2+a_i X+1$ is irreducible over $F$, for each $i=1,2$. Then the following assertions are equivalent:
\begin{enumerate} 
\item $L_1\simeq L_2$,
\item $y_1 = \alpha y_2^2 + \beta y_2 -2 \alpha $, where $\alpha, \beta \in F$ such that $\alpha^2 + a_2 \alpha \beta + \beta^2 =1$  and
\item $ a_1=-3a_2\alpha^2\beta+a_2\beta^3+6\alpha+\alpha^3a_2^2-8\alpha^3$, where $\alpha, \beta \in F$ such that $\alpha^2 + a_2 \alpha \beta + \beta^2 =1$.
\end{enumerate}
\end{theoreme} 
\begin{proof} 
Suppose that $L_1/F$ and $L_2/F$ are two cubic extensions, with $y_i$ a generator of $L_i/F$ and $y_i ^3-3y_i = a_i$ with $a_i\in F$, for each $i=1,2$. We let $c_i$ be a root of the polynomial $X^2 + a_i X +1$, for each $i=1,2$.  

Suppose $L_1\simeq L_2$. We have that $L_1(c_1)/ F(c_1)$ is purely cubic by Theorem \ref{purelycubic} and since $L_1 \simeq L_2$, then $L_1(c_1) \simeq L_2(c_2)$; thus $L_2(c_2)/ F(c_1)$ is purely cubic and $F(c_1)= F(c_2)$ since $c_2 \in F(c_1)$. By Theorem \ref{purelycubic}, we then know that $$u_i = \frac{c_i y_i-1}{y_i-c_i}$$ is a pure cubic generator of $L_i (c_i)/ F(c_i)$ such that $u_i^3 = c_i$, $i=1,2$. 
By Lemma \ref{extensionsthesame}, we have that $u_1 = d u_2^j$ where $j=1,2$ and $d\in F(c_1)=F(c_2)$ and $c_1 = d^3 c_2^j$.
Note that $u_i$ can also be expressed via the basis $1, \ y_i , \ y_i^2$ as 
$$u_i = \frac{c_i}{c_ia_i +2} (y_i^2 +c_i y_i -2)$$ Similarly,
 $$y_i= \frac{c_i u_i-1}{u_i-c_i} =-\frac{1}{c_i}u_i^2-u_i.$$
We have \begin{align*} u_i^2 &= \frac{c_i^2}{(c_ia_i +2)^2} (y_i^2 +c_iy_i -2)^2\\&=\frac{c_i^2}{(c_ia_i +2)^2}((c_i^2-1)y_i^2+(a_i+2c_i)y_i+2c_ia_i+4)\\&=-\frac{c_i^2}{(c_ia_i +2)}(y_i^2+\frac{1}{c_i} y_i-2). \end{align*}
When $j=1$,
\begin{align*}y_1& = -\frac{1}{c_1}u_1^2-u_1=  -\frac{1}{dc_2}u_2^2-d u_2\\
&=\frac{c_2}{d(c_2a_2 +2)}(y_2^2+\frac{1}{c_2} y_2-2)- \frac{dc_2}{c_2a_2 +2} (y_2^2 +c_2y_2 -2) \qquad\qquad\qquad (*)\\
&=\frac{c_2(1-d^2)}{d(c_2a_2 +2)}y_2^2+\frac{(1-c_2^2d^2)}{d(c_2a_2 +2)}y_2-2\frac{c_2(1-d^2)}{d(c_2a_2 +2)} \\
 \end{align*}
We set $\alpha =  \frac{c_2(1-d^2)}{d(c_2a_2 +2)}$ and $\beta = \frac{(1-c_2^2d^2)}{d(c_2a_2 +2)}$. As $L_1\simeq L_2$ and $\{ y_2^2, \ y_2 ,\ 1\}$ form a basis, $\alpha $ and $\beta$ are in $F$. Moreover,
$$\alpha^2 + a_2 \alpha \beta + \beta^2 =1, $$ 
since $c_2^2 + a_2 c_2 +1=0$. 

When $j=2$, as $u_2^3=c_2$, we have
\begin{align*}y_1& = -\frac{1}{c_1}u_1^2-u_1=  -\frac{1}{dc_2^2}u_2^4-d u_2^2 =  -\frac{1}{d c_2}u_2-d u_2^2\\
&=-\frac{1}{d(c_2a_2 +2)}(y_2^2+c_2 y_2-2)+ \frac{dc_2^2}{c_2a_2 +2} (y_2^2 +\frac{1}{c_2}y_2 -2) \qquad \qquad \qquad (**)\\
&=\frac{(c_2^2d^2-1)}{d(c_2a_2 +2)}y_2^2+\frac{c_2(d^2-1)}{d(c_2a_2 +2)}y_2-2\frac{(c_2^2d^2-1)}{d(c_2a_2 +2)}\\
 \end{align*}
 As before, setting $\alpha = \frac{(c_2^2d^2-1)}{d(c_2a_2 +2)}$ and $\beta = \frac{c_2(d^2-1)}{d(c_2a_2 +2)}$, we have that $\alpha $ and $\beta$ are in $F$ and
 $$\alpha^2 + a \alpha \beta + \beta^2 =1.$$
 Conversely, if $y_1 = \alpha y_2^2 + \beta y_2 -2 \alpha $, where $\alpha, \beta \in F$ are such that $\alpha^2 + a_2 \alpha \beta + \beta^2 =1$, and $y_2$ is such that $y_2^3-3y_2-a_2=0$, then 
 \begin{align*}
 y_1^3 -3y_1&= 3\alpha (\alpha ^2+\alpha a_2\beta+\beta^2-1)y_2^2 \\& \qquad\qquad \qquad+3\beta(\alpha^2+a_2\alpha \beta+\beta^2-1)y_2-3a_2\alpha^2\beta+a_2\beta^3+6\alpha+\alpha^3a_2^2-8\alpha^3 
 \\ &=-3a_2\alpha^2\beta+a_2\beta^3+6\alpha+\alpha^3a_2^2-8\alpha^3.
 \end{align*} Clearly, $(2)$ and $(3)$ are also equivalent.
\end{proof}
The next corollary then describes the Galois action on the generator in the Galois closure. 
\begin{corollaire}\label{rootour}
Let $y_1$ and $y_2$ be two distinct roots of an irreducible polynomial $X^3-3X-a$ in $\overline{F}$. Then 
$$y_1 =-\frac{1+2f}{a} y_2^2 +f y_2 + \frac{2(1+2f)}{a}$$ where $f\in \overline{F}$ is a root of the polynomial
$$X^2 + X + \frac{a^2-1}{a^2-4} $$
\end{corollaire}
\begin{proof}
From Corollary \ref{GApure} and following the proof of Theorem  \ref{extensionsthesame1}, as $u_1=\xi u_2$, we have that $(*)$ is satisfied with $d=\xi$. Thus,
$$y_1 = \frac{c(1-\xi^2)}{\xi(ca +2)}y_2^2+\frac{(1-\xi^2c^2)}{\xi(ca +2)}y_2-2\frac{c(1-\xi^2)}{\xi(ca+2)} \qquad\qquad(1) $$
In $(1)$ we set $f = \frac{1-\xi^2c^2}{\xi(ca +2)}=-\frac{\xi c a+1}{\xi(ca +2)}$, then 
$$f^2 +f=   \frac{(\xi c a-1)^2 +(\xi c a-1)(ca+2) }{(ca +2)^2 }= \frac{(a^2-1)(ca+1)}{c^2a^2+4 ca+4}= \frac{a^2-1}{a^4-4} $$
since $c^2 = -ca-1$. Moreover, $\frac{2f+1}{a}= \frac{c(1-\xi^2)}{\xi(ca +2)}$. Hence the result.
\end{proof}
\begin{remarque} \label{rootourre}
Note that $f$ in Corollary \ref{rootour} is such that $f=\frac{-6+ a r}{3(a^2-4)}$, where $r$ is a root of the quadratic resolvent 
$$R(X)= X^2 +3 a X + (-27+9a^2)$$ of the polynomial $f(X)$.
\end{remarque}

\subsubsection{$p=3$} 
When $p=3$, if $L_1\simeq L_2$, then $L_1$ is separable if, and only if, $L_2$ is separable. Note that if $L_1$ and $L_2$ is inseparable, then $L_1\simeq L_2$, by the proof of Lemma \ref{extensionsthesame}, since $(1)$ and $(2)$ in the proof are always satisfied. If $L_1$ and $L_2$ are separable, then the following Theorem accomplishes what we need, in accordance with Corollary \ref{classification}. 
\begin{theoreme} \label{char3extensionsthesame}
Suppose that $p= 3$, and let $L_i = F(y_i )/F$ $(i = 1, 2)$ be two separable extensions of degree $3$ of the form $y_i^3 + a_i y_i + a_i^2=0$, $a_i \in F$, $i = 1, 2$. Then the following statements are equivalent:
\begin{enumerate}
\item $L_1 \simeq L_2$.
\item $y_2 =-\beta (\frac{j}{a_1}y_1+ \frac{1}{a_1} w )$
  where $w\in F$, $j=1,2$ and $\beta = - j a_1 - (\frac{ w^3}{a_1} +w) \in F$;
\item  $$a_2  = \frac{(ja_1^2 + (w^3 + a_1 w) )^2}{a_1^3},$$ where $j=1,2$ and $w\in F$.
\end{enumerate}
\end{theoreme}

\begin{proof}
Via evaluation of the discriminant of the polynomial $X^3 +a_iX +a_i^2$, the Galois closure of $L_i$ is seen to be $L_i ( b_i)$ where $b_i^2 =-a_i$. Moreover, $L_i(b_i)/ F_i (b_i)$ is an Artin-Schreier extension with Artin-Schreier generator $\frac{y_i}{b_i}$ possessing minimal polynomial $X^3 - X +b_i$. Suppose that $L_1\simeq L_2$. Then $L_1(b_1) = L_2(b_2)$ is the common Galois closure of $L_1/F$ and $L_2/F$, and $\frac{y_1}{b_1}$ and $\frac{y_2}{b_2}$ are two Artin-Schreier generators of the same Artin-Schreier extension. Thus, by \cite[Proposition 5.8.6]{Vil}, we know that $ \frac{y_2}{b_2} =j \frac{y_1}{b_1}+ c$ with $1 \leq j \leq 2$ and $c \in F$, and that 
 $$ b_2 = jb_1 + (c^3 - c). \qquad\qquad\qquad \ (*)$$ We have
 \begin{align*} y_2 &= b_2 \left(j \frac{y_1}{b_1}+ c\right)= \frac{jb_2}{b_1}y_1+ b_2 c=-\frac{jb_2b_1}{a_1}y_1+ b_2 c\\& =-\frac{jb_2b_1}{a_1}y_1 - \frac{b_2b_1}{a_1} b_1c =-b_1b_2\left(\frac{j}{a_1}y_1+ \frac{1}{a_1} b_1c \right).\end{align*}
As $y_2\in F(y_1)$, we have that $b_2 c\in F$, $\frac{b_2}{b_1}\in F$, whence $w:=b_1c\in F$. Multiplication of $(*)$ by $b_1$ yields
 \begin{align*}  \beta:= b_1b_2 &= -ja_1 + (b_1c^3 - b_1c)\\
 &=- ja_1- \left(\frac{w^3}{a_1} + w\right). \\
 \end{align*}
 Thus,
 $$a_2 = \frac{(-ja_1 - (\frac{w^3}{a_1} + w) )^2}{a_1} = \frac{(ja_1^2 + (w^3 + a_1 w) )^2}{a_1^3} $$
 and
\begin{equation} y_2 = - \beta \left(\frac{j}{a} y_1 + \frac{1}{a_1}w\right) .\label{thisthing} \end{equation}
Conversely, suppose that \eqref{thisthing} holds, where $w\in F$ and $\beta =-  j a_1 - (\frac{ w^3}{a_1} +w)$. Then, since $y_1^3= -a_1y_1 -a_1^2$, we have
 \begin{align*} y_2^3&= -\beta^3 ( \frac{j}{a_1}y_1+ \frac{1}{a_1} w )^3 = -\beta^3  \frac{j}{a_1^3}y_1^3- \frac{\beta^3}{a_1^3} w^3  = \beta^3  \frac{j}{a_1^2}y_1 + \beta^3  \frac{j}{a_1}- \frac{\beta^3}{a_1^3} w^3 \\ 
 &=- \frac{ \beta^2}{a_1} \left( \beta \left( \frac{j}{a_1} y_1 + \frac{1}{a_1}w \right) \right) +\frac{\beta^3 }{a_1^2}\left(  w +ja_1 - \frac{w^3}{a_1}\right)  \\ 
&=-\frac{ \beta^2}{a_1}  y_2 -\frac{ \beta^4}{a_1^2}   =-a_2 y_2 - a_2^2. \end{align*}
where $a_2= \frac{ \beta^2}{a_1} $.
Finally, one may easily compute that conditions $(2)$ and $(3)$ are equivalent. 
 \end{proof}
 In this case, we also deduce from Lemma \ref{char3extensionsthesame} the Galois action on the generator in the Galois closure via the following corollary. 
 \begin{corollaire} Let $p = 3$, and let $y_1$ and $y_2$ be two distinct roots in $\overline{F}$ of an irreducible polynomial $X^3+aX+a^2$. Then $y_1 =  y_2+ lb$, where $b^2 = -a$ and $l=1, 2$. 
 \end{corollaire} 
 \begin{proof} Following the notation in the proof of Theorem \ref{char3extensionsthesame}, we have that $\frac{y_1}{b} = \frac{y_2}{b} +l$  where $l=1,2$ and $b^2 = -a$ since $\frac{y_i}{b}$, $i=1,2$ are roots of an Artin-Schreier polynomial. Thus $y_1 =  y_2+ l b$. 
 \end{proof}
 \subsection{Galois cubics}
 We conclude this section with the classification of Galois cubics. In this section, we will prove that if $L/F$ is a Galois cubic extension, then
 \begin{enumerate} 
 \item If $p=3$, then $L/F$ is an {\sf Artin-Schreier} extension; that is, there is a generator $y$ of $L/F$ such that $y^3 -y-a$ with $a \in F$. 
 \item If $p\neq 3$ and $F$ contains a primitive $3^{rd}$ root of unity, then $L/F$ is an {\sf Kummer} extension; that is, there is a generator $y$ of $L/F$ such that $y^3 -a$ with $a \in F$.
  \item If $p\neq 3$ and $F$ does not contain a primitive $3^{rd}$ root of unity, then $L/F$ is an extension with a generator $y$ of $L/F$ such that $y^3 -3y- \frac{2a^2+2a-1}{a^2+  a+1}$ with $a \in F$. 
 \end{enumerate}
Parts $(1)$ and $(2)$ are well known; we include a proof to illustrate how these forms can be deduced directly from Corollary \ref{classification} to give a explicit Artin-Schreier (resp., Kummer) generator.
 \begin{theoreme} \label{AS}
Let $p=3$, let $L/F$ be a Galois extension of degree $3$. Then there is a primitive element $z$ such that its minimal polynomial is of the form $R(z)= z^3 -z -a$. Furthermore, this primitive element is explicitly determined.
\end{theoreme}
\begin{proof} 
By Corollary \ref{classification}, we know that there is a primitive element $y'$ such that its minimal polynomial is of the form $$S(X)=X^3 +b X + b^2.$$ The discriminant of such a polynomial is equal to $-4b^3$. As $L/F$ is Galois, the discriminant is a square, thus $-b$ is a square, say $-b= a^2$. With $y= y'/a$, it follows that $y^3 -y =a$. 
\end{proof} 
In the case that $p \neq 3$, a cubic extension $L/K$ being Galois, $F$ containing a primitive $3^{rd}$ root of unity and $L/K$ being purely cubic are closely related as explained in the following theorem.
\begin{theoreme} \label{purely}
Let $p \neq 3$. Let $L/F$ be a cubic extension.
\begin{enumerate} 
\item If $F$ contains a primitive $3^{rd}$ root of unity, then $L/F$ is Galois if, and only if, it is purely cubic.
\item If $L/F$ is purely cubic, then $L/F$ is Galois if, and only if, $F$ contains a primitive $3^{rd}$ root of unity.
\end{enumerate}
\end{theoreme}
\begin{proof} 
Let $p \neq 3$. Let $L/F$ be a cubic extension.
\begin{enumerate}
\item Suppose that $F$ contains a primitive $3^{rd}$ root of unity $\xi$. 

Suppose that $L/F$ is purely cubic. Then we may find a primitive element $y$ such that its minimal polynomial is $y^3 = a$, for some $a \in F$. In the usual way, the three elements $y$, $\xi y$, and $\xi^2 y$ are the roots of the minimal polynomial of $y$, and they are all contained in $L$. Thus, $L/F$ is Galois, and $\text{\emph{Gal}}(L/F) = \mathbb{Z}/3 \mathbb{Z}$.

Suppose now that $L/F$ is Galois, and let $y$ be a primitive element of $L/F$ with minimal equation $y^3 + e y^2 + fy +g=0$. By Corollary \ref{classification}, if $3eg= f^2$, then $L/F$ is purely cubic. Suppose then that $3eg\neq f^2$. By Corollary \ref{classification}, there exists a primitive element $z$ (which is explicitly determined) with minimal polynomial $T(X) = X^3 - 3 X -a$. 
By Corollary \ref{purelycubicclosurecorollary}, we have that $L/K$ is purely cubic, and Theorem \ref{purelycubicclosurecorollary} gives an explicit purely cubic genererator.
\item Suppose $L/F$ be a purely cubic extension, so that $y^3 = a$ for some $a \in F$. As in \cite[Theorem 2.3]{Con}, $L/K$ is Galois if, and only if,
\begin{enumerate}
\item[$\bullet$] $p \neq 2$ and the discriminant $\Delta = -27 a^2$ is a square in $F$, which in turn is equivalent to $F$ containing a primitive $3^{rd}$ root of unity, or
\item[$\bullet$] $p=2$ and the resolvent polynomial $R(X)= X^2 + a X + a^2$ is reducible, which is true if, and only if, $X^2 +X +1$ is reducible, i.e., $F$ contains a primitive $3^{rd}$ root of unity.
\end{enumerate} 
\end{enumerate} 
\end{proof} 
 Theorem \ref{purely} proves that when $F$ contains a primitive $3^{rd}$ root of unity, a cubic extension $L/F$ is Galois if, and only if, the extension $L/F$ is purely cubic. Particularly, when $F$ does not contain a primitive $3^{rd}$ root of unity and $L/F$ is Galois, then by Corollary \ref{classification}, $L/F$ admits a primitive element with minimal polynomial of the form $X^3 -3X-a$, where $a \in F$. It follows that when $p\neq 3$, it remains to identify the Galois extensions with a generator with minimal polynomial of the form $X^3 -3X-a$, where $a \in F$.

Shanks studied Galois cubic extensions of $\mathbb{Q}$ with generation $$y^3 +a y^2 -(a+3)y+1,$$ where $a \in \mathbb{Z}$ \cite{Shanks}. This led also to the definition of a \emph{Shanks} cubic function field \cite{RoWe} as a Galois cubic extension $L/ \mathbb{F}_q(x)$ with generating equation $y^3 +a y^2 -(a+3)y+1$, where $a \in  \mathbb{F}_q[x]$.  We note that over $\mathbb{Q}$ or $\mathbb{F}_q(x)$, the Shanks cubics so defined do not include all Galois cubic extensions. One may, however, show that any Galois cubic extension $L/F$ such that $F$ does not contain a primitive $3^{rd}$ root of unity admits a generator $y$ with minimal equation $X^3 +a X^2 -(a+3)X+1$, where $a\in F$. We give the proof of this, which A. Brumer shared with us in a helpful discussion; we will use this form to identify the missing Galois extensions.

 \begin{lemma}\label{shanks}
Let $p \neq 3$. Suppose that $F$ does not contain a primitive $3^{rd}$ root of unity. A Galois cubic extension $L/F$ has a primitive element $y$ with minimal polynomial $X^3 +a X^2 -(a+3)X+1$, where $a \in F$. Moreover, $ \sigma(y) = -1/(y-1)$, where $\sigma$ is a generator of $\text{\emph{Gal}}(L/F)$.
 \end{lemma} 
\begin{proof} 
Let $L = F(z)$, and let $\Sigma$ be a generator of the Galois group $\text{\emph{Gal}}(F(z)/F)$. Then $\sigma (z) = \Sigma (z) $, where $$\Sigma= \left( \begin{array}{cc} a & b \\ c & d \end{array} \right) \in GL_2(F),$$ where we denote $\Sigma(z)= (a z+ b)/(cz+d)$. As $\sigma$ is of order $3$, we have $\Sigma^3-I=0$, where $I$ denotes the $2 \times 2$ identity matrix, and the minimal polynomial of $\Sigma$ is a polynomial of degree $1$ or $2$ dividing $X^3 -1= (X-1) (X^2 + X+1)$. The polyomial $X^2 +X +1$ is irreducible, as $F$ does not contain a primitive $3^{rd}$ root of unity by assumption. Thus, the minimal polynomial is either $X-1$ or $X^2+X+1$. If the minimal polynomial is $X-1$, then $1$ is the only eigenvalue for $\Sigma$, and thus $\Sigma$ is either similar to $I$ or $$\left( \begin{array}{cc} 1 & 1 \\ 0 & 1 \end{array} \right),$$ but neither of these has order $3$, as $p \neq 3$. It follows that the minimal polynomial of $\Sigma$ is equal to $X^2 +X+1$ and 
therefore similar to the matrix $$\left( \begin{array}{cc} 0 & -1 \\ 1 & -1 \end{array} \right).$$ Thus, there is a matrix $S \in GL_2 (F)$ such that 
$$S^{-1} \Sigma S = \left( \begin{array}{cc} 0 & -1 \\ 1 & -1 \end{array} \right).$$ We let $y= S^{-1}(z)$. We obtain
$$\sigma (y)= S^{-1}\sigma(z)=S^{-1}\Sigma(z) = S^{-1} \Sigma S (y) = \left( \begin{array}{cc} 0 & -1 \\ 1 & -1 \end{array} \right) (y) = \frac{-1}{y-1}.$$
Thus, 
$$\sigma^2 (y) =  \frac{-1}{\sigma(y)-1}= \frac{y-1}{y}.$$
Moreover, $y$ is a root of the polynomial 
$$ X^3 - ( y + \sigma (y)+ \sigma^2 (y)) X^2 + (y \sigma (y) + y \sigma^2 (y) + \sigma (y) \sigma^2 (y)) X- y \sigma (y) \sigma^2 (y).$$ 
We note that
$$y + \sigma (y)+ \sigma^2 (y)= \frac{y^3-3y+1}{y(y-1)};$$ 
$$z \sigma (y) + y \sigma^2 (y) + \sigma (y) \sigma^2 (y)= \frac{y^3-3y^2+1}{y(y-1)}$$
and 
$$y \sigma (y) \sigma^2 (y)= -1.$$ 
Defining $$a = - \frac{y^3-3y+1}{y(y-1)},$$ we then have $$a +3 = - \frac{y^3-3y^2+1}{y(y-1)},$$ and the minimal polynomial of $y$ is $X^3 +a X^2 -(a+3)X+1$, with $a \in F$. 
\end{proof}
In next result, we prove that when $F$ does not contain a primitive $3^{rd}$ root of unity, any Galois cubic extension has a generator $y$ with minimal polynomial $$X^3 - 3X - \frac{2b^2+2b-1}{b^2+ b+1}$$ announced at the beginning of \S 1. In doing so, we give an explicit one-to-one correspondence between $y$ and the generator $z$ with minimal polynomial of the form $X^3 +a X^2 -(a+3)X+1$. 
\begin{theoreme} \label{shanksconversion}
Let $p \neq 3$. Suppose that $F$ does not contain a primitive $3^{rd}$ root of unity.  Let $L/F$ be a Galois cubic extension. 
Then there is a generator $w$ of $L/F$ whose minimal polynomial equal to $$X^3 -3 X -  \frac{2b^2+2b-1}{b^2+  b+1},$$ where $b\in F$. More precisely,  $y$ is a generator with minimal polynomial $X^3 +a X^2 -(a+3)X+1$, where $a \in F$ if, and only if,  $w= \frac{3 + a y}{3-(a+3)y}$ is a generator with minimal polynomial equal to $X^3 -3 X -  \frac{2a^2+6a-9}{a^2+ 3 a+9},$ where $a= 3b \in F$. Furthermore, $\sigma (y)= -1/(y-1)$ and $ \sigma (w)= - \frac{ (b^2+b+1) w +b(b-2)}{ (b^2-1) w + (b^2 +b+1)},$ where $\sigma$ is a fixed choice of generator of $\text{\emph{Gal}}(L/F)$.
\end{theoreme} 
\begin{proof} 
By Lemma \ref{shanks}, we have that there is a generator $y$ with minimal polynomial $X^3 +a X^2 -(a+3)X+1$, where $a \in F$. Via the change of basis $$w= \frac{3 + a y}{3-(a+3)y}:= M(y)$$ (see Theorem \ref{classification}), where $M$ is the matrix defined as $$M=\left(  \begin{array}{cc} 3& a \\ 3 & -(a+3)\end{array} \right),$$ we find via Corollary \ref{classification} that $w$ satisfies the cubic equation 
$$w^3 -3 w = \frac{2a^2+6a-9}{a^2+ 3 a+9},$$
Letting $b=\frac{a}{3}$, we then obtain
$$w^3 -3 w = \frac{2b^2+2b-1}{b^2+ b+1}.$$
Moreover, 
\begin{align*} \sigma (w) &= M \cdot \sigma (y)= M \left( \begin{array}{cc} 0 & -1 \\ 1 & -1 \end{array} \right) M^{-1}w \\
&=  \left( \begin{array}{cc} -(a^2+3a+9) & -a(a+6) \\ a^2-9 & a^2+3a+9 \end{array} \right) w\\ 
&= \frac{ -(a^2+3a+9) w -a(a-6)}{ (a^2-9) w + a^2 +3a+9}= \frac{ -(b^2+b+1) w -b(b-2)}{ (b^2-1) w + (b^2 +b+1)}.
\end{align*}
Finally, the reverse transformation 
$$y = \frac{3(w-1)}{a+w(a+3)},$$ 
converts this generator $w$ into a Shank generator.
\end{proof} 
\begin{remarque} 
We note that the quantity $a^2 + a+1= (a + \xi) (a+ \xi^2)$ in Theorem \ref{shanksconversion} is a norm of $F(\xi)$ over $F$.
\end{remarque}
Suppose that $F = K$ is a global field, i.e., a function field over a finite field, or a number field. If the base field $F=K$ is a function field with field of constants $\mathbb{F}_q$, let $\mathcal{O}_{K,x}$ denote the ring of integers of $K$ over $\mathbb{F}_q[x]$, where $x \in K \backslash \mathbb{F}_q$. Then one can write the element $b$ in the Theorem \ref{shanksconversion} as $ b= \frac{A}{B}$ with $A, B \in \mathcal{O}_{K,x}$. If the base field $K$ is a number field, then the same can be done in $O_K$, the ring of integers over $\mathbb{Z}$ in $K$. We obtain the following corollary in terms of this decomposition. 
\begin{corollaire} \label{galoisglobal}
Let $K$ be a global field which does not contain a primitive $3^{rd}$ root of unity, and let $L/K$ be a Galois cubic extension. 
Then there is a generator $w$ of $L/K$ with minimal polynomial equal to $$X^3 -3 X -  \frac{2A^2+2AB-B^2}{A^2+  AB+B^2},$$ where $A, B$ lie in $\mathcal{O}_{K,x}$ if $K/\mathbb{F}_q$ is a function field and $\mathcal{O}_K$ if $K$ is a number field.
\end{corollaire} 
When $K=\mathbb{Q}$, or $K=\mathbb{F}_q(x)$ with $q \equiv -1 \mod 3$, this yields an additional corollary on the irreducible factors appearing in the denominator of the parameter $b$ in the form $X^3 -3 X - b$.
\begin{proposition} \label{denom}
Let $K=\mathbb{Q}$, or $K=\mathbb{F}_{q}(x)$ with $q \equiv -1 \mod 3$, and let $L/K$ a Galois cubic extension. By Theorem  \ref{shanksconversion}, $L/K$ has a primitive element with minimal polynomial of the form $$T(X) = X^3 -3 X - b$$ with $b = \frac{P}{Q}$, where $P,Q \in \mathbb{Z}\ or \ \mathbb{F}_q[x]$ such that $(P,Q) = 1$. Then we have
$$Q=w \prod_i Q_i ^{e_i} \qquad (e_i \in \mathbb{N})$$
\begin{enumerate}
\item  if $K=\mathbb{Q}$, then each $Q_i$ is a prime number such that $Q_i\equiv 1 \mod 3 $ and $w = \pm 1$; and
\item  if $K=\mathbb{F}_{q}(x)$ with $q \equiv -1 \mod 3$, then each $Q_i \in \mathbb{F}_q[x]$ is unitary and irreducible of even degree and $w \in \mathbb{F}_q^*$.
\end{enumerate}
\end{proposition} 
\begin{proof} 
Let $\xi$ be a primitive $3^{rd}$ root of unity. 
\begin{enumerate} 
\item By \cite[Corollary 10.4]{Neu}, a prime number $p$ splits in $\mathbb{Q}(\xi)$ if, and only if, $p \equiv 1 \mod 3 $. Moreover, by \cite[Proposition 10.2]{Neu} , the ring of integers of $\mathbb{Q}(\xi)$ is equal to $\mathbb{Z}[\xi]$, and by \cite[Theorem 11.1]{Wash}, the ring $\mathbb{Z}[\xi]$ is a unique factorisation domain. Thus, $Q_i$ is a prime number with $Q_i \equiv 1 \mod 3 $, if and only if $$Q_i =(A_i + \xi B_i)(A_i + \xi^2 B_i)=A_i^2+A_iB_i + B_i^2$$ where $A_i,B_i \in \mathbb{Z}$ are coprime.
\item  By \cite[Theorem 3.46]{LiNi}, an irreducible polynomial over $\mathbb{F}_q$ of degree $n$ factors over $\mathbb{F}_{q^k}[x]$ into $\gcd(k,n)$ irreducible polynomials of degree $n/\gcd(k,n)$. It follows that an irreducible polynomial over $\mathbb{F}_{q}$ factors in $ \mathbb{F}_{q^2}$ into polynomials of smaller degree if, and only if, $\gcd(2,n) > 1$, i.e., $2 | n$. As a consequence, the irreducible polynomials in $\mathbb{F}_{q^2}[x]$ are those irreducible polynomials of odd degree in $\mathbb{F}_q[x]$ or those occurring as factors of irreducible polynomials of even degree in $\mathbb{F}_q[x]$. Thus, only even degree irreducible polynomials $Q_i \in \mathbb{F}_q[x]$ can be written as a norm 
$$Q_i =(A_i + \xi B_i)(A_i + \xi^2 B_i)=A_i^2+A_iB_i + B_i^2$$ with $A_i, B_i \in \mathbb{F}_q[x]$ coprime. 
\end{enumerate} 
Since in both cases, the norm map sending $\alpha + \xi \beta$ to $\alpha^2 + \alpha \beta + \beta^2$ is multiplicative and $K(\xi)$ is a unique factorisation domain, we obtain the result. 
\end{proof}
\begin{remarque}
Conversely, when $Q = w \prod_i Q_i $ is of the form specified in Proposition \ref{denom}, then one may find a Galois extension with primitive element possessing minimal polynomial of the form $$T(X) = X^3 -3 X - b$$ with $b = \frac{P}{Q}$, where $P \in \mathbb{Z}$ or $\mathbb{F}_q[x]$. As $K(\xi)$ is a unique factorisation domain, it follows that any $Q_i$ can be uniquely written as a product $Q_i = (A_i + \xi B_i)(A_i + \xi^2 B_i)$. Therefore, writing each irreducible factor $Q_i\;|\; Q$ in this way, and using the norm map, we may then write $Q= A^2 +AB+B^2$ for some $A,  B \in \mathbb{Z}$ or $\mathbb{F}_q[x]$. These coprime $A$ and $B$ are precisely all of the possible such elements of $\mathbb{Z}$ or $\mathbb{F}_q[x]$ which result in a Galois cubic extension with generation $y^3 -3y - \frac{P}{Q}$, and for such a Galois extension, $P = 2A^2+2AB-B^2$. This remark together with Proposition \ref{denom} thus describes how to obtain all Galois cubic extensions of $\mathbb{Q}$, and of $\mathbb{F}_q(x)$ when $q \equiv -1 \mod \ 3$.
\end{remarque} 

\section{Decomposition of cubic polynomials over a finite field} 
Let $\mathbb{F}_{s}$ denote a finite field with $ s=p^m$ elements, where $p$ is a prime integer. In this section, we use the forms of Corollary \ref{classification} to obtain a simple characterisation of the irreducibility and splitting of cubic polynomials over any finite field. This will be needed for our detailed study of ramification and splitting for cubic function fields. 

As it will be useful in a few results which follow, particularly in characteristics $2$ and $3$, we state the usual definition of the trace map from the finite field $\mathbb{F}_s= \mathbb{F}_{p^m}$ to the prime field $\mathbb{F}_p$:
$$ \text{Tr} : \mathbb{F}_{p^m} \rightarrow \mathbb{F}_{p},\qquad \text{Tr} (\alpha)= \alpha + \alpha^p + \cdots + \alpha^{p^{m-1}}.$$
If $f(X)= X^3 + eX^2 + f X+ g$ is any cubic polynomial where $e,f, g \in \mathbb{F}_s$, then by Theorem \ref{irred}, $f(X)$ is reducible in the following cases:
\begin{itemize} 
\item[$\bullet$] $g=0$
\item[$\bullet$] $-27g^2-2f^3+9egf=0$ $\qquad \qquad \qquad \qquad \qquad \qquad (*)$
\item[$\bullet$] $p=3$, $-f^2e^2+ge^3+f^3=0$
\end{itemize}
Furthermore, in the proof of Theorem \ref{irred}, we give an explicit root of $f(X)$ in each of these cases. It follows that when any of these occur, we may write $f(X)= (X-r)Q(X)$, where $r \in \mathbb{F}_s$ is this explicit root and $Q(X) \in \mathbb{F}_s[X]$ is quadratic, so that it remains only to study the splitting of a quadratic polynomial over $\mathbb{F}_s$. We have: \begin{itemize} \item[$\bullet$] If $p \neq 2$, then as is true for a quadratic polynomial over a field of characteristic zero, $Q(X)$ is reducible over $\mathbb{F}_s$ if, and only if, its discriminant is a square in $\mathbb{F}_s$, and $Q(X)$ has distinct roots when this discriminant is nonzero. \item[$\bullet$] If $p=2$, the discriminant is not the requisite tool; instead, the trace map $\text{Tr}(\cdot)$ defined above may be used \cite[Proposition 1]{Pomm}. \end{itemize}  On the other hand, when the cases listed in $(*)$ above do not occur, then also as in Theorem \ref{irred}, there is an explicit change of variable which reduces the study of the decomposition of $f(X)$ over $\mathbb{F}_s$ to one of the following three forms of polynomial:
\begin{enumerate} 
\item $X^3 -a$, $a\in \mathbb{F}_s$;
\item $X^3 -3X-a$, $a\in \mathbb{F}_s$, when $p\neq 3$; 
\item $X^3 +aX+a^2$, $a\in \mathbb{F}_s$, when $p=3$. 
\end{enumerate}
In this section, we thus proceed to study the decomposition of these forms over $\mathbb{F}_s$.
\subsection{$X^3 -a$, $a\in \mathbb{F}_s$}
In characteristic $3$, the study is immediate, as $X^3-a$ is irreducible if, and only if, $a$ is not a cube, and as soon as $a$ is a cube, say $a= b^3$, $b\in \mathbb{F}_{s}$, then $X^3 -a=(X-b)^3$. Thus, we are left to study the case $p\neq 3$. We let $\xi$ denote a primitive $3^{rd}$ root of unity.
\begin{lemma}\label{splittingpolypurelycubic}
Let $p \neq 3$ and $f(X) = X^3 -a$, where $a\in\mathbb{F}_{s}$. Then all possible decompositions of $f(X)$ over $\mathbb{F}_{s}$ are as follows:
\begin{enumerate} 
\item $f(X) $ is irreducible if, and only if, $s \equiv 1 \mod 3$ and $a^{(s-1)/3} = \xi^i$, for some $i=1,2$.
\item $f(X) = (X-\alpha ) Q(X)$ where $\alpha \in \mathbb{F}_{s}$ and $Q(X)$ is an irreducible quadratic polynomial if, and only if, $s \equiv - 1 \mod 3$. 
\item The case $f(X) = (X-\alpha ) (X-\beta)^2 $ with $\alpha , \beta \in \mathbb{F}_{s}$ and $\alpha \neq \beta$ cannot occur.
\item $f(X) $ has a unique root with multiplicity $3$ if, and only if, $a = 0$. 
\item $f(X)  = (X-\alpha ) (X-\beta) (X- \gamma)$ with $\alpha , \beta, \gamma \in \mathbb{F}_{s}$ distinct if, and only if, $s \equiv 1 \mod 3$ and $a^{(s-1)/3} = 1$.
\end{enumerate}  
\end{lemma} 
\begin{proof} 

\begin{enumerate}
\item 
A cubic polynomial over any field is irreducible if, and only if, it has no root in that field. Thus, $f(X) $ is irreducible if, and only if, $a$ is not a cube in $\mathbb{F}_{s}$. If $s \equiv -1 \mod 3 $, then $a$ is always a cube. Indeed, any element $b \in \mathbb{F}_{s}$ is a cube: $\mathbb{F}_{s}^*$ is a group under multiplication of order $s-1$, and by Lagrange's theorem, $b^{s -1}= 1$ in $\mathbb{F}_{s}$. Thus, $b^{2s-1} = b$. If $s \equiv -1 \mod 3$, then there is $l\in \mathbb{Z}$ such that $s = 3l-1$ and $b^{2s -1} = b^{6l -3}= (b^{2l-1})^3=b$, whence $b$ is a cube. 

As a consequence, if $a$ is not a cube, then $s \equiv 1 \mod 3$. Furthermore, if $s \equiv 1 \mod 3$, so that 
$3 \;|\; s-1$, we let $y$ in $\overline{\mathbb{F}_{s}} $ be a root of $X^3-a$. We may write $$a^{\frac{s -1}{3}} =y^{s -1}$$ in $\overline{\mathbb{F}_{s}} $. Also, $\text{\emph{Gal}}( \mathbb{F}_{s^{3}} / \mathbb{F}_{s}) = \langle \phi \rangle $, where $\phi$ is the Frobenius automorphism of $\mathbb{F}_{s^{3}} $ over $\mathbb{F}_s$ sending $\alpha \rightarrow \alpha^{s}$. The polynomial $f(X) $ is irreducible over $\mathbb{F}_s$ if, and only if, $\mathbb{F}_{s} (y)= \mathbb{F}_{s^{3}}$. This is equivalent to $y^{s-1}= \xi^i$, with $i=1,2$. Indeed, when $y^{s} = y$, then $\text{\emph{Gal}}( \mathbb{F}_{s} (y) / \mathbb{F}_{s} ) = \{ Id\}$, whence $\mathbb{F}_{s} (y) = \mathbb{F}_{s}$, so that $X^3  -a$ would not be irreducible and when $y^{s} \neq y$, then as $\xi y$ and $\xi^2 y$ are the other two roots of $X^3-a$ and $\phi \in \text{\emph{Gal}}( \mathbb{F}_{s} (y) / \mathbb{F}_{s} ) $, then $\phi $ sends a root of $X^3 -a$ to another such root, so that $y^s= \xi^i y$ for some $i=1,2$. In this case, $|\text{\emph{Gal}}( \mathbb{F}_{s} (y) / \mathbb{F}_{s} )|=3$ and $X^3 -a$ is irreducible. 
\item $f(X) = (X-\alpha ) Q(X)$ where $\alpha \in \mathbb{F}_{s}$ and $Q(X)$ is an irreducible quadratic polynomial, if and only if $f(X)$ has a root $\alpha$ in $\mathbb{F}_{s}$ and the two other roots $\xi \alpha$ and $\xi^2 \alpha$ are not in $\mathbb{F}_{s}$. Equivalently, $a$ is a cube and $\xi \notin \mathbb{F}_{s}$ .Thus, using what we proved in $(1)$, we obtain the result.
\item If $f(X)$ has a root $\alpha$, then since the two other roots are $\xi \alpha $ and $\xi^2 \alpha $, it follows immediately that these three roots are all distinct unless $\alpha=0$, in which case they are all equal. Thus, the case $f(X) = (X-\alpha ) (X-\beta)^2 $ with $\alpha , \beta \in \mathbb{F}_{s}$ and $\alpha \neq \beta$ cannot occur.
\item If $f(X) = (X - \alpha)^3$, then the equality $X^3-a =  (X - \alpha)^3= X^3-3\alpha X^2 +3 \alpha^2 X - \alpha^3$ implies that $-\alpha^3 = 0$. It again follows that $\alpha = 0$, and hence that $X^3-a = X^3$, which holds if, and only if, $a = 0$. 
\item This is the complement of all other cases.
\end{enumerate}
\end{proof}
\subsection{$X^3 -3X-a$, $a\in \mathbb{F}_s$, $p\neq 3$} This is the only case remaining for the study of the splitting of a polynomial over a finite field when $p\neq 3$. 
\begin{lemma} \label{splittingpolyqneq1mod3}
Let $p \neq 3$ and $f(X) = X^3 - 3 X -a$, where $a \in \mathbb{F}_{s}$. We denote by $\Delta: = -27(a^2-4)$ the discriminant of $f(X)$ and $\delta \in \overline{\mathbb{F}_s}$ such that $\delta^2 =a^2-4$. Then all possible decompositions of $f(X)$ over $\mathbb{F}_{s}$ are as follows:
\begin{enumerate} 
\item $f(X)$ is irreducible if, and only if, 
\begin{enumerate}
\item $p\neq 2$, \begin{enumerate} \item $s \equiv 1 \mod 3$, $\Delta $ is a non-zero square  and $\frac{1}{2} (a + \delta)$ is not a cube in $\mathbb{F}_{s}$, or \item $s=5$ and $a =\pm 1$; or \end{enumerate}
\item $p=2$ $(s=2^m)$, $a \neq 0 $, $\text{\emph{Tr}}(1/a^2) = \text{\emph{Tr}}(1)$, and the roots of $T^2 +aT +1$ are non-cubes in $\mathbb{F}_{s}$, when $m$ is even (resp. in $\mathbb{F}_{s^{2}}$, when $m$ is odd).
\end{enumerate} 
\item $f(X)= (X-\alpha ) Q(X)$ where $\alpha \in \mathbb{F}_{s}$ and $Q(X)$ is an irreducible quadratic polynomial if, and only if, 
\begin{enumerate}
\item $p\neq 2$ and $\Delta $ is a non-square in $\mathbb{F}_{s}$; or
\item $p=2$, $a \neq 0$, and $\text{\emph{Tr}}(1/a^2) \neq \text{\emph{Tr}}(1)$.
\end{enumerate} 
\item $f(X) = (X-\alpha ) (X-\beta)^2 $ with $\alpha , \beta \in \mathbb{F}_{s}$ and $\alpha \neq \beta$ if, and only if, $a = \pm 2$. 
\item $f(X)$ may never have a unique root with multiplicity $3$.
\item $f(X) = (X-\alpha ) (X-\beta) (X- \gamma)$ with $\alpha , \beta, \gamma \in \mathbb{F}_{s}$ distinct if, and only if, $a\neq \pm 2$ and 
\begin{enumerate}
\item $p\neq 2$, $\Delta$ is a square, and either \begin{enumerate}\item[$\bullet$] $s \equiv 1 \mod 3$, and $\frac{1}{2} (a + \delta)$ is a cube in $\mathbb{F}_{s}$, or \item[$\bullet$] $s\neq 5$, or \item[$\bullet$] $s=5$ and $a\neq \pm 1 $; \end{enumerate} or
\item $p=2$ $(s=2^m)$, $\text{\emph{Tr}}(1/a^2) = \text{\emph{Tr}}(1)$ and the roots of $T^2 +aT +1$ are cubes in $\mathbb{F}_{s}$, when $m$ is even (resp. in $\mathbb{F}_{s^{2}}$, when $m$ is odd). 
\end{enumerate} 
\end{enumerate}  
\end{lemma} 

\begin{proof} 
\begin{enumerate}
\item 
\begin{enumerate} 
\item Suppose $p\neq 2$. We denote by $t$ a root of the polynomial $X^2 + 3$ in $\overline{\mathbb{F}_s}$.
\begin{enumerate} 
\item Suppose also that $s \equiv 1 \mod 3$. Thus, $-3$ is a square in $\mathbb{F}_{s}$ and $\mathbb{F}_{s}(t)=\mathbb{F}_{s}$. 
As $-3$ is a square in $\mathbb{F}_s$, it follows that $a^2-4$ is a square in $\mathbb{F}_s$ if, and only if, $\Delta$ is. When $\Delta$ is a square in $\mathbb{F}_{s}$, there is $\delta \in \mathbb{F}_{s}$ such that $a^2 -4=\delta^2$. By \cite[Theorem 3]{Dickson}, the polynomial $f(X)$ is irreducible over $\mathbb{F}_{s}$ if, and only if, its discriminant $\Delta$ is a nonzero square and the element $\frac{1}{2} (a + \delta) $ is a non-cube in $\mathbb{F}_{s}$.

\item  If $s \equiv -1 \mod 3$, then again by \cite[Theorem 3]{Dickson}, the irreducible cubic polynomials in $\mathbb{F}_{s}[X]$ are given by \begin{equation*}X^3 - 3X \nu^{-\frac{1}{3}(s-2)(s +1)} - \nu - \nu^{s} = 0,\qquad\qquad (*)\end{equation*} where $\nu \in \mathbb{F}_{s}(t)$ is a non-cube in $\mathbb{F}_{s}(t)$. Note that $\mathbb{F}_{s}(t)= \mathbb{F}_{s^{2}}$ since $s \equiv -1 \mod 3$. Thus, if $X^3 - 3X - a$ is an irreducible polynomial in $\mathbb{F}_{s}[X]$,  then $$\nu^{-\frac{1}{3}(s-2)(s +1)}  = 1$$ for a non-cube $\nu\in \mathbb{F}_{s}(t)$. We also have $$\frac{1}{3}(s-2)(s +1)  < s^2-1=(s-1)(s+1).$$ As $\nu\in \mathbb{F}_{s}(t)$, it must also be true that $\nu$ is a $(s^2-1)^{st}$ root of unity. Thus $$\frac{1}{3}(s-2)(s +1)  \;\Big|\;s^2-1 = (s-1)(s +1) ,$$ so that $\frac{1}{3}(s-2)$ divides $(s-1)$, i.e., $$(s-2) \;\Big|\; 3(s-1).$$ As $(s-2,s-1)=1$, we obtain $(s-2 )\;|\; 3$. As $s \geq p > 3$, it follows that $s = p = 5$. Hence $$1 =\nu^{-\frac{1}{3}(s-2)(s +1)}  = \nu^{-6},$$ so that by $(*)$, $a = \nu^5 + \nu$ for the non-cube $6^{th}$ root of unity $\nu$. Let $\zeta \in  \mathbb{F}_{25}$ be a primitive $6^{th}$ root of unity. As $\nu$ is a non-cube in $\mathbb{F}_{25}$, it follows that $\nu = \zeta^i$ for some $i=1,2,4,5$, and thus that $\nu$ is either a primitive $3^{rd}$ or $6^{th}$ root of unity. 
 
 In the case that $\nu$ is a primitive $3^{rd}$ root of unity, we have $$a = \nu + \nu^5= \nu + \nu^2= -1,$$ whereas when $\nu$ is a primitive $6^{th}$ root of unity, as $\nu^2 = \nu-1$, we have $$\begin{array}{ccl} \nu + \nu^5&=& \nu + \nu (\nu -1)^2= \nu^3-2\nu^2 +2\nu\\ 
 &=& \nu (\nu -1) -2(\nu -1) +2\nu = \nu^2-\nu +2=  1.\end{array}$$ It follows that in this case, $X^3 - 3X - a$ is irreducible over $\mathbb{F}_{s}$ if, and only if, $s=5$ and $a= \pm 1$.
 \end{enumerate}
 \item When $p = 2$, then by \cite[Theorem 1]{Williams}, the polynomial $$X^3 - 3X - a = X^3 + X + a \in \mathbb{F}_{s}[X]$$ is irreducible over $\mathbb{F}_s$ if, and only if, $\text{Tr}(1/a^2)= \text{Tr}(1)$ and the roots of $$ T^2  + aT + 1 \in \mathbb{F}_{s}[T]$$ are non-cubes in $\mathbb{F}_{s}$, when $m$ is even (resp. in $\mathbb{F}_{s^{2}}$, when $m$ is odd). 
\end{enumerate}
\item  $f(X)= (X-\alpha ) Q(X)$, where $\alpha \in \mathbb{F}_{s}$ and $Q(X)$ is an irreducible quadratic polynomial if and only if $f(X)$ has a root $\alpha$ in $\mathbb{F}_{s}$ and the two other roots $\alpha_1$ and $\alpha_2$ are not in $\mathbb{F}_{s}$.
As noted in Corollary \ref{rootour} and Remark \ref{rootourre}, the other two roots of $f(X)$ in $\overline{\mathbb{F}_{s}}$ are given by
$$\alpha_1 = -\frac{1+2f}{a} z^2 +f z + \frac{2(1+2f)}{a}$$ 
and 
$$\alpha_2= \frac{1+2f}{a} z^2 +(-1-f) z - \frac{2(1+2f)}{a},$$ 
where $$f= \frac{-6+ a r}{3(a^2-4)}$$ and $r$ is a root of the quadratic resolvent $$R(X)= X^2 +3 a X + (-27+9a^2)$$ of $f(X)$. As a consequence, $f(X)= (X-\alpha ) Q(X)$ where $\alpha \in \mathbb{F}_{s}$ and $Q(X)$ is irreducible quadratic if, and only if, $f(X)$ is not irreducible and $R(X)$ is irreducible in $\mathbb{F}_{s}$. The discriminant of $R(X)$ is equal to $\Delta = -27 (a^2-4)$, whence irreducibility of $R(X)$ over $\mathbb{F}_{s}$ is equivalent to $\Delta$ a non-square in $\mathbb{F}_{s}$ when $p \neq 2$. When $p=2$, we can rewrite $R(X)$ as $R(X)= X^2 + a X + (1+a^2)$. Taking $Y= X/a$, we obtain $$\frac{R(X)}{a^2} = Y^2 + Y + 1 +\frac{1}{a^2}.$$ This polynomial is irreducible if, and only if, $\text{Tr}( 1+\frac{1}{a^2}) \neq 0$, that is, $\text{Tr}(\frac{1}{a^2})\neq \text{Tr}(1)$ (see, \cite[Proposition 1]{Pomm}). Hence the result.

\item The equality $f(X) = (X - \alpha)  (X - \beta)^2$ gives us $$X^3 - 3X - a = (X - \alpha)(X - \beta)^2 = X^3 - (2\beta + \alpha) X^2 + (\beta^2 + 2\alpha \beta) X - \alpha \beta^2.$$ Thus $\alpha = - 2\beta$.
We therefore have $- 3 = \beta^2 - 4\beta^2 = - 3 \beta^2$ and $a = - 2\beta^3$. The first of these implies that $$3(\beta^2 - 1) = 3 \beta^2 - 3 = 0.$$ Thus $\beta = \pm 1$ and $a = \mp 2$. Conversely, when $a = \mp 2$, then 
$$X^3-3 X \mp 2= ( X \pm 2)(X\mp 1 )^2.$$ Hence the result. 
\item	 If $f(X) = (X - \alpha)^3$, then the equality $X^3-3X-a =  (X - \alpha)^3 = X^3 - 3\alpha X^2+ 3 \alpha^2 X - \alpha^3$ implies that $-3\alpha = 0$ and $3 \alpha^2 = -3$. As $p \neq 3$, this is impossible. 
\item This is the complement of all other cases.
\end{enumerate}
\end{proof}

\subsection{$X^3 +aX+a^2$, $a\in \mathbb{F}_s$, $p=3$} This is the final case we must consider.
\begin{lemma}\label{splittingpolychar3partdeux}
Let $p = 3$ and $f(X) = X^3 + aX + a^2$, where $a \in \mathbb{F}_{s}$. Then all possible decompositions of $f(X)$ over $\mathbb{F}_{s}$ are as follows:
\begin{enumerate}
\item $f(X)$ is irreducible if, and only if, $-a$ is a non-zero square in $\mathbb{F}_{s}$, say $-a = b^2$, and $\text{Tr}(b) \neq 0$.
\item $f(X) = (X-\alpha)Q(X)$ where $\alpha \in \mathbb{F}_{s}$ and $Q(X)$ is an irreducible quadratic polynomial if, and only if, $a$ a non-square in $ \in \mathbb{F}_{s}$. 
\item The case $f(X) = (X-\alpha)(X-\beta)^2$ with $\alpha, \beta \in \mathbb{F}_{s}$ and $\alpha \neq \beta$ may never occur.
\item $f(X)$ has a unique root with multiplicity 3 if, and only if, $a=0$.
\item $f(X) = (X-\alpha)(X - \beta)(X - \gamma)$ with $\alpha , \beta, \gamma \in \mathbb{F}_{s}$ distinct  if, and only if, $-a$ is a non-zero square in $\mathbb{F}_{s}$ say $-a = b^2$, and $\text{\emph{Tr}}(b ) =0$.
\end{enumerate}
\end{lemma}

\begin{proof}
This is a direct consequence of \cite[Theorem 2]{Williams}, that the case $$f(X) = (X-\alpha)(X-\beta)^2$$ with $\alpha, \beta \in \mathbb{F}_{s}$ and $\alpha \neq \beta$ never occurs, and that $f(X)$ has a unique root with multiplicity $3$ if, and only if, $a=0$. 
To see this, suppose that $f(X) = (X - \alpha)(X - \beta)^2$ where $\alpha \neq \beta$. Then $$(X - \alpha)(X - \beta)^2 =  X^3 - (2\beta + \alpha) X^2 + (\beta^2 + 2\alpha \beta) X - \alpha \beta^2.$$ It follows that $\alpha = - 2 \beta$, $\beta^2 + 2\alpha \beta = a$, and $\alpha \beta^2 = -a^2$. Thus, in this case, $a=-3 \beta^2 =0$, from which it follows that $f(X)$ has a root of multiplicity $3$. Finally, suppose that $$f(X) = (X - \alpha)^3 = X^3 - \alpha^3.$$ In particular, it follows that $a = 0$, which in turn implies that $a^2 = 0$, so that $f(X) = X^3$, whence $\alpha = 0$. 
\end{proof}
In \S 3, we will use these criteria to specialise to the case where $F=K$ is a function field over $\mathbb{F}_q$. In this context, we will be able to determine ramification and splitting data of our cubic forms by reducing to residue fields, which are finite, so that all of the results of \S 2 will apply.
\begin{remarque} If $F$ is an arbitrary field of characteristic different from $2$ or $3$, then the decomposition of a cubic polynomial over $F$ depends upon the irreducibility of one of the cubic forms
\begin{enumerate} 
\item $X^3 -a$, $a\in F$
\item $X^3 -3X-a$, $a\in F$
\end{enumerate}
In case $(1)$, the decomposition of $X^3 -a$ depends upon where $a$ is a cube in $F$. In case (2), by Theorem \ref{purelycubic}, given a root $c \in \overline{F}$ of $X^2 + a X + 1$, the variable $Y=\frac{cX-1}{X-c}$ satisfies $Y^3=c$. Then, by Theorem \ref{purelycubic}, given a cube root $\gamma \in \overline{F}$ of $c$, we have for each $i=0,1,2$ that
$$\alpha_i = \frac{c \xi^i \gamma -1}{\xi^i \gamma-c}= \frac{ \gamma^4 - \xi^{2i}}{\xi^{2i} \gamma (\xi^i -\gamma^2)}=\frac{ \gamma^2 + \xi^{i}}{\xi^{2i} \gamma } $$ are the roots in $\overline{F}$ of the cubic form $X^3-3X-a$. The decomposition of $X^3-3X-a$ is equivalent to determining which $\alpha_i$ lie in $F$ and are distinct. 
\end{remarque}

\section{Ramification, splitting and Riemann-Hurwitz}

In this section, we let $F=K$ denote a function field with field of constants $\mathbb{F}_q$, where $q = p^n$ and $p > 0$ is a prime integer. We denote by $\mathfrak{p}$ a place of $K$. The \emph{degree} $d_K(\mathfrak{p})$ of $\mathfrak{p}$ is defined as the degree of its residue field, which we denote by $k(\mathfrak{p})$, over the constant field $\mathbb{F}_q$. The cardinality of the residue field $k(\mathfrak{p})$ may be written as $|k (\mathfrak{p})|  = q^{d_K(\mathfrak{p})}$. In an extension $L/K$ and a place $\mathfrak{P}$ of $L$ over $\mathfrak{p}$, we let $f(\mathfrak{P}|\mathfrak{p}) = [k(\mathfrak{P}):k (\mathfrak{p})]$ denote the inertia degree of $\mathfrak{P}|\mathfrak{p}$. We let $e(\mathfrak{P}|\mathfrak{p})$ be the ramification index of $\mathfrak{P}|\mathfrak{p}$, i.e., the unique positive integer such that $v_\mathfrak{P}(z) = e(\mathfrak{P}|\mathfrak{p})v_\mathfrak{p}(z)$, for all $z \in K$. If $v_\mathfrak{p}(a) \geq 0$, then we let $$\overline{a} := a \mod \mathfrak{p}$$ denote the image of $a$ in $k(\mathfrak{p})$.

Here, we study the ramification and splitting of places in cubic function fields over $K$, and we give an explicit Riemann-Hurwitz formula for separable cubic function fields over $K$. We obtain these results for any prime characteristic $p$, using the classification given in Corollary \ref{classification}, which establishes that any separable cubic extensions $L/K$ has a generator with minimal polynomial equal to one of the following:
\begin{enumerate}
\item $X^3-a$, $a\in K$, when $p\neq 3$
\item $X^3-3X-a$, $a\in K$, when $p\neq 3$
\item $X^3+aX+a^2$, $a\in K$, when $p= 3$.
\end{enumerate}
We would like to point out that by Corollary \ref{classification}, if $p \neq 3$, then every impure cubic extension $L/K$ has a primitive element $y$ with minimal polynomial of the form $$f(X) = X^3 - 3X - a.$$ We mention this here, as we will use it whenever this case occurs in \S 3. When this is used, the element $y$ will denote any such choice of primitive element.
\subsection{Constant extensions} 
In this subsection, we wish to determine when a cubic function field over $K$ is a constant extension of $K$. We do this before a study of ramification and splitting, as constant extensions are unramified and their splitting behaviour is well understood \cite[Chapter 6]{Vil}. In the subsequent subsections, we will thus assume that our cubic extension $L/K$ is not constant, which, as 3 is prime, is equivalent to assuming that the extension is geometric.

\subsubsection{$X^3-a$, $a\in K$, when $p\neq 3$} 
\begin{lemme}\label{constantextension}
Let $p\neq 3$, and let $L/K$ be purely cubic, i.e. there exists a primitive element $y\in L$ such that $y^3 = a$, $a \in K$. Then $L/K$ is constant if, and only if, $a=ub^3$, where $b\in K$ and $u \in \mathbb{F}_q^*$ is a non-cube. In other words, there is a purely cubic generator $z$ of $L/K$ such that $z^3 = u$, where $u \in \mathbb{F}_q^*$.
\end{lemme} 
\begin{proof}
Suppose that $a=ub^3$, where $b\in K$ and $u$ is a non-cube in $\mathbb{F}_q^*$. Then $z=\frac{y}{b}\in L$ is a generator of $L/K$ such that $z^3 = u$. The polynomial $X^3 -u$ has coefficients in $\mathbb{F}_q$, and as a consequence, $L/K$ is constant. 

Suppose then that $L/K$ is constant. We denote by $l$ the algebraic closure of $\mathbb{F}_q$ in $L$, so that $L=Kl$. Let $l= \mathbb{F}_q (\lambda )$, where $\lambda$ satisfies a cubic polynomial $X^3 + e X^2 + f X + g $ with $e, f , g \in \mathbb{F}_q$. Hence, $L= K(\lambda )$. We denote $$\alpha = -2-\frac{ ( 27 g^2 -9efg +2f^3)^2}{27(3ge-f^2)^3} \in \mathbb{F}_q.$$
As $L/F$ is purely cubic, it then follows by Theorem \ref{purelycubic} that either $3eg = f^2$ or the quadratic polynomial $X^2+ \alpha X +1$ has a root in $K$. In both cases, there is a generator $\lambda ' \in L$ such that $$\lambda'^3 = \beta \in \mathbb{F}_q.$$
Hence $\lambda' \in l$. The elements $\lambda'$ and $y$ are two purely cubic generators of $L/K$, whence by Lemma  \ref{extensionsthesame}, it follows that $y = c \lambda'^j$ where $j=1$, or $2 $ and $c\in K$.  Thus, $a = c^3 \beta^j$, where $\beta\in \mathbb{F}_q$. The result follows.
\end{proof}
\subsubsection{$X^3-3X-a$, $a\in K$, when $p\neq 3$}
Via Corollary \ref{classification} and Lemma \ref{extensionsthesame1} , a proof similar to that of Lemma \ref{constantextension} yields the following result. 
\begin{lemme}\label{constantextension2}
Let $p\neq 3$ and  $L/K$ be an impurely cubic extension, so that there is a primitive element $y\in L$ such that $y^3-3y = a$ (see Corollary \ref{classification}). Then $L/K$ is constant if, and only if, $$u= -3a\alpha^2\beta+a\beta^3+6\alpha+\alpha^3a^2-8\alpha^3 \in \mathbb{F}_q^*,$$ for some $\alpha, \beta \in K$ such that $\alpha^2 + a_2 \alpha \beta + \beta^2 =1$. In other words, there is a generator $z$ of $L/K$ such that $z^3 -3z=u$, where $u \in \mathbb{F}_q^*$.
\end{lemme} 
\subsubsection{$X^3+aX+a^2$, $a\in K$, when $p= 3$} In this case, one may prove the following result, similarly to the proof of Lemma \ref{constantextension}, via Corollary \ref{classification} and Lemma \ref{char3extensionsthesame}.
\begin{lemme}\label{constantextension2}
Let $p= 3$ and $L/K$ be a separable cubic extension, so that there is a primitive element $y\in L$ such that $y^3+a y+ a^2=0$ (see Corollary \ref{classification}). Then $L/K$ is constant if, and only if, $$u= \frac{(ja^2 + (w^3 + a w) )^2}{a^3} \in \mathbb{F}_q^*,$$ for some $w\in K$ and $j=1,2$. In other words, there is a generator $z$ of $L/K$ such that $z^3 +uz+u^2=0$, where $u \in \mathbb{F}_q^*$.
\end{lemme} 
\subsection{Splitting and ramification} 
In this section, we describe the splitting and ramifcation of any place of $K$ in a cubic extension $L/K$. As usual, we divide the analysis into the three fundamental cubic forms derived in Corollary \ref{classification}.

\begin{remarque}\label{rkp} 
Given any place $\mathfrak{p}$ of $K$, one may always choose (via weak approximation) an element $x \in K\backslash \mathbb{F}_q$ such that $\mathfrak{p}$ is not above the pole divisor $\mathfrak{p}_\infty$ of $x$ in $\mathbb{F}_q(x)$. As $[L:\mathbb{F}_q(x)] < \infty$, the integral closure $\mathcal{O}_{L,x}$ of the ring $\mathbb{F}_q[x]$ in $L$ is a Dedekind domain \cite[Theorem 5.7.7]{Vil} and a \emph{holomorphy ring} \cite[Corollary 3.2.8]{Sti}: A holomorphy ring is defined in any global field $L$ as the intersection $$\mathcal{O}_S = \bigcap_{\mathfrak{P} \in S} \mathcal{O}_\mathfrak{P}$$ for a nonempty, proper subset $S$ of the collection of all places of $L$. By \cite[Theorem 3.2.6]{Sti}, we have $\mathcal{O}_{L,x}=\mathcal{O}_S$, where $S$ is the collection of places of $L$ which lie above the places of $\mathbb{F}_q(x)$ associated to irreducible polynomials in $\mathbb{F}_q[x]$. By construction, the place $\mathfrak{p}$ of $K$ also lies above a place of $\mathbb{F}_q(x)$ associated to an irreducible polynomial in $\mathbb{F}_q[x]$, so that all places of $L$ which lie above $\mathfrak{p}$ belong to $S$. 
As $\mathcal{O}_{L,x}$ is Dedekind, it follows that $\mathfrak{p}$ factors in $\mathcal{O}_{L,x}$ as \begin{equation} \label{factorisation}\mathfrak{p} \mathcal{O}_{L,x} = \mathfrak{P}_1^{e(\mathfrak{P}_1|\mathfrak{p})} \cdots \mathfrak{P}_r^{e(\mathfrak{P}_r|\mathfrak{p})},\end{equation} where for each $i=1,\ldots,r$, the quantity $e(\mathfrak{P}_i|\mathfrak{p})$ is equal to the ramification index of $\mathfrak{P}_i|\mathfrak{p}$.  By \cite[Proposition 3.2.9]{Sti}, the collection $S$ is in one-to-one correspondence with the maximal ideals of $\mathcal{O}_S = \mathcal{O}_{L,x}$, so that the places $\mathfrak{P}_1,\ldots,\mathfrak{P}_r$ are independent of the choice of such an $x$. 
Also by \cite[Proposition 3.2.9]{Sti}, we have for each place $\mathfrak{P} \in S$ an isomorphism $$\mathcal{O}_{L,x}/\mathfrak{m}_\mathfrak{P} \cong \mathcal{O}_\mathfrak{P} / \mathfrak{P},$$ where $\mathfrak{m}_\mathfrak{P}$ is the maximal ideal of $\mathcal{O}_{L,x}$ associated to $\mathfrak{P}$,  $\mathcal{O}_\mathfrak{P}$ is the valuation ring at $\mathfrak{P}$, and
$\mathcal{O}_\mathfrak{P} / \mathfrak{P}$ is the residue field at the place $\mathfrak{P}$. It follows that for each $i=1,\ldots,r$, the inertia degree $f(\mathfrak{P}_i|\mathfrak{p})$, too, is independent of such a choice of $x$. Finally, the ramification indices $e(\mathfrak{P}_i|\mathfrak{p})$ are also independent of this choice of $x$: The localisation $(\mathcal{O}_{L,x})_{\mathfrak{P}_i}$ has maximal ideal $\mathfrak{m}_ {\mathfrak{P}_i}$, and by the factorisation \eqref{factorisation}, we have $$\mathfrak{p}(\mathcal{O}_{L,x})_{\mathfrak{P}_i} = \mathfrak{m}_ {\mathfrak{P}_i}^{e(\mathfrak{P}_i|\mathfrak{p})}.$$ One may easily show that $(\mathcal{O}_{L,x})_{\mathfrak{P}_i} = \mathcal{O}_{\mathfrak{P}_i}$, so that by \cite[Corollary 11.2]{Neu}, $$\mathcal{O}_{L,x} / \mathfrak{P}_i^{e(\mathfrak{P}_i|\mathfrak{p})} \mathcal{O}_{L,x} \cong  (\mathcal{O}_{L,x})_{\mathfrak{P}_i}/\mathfrak{p}(\mathcal{O}_{L,x})_{\mathfrak{P}_i} = \mathcal{O}_{\mathfrak{P}_i}/\mathfrak{p}(\mathcal{O}_{\mathfrak{P}_i})_{\mathfrak{P}_i},$$ where $\mathcal{O}_{\mathfrak{P}_i}/\mathfrak{p}(\mathcal{O}_{\mathfrak{P}_i})_{\mathfrak{P}_i}$, and hence $e(\mathfrak{P}_i|\mathfrak{p})$, is independent of $x$. The $2r$-tuple $$(e(\mathfrak{P}_1|\mathfrak{p}),f(\mathfrak{P}_1|\mathfrak{p});\ldots; e(\mathfrak{P}_r|\mathfrak{p}),f(\mathfrak{P}_r|\mathfrak{p}))$$ is called the \emph{signature} of $\mathfrak{p}$ in $L$. 
\end{remarque}

\subsubsection{$X^3 -a$, $a\in K$, when $p\neq 3$} 
If the extension $L/K$ is purely cubic, one may find a purely cubic generator of a form which is well-suited to a determination of splitting and ramification, as in the following lemma.
\begin{lemma}\label{purelocalstandardform}
Let $L/K$ be a purely cubic extension. Given a place $\mathfrak{p}$ of $K$, one may select a primitive element $y$ with minimal polynomial of the form $X^3 -a$ such that either 
\begin{enumerate} 
\item $v_\mathfrak{p}(a)=1,2$, or
\item $v_\mathfrak{p} (a)=0$. 
\end{enumerate}
Such a generator $y$ is said to be in {\sf local standard form} at $\mathfrak{p}$.
\end{lemma} 
\begin{proof} 
Let $y$ be a generator of $L$ such that $y^3=a \in K$. Given a place $\mathfrak{p}$ of $K$, we write $v_\mathfrak{p} (a)= 3 j + r$ with $r=0,1,2$. Via weak approximation, one may find an element $c \in K$ such that $v_\mathfrak{p} (c)= j$. Then $\frac{y}{c}$ is a generator of $L$ such that 
$$ \left(\frac{y}{c}\right)^3 =\frac{y^3}{c^3}= \frac{a}{c^3}$$
and $v_\mathfrak{p} \left( \frac{a}{c^3}\right) = r$. Hence the result.
\end{proof}
When a purely cubic extension $L/K$ is separable, one may also easily determine the fully ramified places in $L/K$.
\begin{theoreme}\label{RPC}
Let $p\neq 3$, and let $L/K$ be a purely cubic extension. Given a purely cubic generator $y$ with minimal polynomial $X^3 -a$, a place $\mathfrak{p}$ of $K$ is fully ramified if, and only if, $(v_\mathfrak{p} (a), 3)=1$. 
\end{theoreme} 
\begin{proof} 
Let $\mathfrak{p}$ be a place of $K$ and $\mathfrak{P}$ be a place of $L$ above $\mathfrak{p}$. Suppose that $(v_\mathfrak{p} (a), 3)=1$. Then
$$3v_\mathfrak{P} (y) = v_\mathfrak{P} (y^3)= v_\mathfrak{P} (a)= e(\mathfrak{P}|\mathfrak{p})  v_\mathfrak{p} (a).$$
Since $ (v_\mathfrak{p} (a), 3)=1$, we obtain $ 3 | e(\mathfrak{P}|\mathfrak{p}) $, and as $e(\mathfrak{P}|\mathfrak{p}) \leq 3$, it follows that $e(\mathfrak{P}|\mathfrak{p}) =3$, so that $\mathfrak{p}$ is fully ramified in $L$. 

Conversely, suppose that $(v_\mathfrak{p} (a), 3)\neq 1$. By Lemma \ref{purelocalstandardform}, we know that there exists a generator $z$ of $L$ such that $z^3 -c=0$ and $v_\mathfrak{p} (c)=0$. By Lemma \ref{splittingpolypurelycubic}, we know that the polynomial $X^3 -a$ is either 
\begin{enumerate} 
\item irreducible modulo $\mathfrak{p}$,
\item $X^3 -a = (X-\alpha ) Q(X) \mod \mathfrak{p}$, where $\alpha \in k(\mathfrak{p})$ and $Q(X)$ is an irreducible quadratic polynomial over modulo $\mathfrak{p}$, or
\item $f(X)  = (X-\alpha ) (X-\beta) (X- \gamma)$ modulo $\mathfrak{p}$ with $\alpha , \beta, \gamma \in k(\mathfrak{p})$ all distinct.
\end{enumerate}
In any of these cases, by Kummer's theorem \cite[Theorem 3.3.7]{Sti}, $\mathfrak{p}$ is either inert or there exist $2$ or $3$ place above it in $L$. Thus, $\mathfrak{p}$ cannot be fully ramified in any case.
\end{proof}
By Kummer's theorem \cite[Theorem 3.3.7]{Sti} and Lemma \ref{splittingpolypurelycubic}, we may also obtain the following theorem for places which do not fully ramify.
\begin{theoreme}\label{whatisthisidontknow}
Let $p \neq 3$, and let $L/K$ be a purely cubic extension. Given a place $\mathfrak{p}$ of $K$ that is not fully ramified, one may select a primitive element $y$ with minimal polynomial of the form $X^3 -a$ where $v_\mathfrak{p} (a)=0$. Moreover one can find $x\in K$ such that $\mathfrak{p}$ is a finite place over $\mathbb{F}_q(x)$ and the integral closure of $\mathbb{F}_q[x]$ in $L$ is a Dedekind Domain (see Remark \ref{rkp}). For this place $\mathfrak{p}$ and choice of $a$ and $x$, we have:
\begin{enumerate}
\item $\mathfrak{p}$ is completely split if, and only if, $|k (\mathfrak{p})| \equiv 1 \mod 3$ and $a^{\frac{|k (\mathfrak{p})| -1  }{3}} \equiv 1 \mod \mathfrak{p}$. Moreover, $f(\mathfrak{P}|\mathfrak{p})=1$ for any place $\mathfrak{P}$ of $L$ above $\mathfrak{p}$. The signature of $\mathfrak{p}$ is $(1,1; 1,1; 1,1)$. 
\item $\mathfrak{p}$ is inert  if, and only if, $|k (\mathfrak{p})| \equiv 1 \mod 3$ and $a^{\frac{|k (\mathfrak{p})| -1  }{3}} \notequiv 1 \mod \mathfrak{p}$. Moreover, $f(\mathfrak{P}|\mathfrak{p})=3$ for $\mathfrak{P}$ is a place of $L$ above $\mathfrak{p}$. The signature of $\mathfrak{p}$ is $(1,3)$. 
\item $\mathfrak{p}\mathcal{O}_{L,x}= \mathfrak{P}_1 \mathfrak{P}_2$ where $\mathfrak{P}_1$ and $\mathfrak{P}_2$ are two distinct places of $L$ if, and only if, $|k (\mathfrak{p})|  \equiv -1 \mod 3$. In this case, up to relabelling, the place $\mathfrak{P}_1$ satisfies $f(\mathfrak{P}_1| \mathfrak{p})=1$ and is inert in $L(r)$, whereas $f(\mathfrak{P}_2 | \mathfrak{p})=2$ and $\mathfrak{P_2}$ splits in $L(r)$. The signature of $\mathfrak{p}$ is $(1,1; 1,2)$. 
\end{enumerate} 
\end{theoreme}
\subsubsection{$X^3-3X-a$, $a\in K$,  $p\neq 3$} 
In order to determine the fully ramified places in extensions of this type, we begin with an elementary, but useful, lemma. These criteria and notation will be employed throughout what follows.
\begin{lemme}\label{purelycubicclosure}
Consider the polynomial $X^2 +a X +1$ where $a \in K$, and let $c_-,c_+$ denote the roots of this polynomial. Then $ c_+ \times c_- =1$, $c_+ + c_-= -a$, and $ \sigma( c_\pm ) = c_\mp$, where $\text{\emph{Gal}}(K(c_\pm ) /K)= \{ Id , \sigma \}$ when $c_\pm \notin K$. Let $\mathfrak{p}$ be a place of $K$ and $\mathfrak{p}_{c_\pm}$ be a place of $K(c_\pm )$ above $\mathfrak{p}$. Furthermore, we have:
\begin{enumerate} 
\item For any place $\mathfrak{p}_{c_\pm}$ of $K({c_\pm})$,
$$v_{\mathfrak{p}_{c_\pm}}(c_\pm )= -v_{\mathfrak{p}_{c_\pm}}(c_\mp ).$$
\item For any place $\mathfrak{p}_{c_\pm}$ of $K({c_\pm})$ above a place $\mathfrak{p}$ of $K$ such that $v_{\mathfrak{p}}(a) <0$,
$$ v_{\mathfrak{p}}(a)=-|v_{\mathfrak{p}_{c_\pm}}(c_\pm )|,$$ 
and otherwise, $v_{\mathfrak{p}_{c_\pm}}(c_\pm )=0$.
\end{enumerate}
\end{lemme} 
\begin{proof} 
\begin{enumerate}
\item At any place $\mathfrak{p}_{c_\pm}$ of $K({c_\pm})$, we have 
$$v_{\mathfrak{p}_{c_\pm}}(c_+\times c_-) = v_{\mathfrak{p}_{c_\pm}} (c_+)+ v_{\mathfrak{p}_{c_\pm}}( c_-)=v_{\mathfrak{p}_{c_\pm}}( 1)=0,$$ whence
$$v_{\mathfrak{p}_{c_\pm}}(c_+)=-v_{\mathfrak{p}_{c_\pm}}( c_-).$$
\item As $c_\pm^2 +a  c_\pm +1 =0,$ the elements $c_\pm' =\frac{ c_\pm }{a}$ satisfy
$$c_\pm'^2 + c_\pm' + \frac{1}{a^2} =0.$$
Thus, for any place $\mathfrak{p}_{c_\pm}$ of $K({c_\pm})$ above a place $\mathfrak{p}$ of $K$ such that $v_{\mathfrak{p}}(a) <0$, we obtain
$$v_{\mathfrak{p}_{c_\pm}}(c_\pm'^2 + c_\pm' ) = - 2 v_{\mathfrak{p}_{c_\pm}}(a)  > 0.$$
By the non-Archimedean triangle inequality, this is possible if, and only if, $v_{\mathfrak{p}_{c_\pm}}(c_\pm' ) >0$ or  $v_{\mathfrak{p}_{c_\pm}}(c_\pm' ) =0$. If $v_{\mathfrak{p}_{c_\pm}}(c_\pm' ) >0$, then
$$v_{\mathfrak{p}_{c_\pm}}(c_\pm' )=  - 2 v_{\mathfrak{p}_{c_\pm}}(a) \quad \text{ and } \quad v_{\mathfrak{p}_{c_\pm}}(c_\pm )=  -  v_{\mathfrak{p}_{c_\pm}}(a).$$ 
If on the other hand $v_{\mathfrak{p}_{c_\pm}}(c_\pm' ) =0$, we obtain $v_{\mathfrak{p}_{c_\pm}}(c_\pm)=   v_{\mathfrak{p}_{c_\pm}}(a)$. 
Thus, the latter together with part $(1)$ of this lemma implies that either 
$$v_{\mathfrak{p}_{c_\pm}}(c_+) =  v_{\mathfrak{p}_{c_\pm}}(a) \quad \text{and} \quad v_{\mathfrak{p}_{c_\pm}}(c_-)=-v_{\mathfrak{p}_{c_\pm}}(a)$$ or vice versa (with the roles of $c_-$ and $c_+$ interchanged). 
Moreover, note that $\mathfrak{p}_{c_\pm}$ is unramified in $K(c_\pm)/ K$ so that $ v_{\mathfrak{p}_{c_\pm}}(a)=  v_{\mathfrak{p}}(a)$. For if, when $p\neq 2$, then $K(c_\pm)/K$ has a generator $w$ such that $w^2 = -27 (a^2-4)$ and $ 2| v_{\mathfrak{p}}( -27 (a^2-4))$, thus by Kummer theory, $\mathfrak{p}$ is unramified and when $p =2$, $K(c_\pm)/K$ has a generator $w$ such that $w^2 -w= \frac{1}{a^2}$ and $  v_{\mathfrak{p}}( \frac{1}{a^2})\geq 0$, thus by Artin-Schreier theory, we have that $\mathfrak{p}$ is unramified in $K(c_\pm)$, thus the first part of $(2)$. 
For any place $\mathfrak{p}_{c_\pm}$ of $K(b)$ above a place $\mathfrak{p}$ of $K$ such that $v_{\mathfrak{p}}(a) > 0$. As
$$v_{\mathfrak{p}_{c_\pm}}(c_\pm'^2 + c_\pm' ) = - 2 v_{\mathfrak{p}_{c_\pm}}(a)  < 0,$$
again using the non-Archimedean triangle inequality, we can only have $v_{\mathfrak{p}_{c_\pm}}(c_\pm'^2) <0$, whence $v_{\mathfrak{p}_{c_\pm}}(c_\pm'^2 ) =  - 2 v_{\mathfrak{p}_{c_\pm}}(a).$ 
This implies that $v_{\mathfrak{p}_{c_\pm}}(c_\pm ) =0$. Finally, via the triangle inequality once more, for any $\mathfrak{p}_{c_\pm}$ such that $v_{\mathfrak{p}_{c_\pm}}(a) = 0$, we must have
$v_{\mathfrak{p}_{c_\pm}}(c_\pm ) =0.$
\end{enumerate}
\end{proof}

\begin{theorem} \label{ramification}
 Let $p \neq 3$, and let $L/K$ be an impurely cubic extension and $y$ primitive element with minimal polynomial $f(X) = X^3-3X-a$. Then the fully ramified places of $K$ in $L$ are precisely those $\mathfrak{p}$ such that $v_{\mathfrak{p}}(a)<0$ and $(v_{\mathfrak{p}}(a), 3)=1$. 
\end{theorem} 

\begin{proof} 
As usual, we let $\xi$ be a primitive $3^{rd}$ root of unity. We also let $r$ be a root of the quadratic resolvent $R(X)= X^2 +3aX + ( -27 + 9 a^2)$ of the cubic polynomial $X^3-3X-a$ in $\bar{K}$. As in \cite[Theorem 2.3]{Con}, we know that $L(r)/K(r)$ is Galois, and by Corollary \ref{purelycubicclosurecorollary}, we have that $L(\xi, r) / K(\xi , r)$ is purely cubic. We denote by $\mathfrak{p}$ a place in $K$, $\mathfrak{P}_{\xi,r}$ a place of $L(\xi,r)$ above $\mathfrak{p}$, $\mathfrak{P}=\mathfrak{P}_{\xi,r}\cap L$, and $\mathfrak{p}_{\xi,r}=\mathfrak{P}_{\xi,r}\cap K(\xi , r)$. By Corollary \ref{purelycubicclosurecorollary}, we know that that $L(\xi, r)/K(\xi, r)$ is Kummer; more precisely, there exists $v \in K(\xi , r)$ such that $v^3 = c$ where $c$ is a root of the polynomial $X^2+aX+1$. 

We thus obtain a tower $L(\xi, r) / K(\xi, r) /K(\xi) / K$ with $L(\xi, r )/K(\xi,r )$ Kummer of degree $3$, and where $K(\xi ,r)/ K(\xi)$ and $K(\xi)/K$ are both Kummer extensions of degree $2$. As the index of ramification is multiplicative in towers and the degree of $L(\xi , r) / K(\xi ,r)$ and $K(\xi ,r)/K$ are coprime, the places of $K$ that fully ramify in $L$ are those places of $K$ which lie below those of $K(\xi , r )$ which fully ramify in $L(\xi, r)/K(\xi, r)$. As $L(\xi,r)/K(\xi,r)$ is Kummer, the places of $K(\xi ,r)$ that ramify in $L(\xi,r)$ are described precisely by Kummer theory (see for example \cite[Example 5.8.9]{Vil} as those $\mathfrak{p}_{\xi,r}$ in $K(\xi ,r)$ such that $$(v_{\mathfrak{p}_{\xi,r}} (c_\pm ) , 3)=1.$$ 
Lemma \ref{purelycubicclosure} states that if $v_{\mathfrak{p}}(a) <0$, then $v_{\mathfrak{p}_{\xi, r}}(c_\pm )= \pm v_{\mathfrak{p}_{\xi,r}}(a)$ and that otherwise, $v_{\mathfrak{p}_{\xi,r}}(c_\pm )=0$. Thus, the ramified places of $L/F$ are those places $\mathfrak{p}$ below a place $\mathfrak{p}_{\xi , r}$ of $K(\xi ,r)$ such that $(v_{\mathfrak{p}_{\xi,r}}(a), 3)=1$. Also, $$v_{\mathfrak{p}_{\xi, r}}(a) = e(\mathfrak{p}_{\xi, r} | \mathfrak{p}) v_{\mathfrak{p}}(a),$$ where $e(\mathfrak{p}_{\xi, r} | \mathfrak{p})$ is the ramification index of a place $\mathfrak{p}$ of $K$ in $K(\xi, r)$, equal to $1$, $2$, or $4$, and in any case, coprime with $3$. Thus, $(v_{\mathfrak{p}_{\xi,r}}(a),3)=1$ if, and only if, $(v_{\mathfrak{p}}(a),3)=1$. As a consequence of the above argument, it therefore follows that a place $\mathfrak{p}$ of $K$ is fully ramified in $L$ if, and only if, $v_{\mathfrak{p}}(a)<0$.
\end{proof}

This theorem yields the following corollaries, the first being immediate.

\begin{corollaire} \label{ramificationgalois}
Suppose that $q \equiv - 1 \mod 3$. Let $L/K$ be a Galois cubic extension, so that there exists a primitive element $y$ of $L$ with minimal polynomial $f(X) = X^3-3X-a$. Then the (fully) ramified places of $K$ in $L$ are precisely those places $\mathfrak{p}$ of $K$ such that $v_{\mathfrak{p}}(a)<0$ and $(v_{\mathfrak{p}}(a), 3)=1$. 
\end{corollaire} 
\begin{corollaire} \label{odddegree}
Suppose that $q \equiv - 1 \mod 3$. Let $L/K$ be a Galois cubic extension, so that there exists a primitive element $y$  of $L$ with minimal polynomial $f(X) = X^3-3X-a$. Then, only those places of $K$ of even degree can (fully) ramify in $L$. More precisely, any place $\mathfrak{p}$ of $K$ such that $v_\mathfrak{p}(a) <0$ is of even degree. 
\end{corollaire}
\begin{proof} 
In Lemma \ref{purelycubicclosure}, it was noted that $ \sigma( c_\pm ) = c_\mp$ where $\text{\emph{Gal}}(K(c_\pm ) /K)= \{ Id , \sigma \}$, when $c_\pm \notin K$. Let $\xi$ again be a primitive $3^{rd}$ root of unity. We denote by $\mathfrak{p}$ a place of $K$ and $\mathfrak{p}_\xi$ a place of $K(\xi)$ above $\mathfrak{p}$. We find that
$$v_{\mathfrak{p}_\xi}(c_\pm) = v_{\sigma(\mathfrak{p}_\xi)}(\sigma ( c_\pm ))=v_{\sigma(\mathfrak{p}_\xi)}( c_\mp ).$$
Note that if $\sigma (  \mathfrak{p}_\xi )= \mathfrak{p}_\xi$, it follows that  $v_{\mathfrak{p}_\xi}(c_\pm) =v_{\mathfrak{p}_\xi}( c_\mp )$. However, by Lemma \ref{purelycubicclosure}, we have that, for any place $\mathfrak{p}_\xi$ of $K(\xi)$ above a place $\mathfrak{p}$ of $K$ such that $v_{\mathfrak{p}}(a) <0$, it holds that $v_{\mathfrak{p}_\xi}(c_\pm )= \pm v_{\mathfrak{p}_\xi}(a)$, 
and that $v_{\mathfrak{p}_\xi }(c_\pm )= -v_{\mathfrak{p}_\xi }(c_\mp )$. Thus, for any place $\mathfrak{p}_\xi$ of $K(\xi )$ above a place $\mathfrak{p}$ of $K$ such that $v_{\mathfrak{p}}(a) <0$, we find that $v_{\mathfrak{p}_\xi}(c_\pm )\neq v_{\mathfrak{p}_\xi}(c_\mp )$ and thus $\sigma (  \mathfrak{p}_\xi )\neq \mathfrak{p}_\xi$.
Therefore, by \cite[Theorem 6.2.1]{Vil}, we obtain that $ \mathfrak{p} $ is of even degree, for any place $\mathfrak{p}$ of $K$ such that $v_{\mathfrak{p}}(a)<0$.
\end{proof}
\begin{corollaire} \label{infty}
Suppose that $q \equiv - 1 \mod 3$. Let $L/K$ be a Galois cubic extension, so that there exists a primitive element $y$ of $L$ with minimal polynomial $f(X)=X^3-3 X -a$. Then one can choose a single place $\mathfrak{P}_\infty$ at infinity in $K$ such that $v_{\mathfrak{P}_\infty}(a) \geq 0$. 
\end{corollaire} 
\begin{proof} 

One can choose $x \in K \backslash \mathbb{F}_q$ such that the place $\mathfrak{p}_\infty$ at infinity for $x$ has the property that all of the places in $K$ above it are of odd degree. In order to accomplish this, we appeal to a method similar to the proof of \cite[Proposition 7.2.6]{Vil}; because there exists a divisor of degree 1 \cite[Theorem 6.3.8]{Vil}, there exists a prime divisor $\mathfrak{P}_\infty$ of $K$ of odd degree; for if all prime divisors of $K$ were of even degree, then the image of the degree function of $K$ would lie in $2\mathbb{Z}$, which contradicts \cite[Theorem 6.3.8]{Vil}. Let $d$ be this degree. Let $m \in \mathbb{N}$ be such that $m > 2 g_K - 1$. Then, by the Riemann-Roch theorem \cite[Corollary 3.5.8]{Vil}, it follows that there exists $x \in K$ such that the pole divisor of $x$ in $K$ is equal to $\mathfrak{P}_\infty^m$. By definition, the pole divisor of $x$ in $k(x)$ is equal to $\mathfrak{p}_\infty$. It follows that $$(\mathfrak{p}_\infty)_K = \mathfrak{P}_\infty^m,$$ from which it follows that $\mathfrak{P}_\infty$ is the unique place of $K$ above $\mathfrak{p}_\infty$, and by supposition that $\mathfrak{P}_\infty$ is of odd degree. From this argument, we obtain that, with this choice of infinity, all places above infinity in $k(x)$ are of odd degree. (We also note that we may very well choose $m$ relatively prime to $p$, whence $K/k(x)$ is also separable; in general, $K/k(x)$ as chosen will not be Galois.) 

As $q \equiv -1 \mod 3$, $L/K$ is a Galois extension, and $y$ is a primitive element with minimal polynomial of the form $X^3 -3 X -a$ where $a \in K$, we know that all of the places $\mathfrak{p}$ of $K$ such that $v_\mathfrak{p} (a)<0$, and in particular, all the ramified places, are of even degree (see Corollary \ref{odddegree}). It follows that the process described in this proof gives the desired construction, and the result follows.
\end{proof}
\begin{remarque}
We note that when $K$ is a rational function field, one may use Corollary \ref{infty} to show that the parameter $a$ has nonnegative valuation at $\mathfrak{p}_\infty$ for a choice of $x$ such that $K = \mathbb{F}_q(x)$, and thus such $\mathfrak{p}_\infty$ is unramified.
\end{remarque}
We now turn to a determination of splitting in impurely cubic extensions. We remind the reader that in what follows, as in Remark \ref{rkp}, while the integral closure $\mathcal{O}_{L,x}$ of $\mathbb{F}_q[x]$ in $L$ depends upon the choice of $x$, the residue field $k(\mathfrak{p})$ of a place $\mathfrak{p}$ of $K$ does not.
\begin{theorem}\label{splittingextqneq1mod3}
Let $p \neq 3$, let $L/K$ be an impurely cubic extension, so that there  exists a primitive element $y$ with minimal polynomial $f(X)=X^3-3 X -a$. 
We denote by $r \in \bar{K}$ a root of the quadratic resolvent $R(X)= X^2 +3aX + ( -27 + 9 a^2)$ of $f(X)$. Let $\mathfrak{p}$ be a place of $K$. One can find $x\in K$ such that $\mathfrak{p}$ is a finite place over $\mathbb{F}_q(x)$ and the integral closure of $\mathbb{F}_q[x]$ in $L$ is a Dedekind Domain (see Remark \ref{rkp}). For this place $\mathfrak{p}$ and choice of $x$, we have:

\begin{enumerate} 
\item When $v_{\mathfrak{p}}(a)\geq 0$, then $\mathfrak{p}$ splits as follows:

\begin{enumerate} 
\item $\mathfrak{p}$ is completely split in $L$ if, and only if, 
\begin{enumerate} 
\item $f(X) = (X-\alpha ) (X-\beta) (X- \gamma) \mod \mathfrak{p}$ with $\alpha , \beta, \gamma \in k (\mathfrak{p})$ all distinct or 
\item $f(X) = (X-\alpha ) (X-\beta)^2 \mod \mathfrak{p}$ with $\alpha , \beta \in k (\mathfrak{p})$ and either $r \in K$ or, $r\notin K$ and $\mathfrak{p}$ is totally split in $K(r)$. 
\end{enumerate} 
The signature of $\mathfrak{p}$ is $(1,1; 1,1; 1,1)$. 
\item $\mathfrak{p}$ is inert in $L$ if, and only if, $f(X)$ is irreducible modulo $\mathfrak{p}$. 
The signature of $\mathfrak{p}$ is $(1,3)$. 
\item $\mathfrak{p}\mathcal{O}_{L,x} = \mathfrak{P}_1 \mathfrak{P}_2$ where $\mathfrak{P}_i$, $i=1,2$ place of $L$ above $\mathfrak{p}$ if, and only if, $r \notin K$, $f(X) = (X-\alpha ) (X-\beta)^2 \mod \mathfrak{p}$ with $\alpha , \beta \in k (\mathfrak{p})$ and $\mathfrak{p}$ is inert in $K(r)$. In this case, up to relabelling, the place $\mathfrak{P}_1$ satisfies $f(\mathfrak{P}_1| \mathfrak{p})=1$ and is inert in $L(r)$, whereas $f(\mathfrak{P}_2 | \mathfrak{p})=2$ and $\mathfrak{P_2}$ splits in $L(r)$. The signature of $\mathfrak{p}$ is $(1,1; 1,2)$. 
\item $\mathfrak{p}\mathcal{O}_{L,x} = \mathfrak{P}_1^2 \mathfrak{P}_2$ where $\mathfrak{P}_i$, $i=1,2$ place of $L$ above $\mathfrak{p}$ if, and only if, $r \notin K$, $f(X) = (X-\alpha ) (X-\beta)^2 \mod \mathfrak{p}$ with $\alpha , \beta \in k (\mathfrak{p})$ and $\mathfrak{p}$ is ramified in $K(r)$. Moreover $f(\mathfrak{P}_i |\mathfrak{p})=1$, for $i=1,2$, $\mathfrak{P}_1$ is splits in $L(r)$ and $\mathfrak{P}_2$ is ramified in $L(r)$. 
The signature of $\mathfrak{p}$ is $(2,1; 1,1)$.  \end{enumerate} 
\item  When $v_{\mathfrak{p}}(a)< 0$, we denote by $\pi_\mathfrak{p}$ a prime element for $\mathfrak{p}$ (i.e., $v_\mathfrak{p}(\pi_\mathfrak{p}) = 1$). The place $\mathfrak{p}$ splits as follows:
\begin{enumerate} 
\item $\mathfrak{p}\mathcal{O}_{L,x} = \mathfrak{P}^3$, where $\mathfrak{P}$ is a place of $L$ above $\mathfrak{p}$ if, and only if, $( v_{\mathfrak{p}}(a),3)=1$. That is, $\mathfrak{p}$ is fully ramified. 
The signature of $\mathfrak{p}$ is $(3,1)$. 
\item $\mathfrak{p}$ is completely split in $L$ if, and only if, $3 \;|\; v_{\mathfrak{p}}(a)$, $|k (\mathfrak{p})|  \equiv 1 \mod \mathfrak{p}$ and $\left( \pi_\mathfrak{p}^{-v_{\mathfrak{p}}(a)}a\right)^{(|k (\mathfrak{p})| -1)/3} \equiv 1 \mod \mathfrak{p}$. 
The signature of $\mathfrak{p}$ is $(1,1; 1,1; 1,1)$. 
\item $\mathfrak{p}$ is inert if, and only if, $3 \;|\; v_{\mathfrak{p}}(a)$, $|k (\mathfrak{p})|  \equiv 1 \mod \mathfrak{p}$ and $\left( \pi_\mathfrak{p}^{-v_{\mathfrak{p}}(a)}a\right)^{(|k (\mathfrak{p})| -1)/3} \equiv \xi^i \mod \mathfrak{p}$, where $i=1,2$.  
The signature of $\mathfrak{p}$ is $(1,3)$. 
\item $\mathfrak{p}\mathcal{O}_{L,x}= \mathfrak{P}_1 \mathfrak{P}_2$ where $\mathfrak{P}_1$ and $\mathfrak{P}_2$ are two distinct places of $L$ if, and only if, $3 \;|\; v_{\mathfrak{p}}(a)$ and $|k (\mathfrak{p})|  \equiv -1 \mod \mathfrak{p}$. In this case, up to relabelling, the place $\mathfrak{P}_1$ satisfies $f(\mathfrak{P}_1| \mathfrak{p})=1$ and is inert in $L(r)$, whereas $f(\mathfrak{P}_2 | \mathfrak{p})=2$ and $\mathfrak{P_2}$ splits in $L(r)$. The signature of $\mathfrak{p}$ is $(1,1; 1,2)$. 
\end{enumerate}
\end{enumerate}  
\end{theorem} 

\begin{proof} 
\begin{enumerate}
\item  Let $\mathfrak{p}$ be a place of $K$. Suppose that $v_{\mathfrak{p}}(a)\geq 0$.
By Lemma \ref{splittingpolyqneq1mod3}, $f(X)$ can only be one of the following forms modulo $\mathfrak{p}$:
\begin{enumerate}[(I)]
\item $f(X)$ is irreducible;
\item $f(X)= (X-\alpha ) Q(X)$ where $\alpha \in  k (\mathfrak{p})$ and $Q(X)$ irreducible quadratic polynomial;
\item $f(X) = (X-\alpha ) (X-\beta)^2 $ with $\alpha , \beta \in  k (\mathfrak{p})$ and $\alpha \neq \beta$ ;
\item $f(X) = (X-\alpha ) (X-\beta) (X- \gamma)$ with $\alpha , \beta, \gamma \in  k (\mathfrak{p})$ all distinct.
\end{enumerate}
In case (I), by Kummer's theorem \cite[Theorem 3.3.7]{Sti}, $\mathfrak{p}$ is inert. In case (II), by Kummer's theorem, $\mathfrak{p}\mathcal{O}_{L,x} = \mathfrak{P}_1\mathfrak{P}_2$, where $\mathfrak{P}_i$ with $i=1,2$ are the places above $\mathfrak{p}$. In case (IV), by Kummer's theorem, $\mathfrak{p}$ splits completely. 

Case (III) is thus the only case where Kummer's theorem is not enough to decide. Suppose that $f(X) = (X-\alpha ) (X-\beta)^2 $ with $\alpha , \beta \in  k (\mathfrak{p})$ and $\alpha \neq \beta$, by Kummer's theorem, we know that there is at least two place above $\mathfrak{p}$ in $L$ thus either
\begin{enumerate} 
\item $\mathfrak{p}\mathcal{O}_{L,x} = \mathfrak{P}_1 \mathfrak{P}_2$ where $\mathfrak{P}_i$, $i=1,2$ place of $L$ above $\mathfrak{p}$, or 
\item  $\mathfrak{p}\mathcal{O}_{L,x} = \mathfrak{P}_1^2 \mathfrak{P}_2$ where $\mathfrak{P}_i$, $i=1,2$ place of $L$ above $\mathfrak{p}$, or
\item $\mathfrak{p}\mathcal{O}_{L,x} = \mathfrak{P}_1 \mathfrak{P}_2 \mathfrak{P}_3$ where $\mathfrak{P}_i$, $i=1,2,3$ place of $L$ above $\mathfrak{p}$.
\end{enumerate}
We now work on the tower $L(r)/K(r)/K$, where $r$ is a root the quadratic resolvent of $X^3-3X-a$. 

Since $X^3 - 3X-a$ has a root modulo $\mathfrak{p}$ then it has also a root modulo $\mathfrak{p}_r$ a place in $K(r)$ above $\mathfrak{p}$, and since $L(r)/K(r)$ is Galois, it follows by \cite[Theorem 2.3]{Con} that $\mathfrak{p}_r$ is completely split in $L(r)$. 

In particular, if $r\in K$, i.e., $L/K$ is Galois, then $\mathfrak{p}$ is completely split in $L$. 

Suppose now that $r \notin K$. By \cite[p. 55]{Neu}, we know that $\mathfrak{p}$ is completely split in $L$ if, and only if, (1) $\mathfrak{p}$ is completely split in $K(r)$ and (2) $\mathfrak{p}_r$ completely split in $L(r)$. If $\mathfrak{p}$ ramifies in $K(r)$, then $$\mathfrak{p} \mathcal{O}_{L(r)} = (\mathfrak{P}_{1,r}\mathfrak{P}_{2,r} \mathfrak{P}_{3,r})^2$$ where $\mathfrak{P}_{i,r}$, $i=1,2,3$ are places above $\mathfrak{p}$ in $L(r)$. Again via \cite[p. 55]{Neu}, $\mathfrak{p}$ cannot be completely split in $L$. Thus, either $\mathfrak{p}\mathcal{O}_{L,x} = \mathfrak{P}_1 \mathfrak{P}_2$ or $\mathfrak{p}\mathcal{O}_{L,x} = \mathfrak{P}_1^2 \mathfrak{P}_2$ where each $\mathfrak{P}_i$ ($i=1,2$) is a place of $L$ above $\mathfrak{p}$. Note that $2\;|\; e (  \mathfrak{P}_{r}| \mathfrak{p})$  for any places $\mathfrak{P}_r$ in $L(r)$ above $\mathfrak{p}$.  
If $\mathfrak{p}\mathcal{O}_{L,x} = \mathfrak{P}_1 \mathfrak{P}_2$, then as $e( \mathfrak{P}_i| \mathfrak{p} )=1$, we have that $2\;|\;e( \mathfrak{P}_{i,r}| \mathfrak{P}_i)$
and $\mathfrak{p}\mathcal{O}_{L,x} = \mathfrak{P}_{1,r}^2 \mathfrak{P}_{2,r}^2$, where $\mathfrak{P}_{i,r}$, $i=1,2$ are places above $\mathfrak{p}$ in $L(r)$, which is impossible, as $\mathfrak{p} \mathcal{O}_{L(r)} = (\mathfrak{P}_{1,r}\mathfrak{P}_{2,r} \mathfrak{P}_{3,r})^2$. Thus, in this case, we must have $\mathfrak{p}\mathcal{O}_{L,x} = \mathfrak{P}_1^2 \mathfrak{P}_2$ and $\mathfrak{P}_1$ is split in $K(r)$ and $\mathfrak{P}_2$ ramifies in $K(r)$.

 If $\mathfrak{p}$ is inert in $K(r)$, then by \cite[p. 55]{Neu}, $\mathfrak{p}$ is not completely split in $L$, and $$\mathfrak{p} \mathcal{O}_{L(r)} = \mathfrak{P}_{1,r}\mathfrak{P}_{2,r} \mathfrak{P}_{3,r}$$ $\mathfrak{P}_{i,r}$, $i=1,2,3$ are places above $\mathfrak{p}$ in $L(r)$.
 Thus, the only possibility is $\mathfrak{p}\mathcal{O}_{L,x} = \mathfrak{P}_1 \mathfrak{P}_2$ where each $\mathfrak{P}_i$ ($i=1,2$) is a place of $L$ above $\mathfrak{p}$. Moreover, up to relabelling, $\mathfrak{P}_1$ is inert in $L(r)$ and $\mathfrak{P}_2$ is split in $L(r)$. Hence the result.
\item 
Suppose $v_{\mathfrak{p}}(a)< 0$. When $(v_{\mathfrak{p}}(a),3)=1$, it is a consequence of Theorem \ref{ramification}. When $3| v_{\mathfrak{p}}(a)$, we let $t \in \mathbb{N}$ be such that $v_{\mathfrak{p}}(a)= -3t$. Then $z=\pi_\mathfrak{p}^t y$ is a root of the polynomial $X^3 -3 \pi_\mathfrak{p}^{2t} X -\pi_\mathfrak{p}^{-v_{\mathfrak{p}}(a)} a$, and $v_{\mathfrak{p}} (\pi_\mathfrak{p}^{-v_{\mathfrak{p}}(a)} a)=0$. In particular, we have $$X^3 -3\pi_\mathfrak{p}^{2t} X -\pi_\mathfrak{p}^{-v_{\mathfrak{p}}(a)} a \equiv X^3  -\pi_\mathfrak{p}^{-v_{\mathfrak{p}}(a)} a \mod \mathfrak{p}.$$ 
Thus the theorem is a consequence of Lemma \ref{splittingpolypurelycubic}.
\end{enumerate}
The inertia degrees are a direct consequence of the fundamental equality \cite[Proposition]{Sti}. 	
\end{proof}

We now determine the splitting only in terms of $a$. For the purpose of clarity, we split the the argument into cases $p=2$ and $p\neq 2$, as the proofs are different in each case.

\begin{corollaire}\label{splittingextqneq1mod31}
Suppose that $p\neq 2, 3$. Let $L/K$ be an impurely cubic extension, so that there exists a primitive element $y$ with minimal polynomial $f(X)=X^3-3 X -a$. 
We denote $\Delta = -27(a^2-4)$ the discriminant of $X^3-3X-a$, $\delta \in \overline{K}$ such that $\delta^2 = a^2 - 4$, and $r \in \overline{K}$ a root of the quadratic resolvent $R(X)= X^2 +3aX + ( -27 + 9 a^2)$ of $f(X)$. Let $\mathfrak{p}$ be a place of $K$. One can find $x\in K$ such that $\mathfrak{p}$ is a finite place over $\mathbb{F}_q(x)$ and the integral closure of $\mathbb{F}_q[x]$ in $L$ is a Dedekind Domain (see Remark \ref{rkp}). For this place $\mathfrak{p}$ and choice of $a$ and $x$, we have:
 \begin{enumerate} 
\item When $v_{\mathfrak{p}}(a)\geq 0$, then $\mathfrak{p}$ splits as follows:

\begin{enumerate} 
\item $\mathfrak{p}$ is completely split in $L$ if, and only if, 
\begin{enumerate}
\item $\Delta$ is not a square in $K$, $a\equiv \pm 2 \mod \mathfrak{p}$, and there is a $w\in K$ such that both $v_\mathfrak{p}(\Delta w^2)=0 $ and $\Delta w^2$ a square modulo $\mathfrak{p}$, or
\item  $\Delta $ is a square in $K$ and \begin{itemize} \item[$\bullet$] $|k (\mathfrak{p})|  \equiv 1 \mod 3$  and $\frac{1}{2} (a \pm \delta)$ is a cube modulo $\mathfrak{p}$, or \item[$\bullet$] $|k (\mathfrak{p})| =5$ and $a\notequiv \pm 1 \mod \mathfrak{p}$. \end{itemize}
\end{enumerate}
The signature of $\mathfrak{p}$ is $(1,1; 1,1; 1,1)$. 
\item $\mathfrak{p}$ is inert in $L$ if, and only if, 
 \begin{itemize}
 \item[$\bullet$]  $|k (\mathfrak{p})| \equiv 1 \mod 3$, $\Delta$ is a nonzero square modulo $\mathfrak{p}$  and $\frac{1}{2} (\overline{a} \pm \delta)$ is not a cube in $k (\mathfrak{p})$, or \item[$\bullet$]  $|k (\mathfrak{p})| =5$ and $a\equiv \pm 1 \mod \mathfrak{p}$.  \end{itemize}
The signature of $\mathfrak{p}$ is $(1,3)$. 
\item $\mathfrak{p}\mathcal{O}_{L,x} = \mathfrak{P}_1 \mathfrak{P}_2$ where $\mathfrak{P}_i$, $i=1,2$ place of $L$ above $\mathfrak{p}$ if, and only if, 
\begin{enumerate}
\item  $\Delta$ is not a square in $K$, $a\equiv \pm 2 \mod \mathfrak{p}$, and  there is no $w\in K$ such that both $v_\mathfrak{p}(\Delta w^2)=0 $ and $\Delta w^2$ is a square modulo $\mathfrak{p}$; 
\item $\Delta$ is a not square modulo $\mathfrak{p}$.
\end{enumerate}
In this case, up to relabelling, the place $\mathfrak{P}_1$ satisfies $f(\mathfrak{P}_1| \mathfrak{p})=1$ and is inert in $L(r)$, whereas $f(\mathfrak{P}_2 | \mathfrak{p})=2$ and $\mathfrak{P_2}$ splits in $L(r)$. The signature of $\mathfrak{p}$ is $(1,1; 1,2)$. 
\item $\mathfrak{p}\mathcal{O}_{L,x} = \mathfrak{P}_1^2 \mathfrak{P}_2$ where $\mathfrak{P}_i$, $i=1,2$ place of $L$ above $\mathfrak{p}$ if, and only if, $\Delta$ is not a square in $K$, $a \equiv \pm 2 \mod \mathfrak{p}$, and  $(v_{\mathfrak{p}}(\Delta ), 2)=1$. Moreover, $f(\mathfrak{P}_i |\mathfrak{p})=1$, for $i=1,2$, $\mathfrak{P}_1$ is split in $L(r)$, and $\mathfrak{P}_2$ is ramified in $L(r)$. The signature of $\mathfrak{p}$ is $(2,1; 1,1)$. 
\end{enumerate} 
\item  When $v_{\mathfrak{p}}(a)< 0$, we denote by $\pi_\mathfrak{p}$ a prime element for $\mathfrak{p}$ (i.e., $v_\mathfrak{p}(\pi_\mathfrak{p}) = 1$). The place $\mathfrak{p}$ splits as follows:
\begin{enumerate} 
\item $\mathfrak{p}\mathcal{O}_{L,x} = \mathfrak{P}^3$, where $\mathfrak{P}$ is a place of $L$ above $\mathfrak{p}$ if, and only if, $( v_{\mathfrak{p}}(a),3)=1$. That is, $\mathfrak{p}$ is fully ramified. 
The signature of $\mathfrak{p}$ is $(3,1)$. 
\item $\mathfrak{p}$ is completely split in $L$ if, and only if, $3 \;|\; v_{\mathfrak{p}}(a)$, $|k (\mathfrak{p})|  \equiv 1 \mod \mathfrak{p}$ and $\left( \pi_\mathfrak{p}^{-v_{\mathfrak{p}}(a)}a\right)^{(|k (\mathfrak{p})| -1)/3} \equiv 1 \mod \mathfrak{p}$. 
The signature of $\mathfrak{p}$ is $(1,1; 1,1; 1,1)$. 
\item $\mathfrak{p}$ is inert if, and only if, $3 \;|\; v_{\mathfrak{p}}(a)$, $|k (\mathfrak{p})|  \equiv 1 \mod \mathfrak{p}$ and $\left( \pi_\mathfrak{p}^{-v_{\mathfrak{p}}(a)}a\right)^{(|k (\mathfrak{p})| -1)/3} \equiv \xi^i \mod \mathfrak{p}$, where $i=1,2$.  
The signature of $\mathfrak{p}$ is $(1,3)$. 
\item $\mathfrak{p}\mathcal{O}_{L,x}= \mathfrak{P}_1 \mathfrak{P}_2$ where $\mathfrak{P}_1$ and $\mathfrak{P}_2$ are two distinct places of $L$ if, and only if, $3 \;|\; v_{\mathfrak{p}}(a)$ and $|k (\mathfrak{p})|  \equiv -1 \mod \mathfrak{p}$. In this case, up to relabelling, the place $\mathfrak{P}_1$ satisfies $f(\mathfrak{P}_1| \mathfrak{p})=1$ and is inert in $L(r)$, whereas $f(\mathfrak{P}_2 | \mathfrak{p})=2$ and $\mathfrak{P_2}$ splits in $L(r)$. The signature of $\mathfrak{p}$ is $(1,1; 1,2)$. 
\end{enumerate}
\end{enumerate}  
\end{corollaire} 
\begin{proof} 
This is a consequence of Theorem \ref{splittingextqneq1mod3}, Lemma  \ref{splittingpolyqneq1mod3}, and the following remark: 
Note that when $p\neq 2$, $K(r)$ is a quadratic Kummer extension of $K$ with a generator $y$ such that $y^2 = \Delta$. Moreover, by a study of how a polynomial of the form $X^2 -d$ modulo $\mathfrak{p}$ decomposes over a finite field of characteristic not $2$, using again Kummer's threorem \cite[Theorem 3.3.7]{Sti} and general Kummer Theory \cite[Proposition 3.7.3]{Sti}, one may find that:
\begin{enumerate}
\item[$\bullet$] $\mathfrak{p}$ is ramified in $K(r)$ if, and only if, $(v_{\mathfrak{p}}(\Delta), 2)=1$;
\item[$\bullet$] $\mathfrak{p}$ is completely split in $K(r)$ if, and only if, there is a $w\in K$ such that $v_\mathfrak{p}(\Delta w^2)=0 $ and $\Delta w^2$ is a square modulo $\mathfrak{p}$;
\item[$\bullet$] $\mathfrak{p}$ is inert in $K(r)$ if, and only if, there is no $w\in K$ such that $v_\mathfrak{p}(\Delta w^2)=0 $ and $\Delta w^2$ is a square modulo $\mathfrak{p}$.
\end{enumerate}	
\end{proof}

We now turn to the case $p = 2$.

\begin{corollaire}\label{splittingextqneq1mod32}
Suppose that $p=2$. Let $L/K$ be an impurely cubic extension, so that there exists a primitive element $y$ with minimal polynomial $f(X)=X^3-3 X -a$. We denote by $|k (\mathfrak{p})| = 2^m $ the cardinality of the residue field at $\mathfrak{p}$ and $r \in \overline{K}$ a root of the quadratic resolvent $R(X)= X^2 +3aX + ( -27 + 9 a^2)$ of the cubic polynomial $X^3-3X-a$. Let $\mathfrak{p}$ be a place of $K$ and $\emph{Tr}$ the trace map from $k(\mathfrak{p})$ into the prime field $\mathbb{F}_p$ (\S 2). One can find $x\in K$ such that $\mathfrak{p}$ is a finite place over $\mathbb{F}_q(x)$ and the integral closure of $\mathbb{F}_q[x]$ in $L$ is a Dedekind Domain (see Remark \ref{rkp}). For this place $\mathfrak{p}$ and choice of $x$, we have:

\begin{enumerate} 
\item When $v_{\mathfrak{p}}(a)\geq 0$, then $\mathfrak{p}$ splits as follows:

\begin{enumerate} 
\item $\mathfrak{p}$ is completely split in $L$ if, and only if, 

\begin{enumerate}
\item $a\equiv 0 \mod \mathfrak{p}$  and  $r\in K$, or 
\item $a\equiv 0 \mod \mathfrak{p}$, $r \notin K $ and there is $w \in K$ such that $v_{\mathfrak{p}}( \frac{1}{a^2}+1 -w^2+w )\geq 0$ and $\text{\emph{Tr}}(\frac{1}{a^2}+1 -w^2+w)=0 \mod \mathfrak{p}$.
\item  $a \notequiv 0 \mod \mathfrak{p}$, $\text{\emph{Tr}}(1/a^2) \equiv \text{\emph{Tr}}(1) \mod \mathfrak{p}$ and the roots of $T^2 +aT +1 \mod \mathfrak{p}$ are cubes in $\mathbb{F}_{2^m}$, when $m$ is even (resp. in $\mathbb{F}_{2^{2m}}$, when $m$ is odd). 
\end{enumerate}
The signature of $\mathfrak{p}$ is $(1,1; 1,1; 1,1)$. 
\item $\mathfrak{p}$ is inert in $L$ if, and only if, $a \notequiv 0 \mod \ \mathfrak{p}$, $\text{\emph{Tr}}(1/a^2) \equiv \text{\emph{Tr}}(1) \mod \mathfrak{p}$ and the roots of $T^2 +aT +1 \mod \mathfrak{p}$ are non-cubes in $\mathbb{F}_{2^m}$, when $m$ is even (resp. in $\mathbb{F}_{2^{2m}}$, when $m$ is odd).
The signature of $\mathfrak{p}$ is $(1,3)$. 
\item $\mathfrak{p}\mathcal{O}_{L,x} = \mathfrak{P}_1 \mathfrak{P}_2$ where $\mathfrak{P}_i$, $i=1,2$ place of $L$ above $\mathfrak{p}$ if, and only if, $r \notin K$ and either \begin{itemize} \item[$\bullet$]  $a\equiv 0 \mod \mathfrak{p}$, there is $w \in K$ such that $v_{\mathfrak{p}}( \frac{1}{a^2}+1 -w^2+w )\geq 0$ and $\text{\emph{Tr}}(\frac{1}{a^2}+1 -w^2+w)\not\equiv 0 \mod \mathfrak{p}$, or \item[$\bullet$]  $a\notequiv 0 \mod \mathfrak{p}$ and
 $\text{\emph{Tr}}(1/a^2) \not\equiv \text{\emph{Tr}}(1)\mod \mathfrak{p}$. \end{itemize} In this case, up to relabelling, the place $\mathfrak{P}_1$ satisfies $f(\mathfrak{P}_1| \mathfrak{p})=1$ and is inert in $L(r)$, whereas $f(\mathfrak{P}_2 | \mathfrak{p})=2$ and $\mathfrak{P_2}$ splits in $L(r)$. The signature of $\mathfrak{p}$ is $(1,1; 1,2)$. 
\item $\mathfrak{p}\mathcal{O}_{L,x} = \mathfrak{P}_1^2 \mathfrak{P}_2$ where $\mathfrak{P}_i$, $i=1,2$ place of $L$ above $\mathfrak{p}$ if, and only if, $r\notin K$, $a \equiv 0 \mod \mathfrak{p}$ and there is $w \in K$ such that $(v_{\mathfrak{p}}( \frac{1}{a^2}+1 -w^2+w ) , 2)=1$ and $v_{\mathfrak{p}}( \frac{1}{a^2} +1 -w^2+w )<0$. Moreover $f(\mathfrak{P}_i |\mathfrak{p})=1$, for $i=1,2$, $\mathfrak{P}_1$ is split in $L(r)$, and $\mathfrak{P}_2$ is ramified in $L(r)$.  The signature of $\mathfrak{p}$ is $(2,1; 1,1)$. 
\end{enumerate} 
\item  When $v_{\mathfrak{p}}(a)< 0$, we denote by $\pi_\mathfrak{p}$ a prime element for $\mathfrak{p}$ (i.e., $v_\mathfrak{p}(\pi_\mathfrak{p}) = 1$). The place $\mathfrak{p}$ splits as follows:
\begin{enumerate} 
\item $\mathfrak{p}\mathcal{O}_{L,x} = \mathfrak{P}^3$, where $\mathfrak{P}$ is a place of $L$ above $\mathfrak{p}$ if, and only if, $( v_{\mathfrak{p}}(a),3)=1$. That is, $\mathfrak{p}$ is fully ramified. 
The signature of $\mathfrak{p}$ is $(3,1)$. 
\item $\mathfrak{p}$ is completely split in $L$ if, and only if, $3 \;|\; v_{\mathfrak{p}}(a)$, $|k (\mathfrak{p})|  \equiv 1 \mod \mathfrak{p}$ and $\left( \pi_\mathfrak{p}^{-v_{\mathfrak{p}}(a)}a\right)^{(|k (\mathfrak{p})| -1)/3} \equiv 1 \mod \mathfrak{p}$. 
The signature of $\mathfrak{p}$ is $(1,1; 1,1; 1,1)$. 
\item $\mathfrak{p}$ is inert if, and only if, $3 \;|\; v_{\mathfrak{p}}(a)$, $|k (\mathfrak{p})|  \equiv 1 \mod \mathfrak{p}$ and $\left( \pi_\mathfrak{p}^{-v_{\mathfrak{p}}(a)}a\right)^{(|k (\mathfrak{p})| -1)/3} \equiv \xi^i \mod \mathfrak{p}$, where $i=1,2$. 
The signature of $\mathfrak{p}$ is $(1,3)$. 
\item $\mathfrak{p}\mathcal{O}_{L,x}= \mathfrak{P}_1 \mathfrak{P}_2$ where $\mathfrak{P}_1$ and $\mathfrak{P}_2$ are two distinct places of $L$ if, and only if, $3 \;|\; v_{\mathfrak{p}}(a)$ and $|k (\mathfrak{p})|  \equiv -1 \mod \mathfrak{p}$. In this case, up to relabelling, the place $\mathfrak{P}_1$ satisfies $f(\mathfrak{P}_1| \mathfrak{p})=1$ and is inert in $L(r)$, whereas $f(\mathfrak{P}_2 | \mathfrak{p})=2$ and $\mathfrak{P_2}$ splits in $L(r)$. The signature of $\mathfrak{p}$ is $(1,1; 1,2)$. 
\end{enumerate}
\end{enumerate}  
\end{corollaire} 
\begin{proof} 
This is a consequence of Theorem \ref{splittingextqneq1mod3}, Lemma  \ref{splittingpolyqneq1mod3}, and the following remark: Note that when $p=  2$, $K(r)$ is a quadratic Artin-Schreier extension of $K$ with a generator $y$ such that $y^2-y = 1+ \frac{1}{a^2}$. Moreover, studying the decomposition of a polynomial of the form $X^2-X-d$, $d\in K$ decomposes over a finite field of characteristic $2$ (see, for example, \cite[Proposition 1]{Pomm}), using Kummer's theorem \cite[Theorem 3.3.7]{Sti} and Artin-Schreier theory \cite[Proposition 3.7.8]{Sti}, we find that: 
\begin{itemize}
\item[$\bullet$] $\mathfrak{p}$ is ramified in $K(r)$ if, and only if, there is $w \in K$ such that $(v_{\mathfrak{p}}( \frac{1}{a^2}+1 -w^2+w ) , 2)=1$ and $v_{\mathfrak{p}}( \frac{1}{a^2} +1 -w^2+w )<0$,
\item[$\bullet$] $\mathfrak{p}$ is completely split in $K(r)$ if, and only if, there is $w \in K$ such that $v_{\mathfrak{p}}( \frac{1}{a^2}+1 -w^2+w )\geq 0$ and $\text{Tr}(\frac{1}{a^2}+1 -w^2+w)\equiv  0 \mod \mathfrak{p}$, and
\item[$\bullet$] $\mathfrak{p}$ is inert in $K(r)$ if, and only if, there is $w \in K$ such that $v_{\mathfrak{p}}( \frac{1}{a^2}+1 -w^2+w )\geq 0$ and $\text{Tr}(\frac{1}{a^2}+1 -w^2+w) \not\equiv  0 \mod \mathfrak{p}$.
\end{itemize}
Hence the result.
\end{proof}

We may then conclude the splitting for Galois extensions using the form found in Corollary \ref{galoisglobal}.
\begin{corollaire}\label{splittingextqneq1mod33}
Suppose that $q \equiv -1 \mod 3$. Let $L/K$ be a Galois cubic extension, so that there is a primitive element $y$ with minimal polynomial $f(X)=X^3-3 X -a$. As a consequence of Theorem \ref{galoisglobal}, we may write the minimal polynomial $$T(X)=X^3-3X-a,$$ where $$a = \frac{2A^2 -2 AB- B^2}{A^2 - A B+ B^2},$$ for some $A,B\in \mathcal{O}_{K}$. Let $\mathfrak{p}$ be a place of $K$. One can find $x\in K$ such that $\mathfrak{p}$ is a finite place over $\mathbb{F}_q(x)$ and the integral closure of $\mathbb{F}_q[x]$ in $L$ is a Dedekind Domain (see Remark \ref{rkp}). For this place $\mathfrak{p}$ and choice of $x$, we have:
\begin{enumerate} 
\item When $v_{\mathfrak{p}}(a)> 0$, then $\mathfrak{p}$ completely split. 
The signature of $\mathfrak{p}$ is $(1,1; 1,1; 1,1)$. 
\item When $v_{\mathfrak{p}}(a)=0$, then $\mathfrak{p}$ splits as follows:

\begin{enumerate} 
\item $\mathfrak{p}$ is inert in $L$ if, and only if, 
\begin{enumerate}
\item $|k (\mathfrak{p})| \equiv 1 \mod 3$ and $\frac{A +\xi^2 B}{A + \xi B}$ is not a cube modulo $\mathfrak{p}$; 
\item $|k (\mathfrak{p})|  = 2^m$, $m$ odd, and $\frac{A +\xi^2 B}{A + \xi B}$ is not a cube in $\mathbb{F}_{2^{2m}}$; or
\item $|k (\mathfrak{p})| =5$, $A\equiv \pm B \mod \mathfrak{p}$, $A\equiv  2 B \mod \mathfrak{p}$, or $A\equiv  0 \mod \mathfrak{p}$.
\end{enumerate}
The signature of $\mathfrak{p}$ is $(1,3)$. 
\item $\mathfrak{p}$ is completely split in $L$, otherwise. 
The signature of $\mathfrak{p}$ is $(1,1; 1,1; 1,1)$. 
\end{enumerate} 
\item  When $v_{\mathfrak{p}}(a)< 0$, we denote by $\pi_\mathfrak{p}$ a prime element for $\mathfrak{p}$ (i.e., $v_\mathfrak{p}(\pi_\mathfrak{p}) = 1$). The place $\mathfrak{p}$ splits as follows:
\begin{enumerate} 
\item $\mathfrak{p}\mathcal{O}_{L,x} = \mathfrak{P}^3$, where $\mathfrak{P}$ is a place of $L$ above $\mathfrak{p}$ if, and only if, $( v_{\mathfrak{p}}(a),3)=1$. That is, $\mathfrak{p}$ is fully ramified. 
The signature of $\mathfrak{p}$ is $(3,1)$. 
\item $\mathfrak{p}$ is completely split in $L$ if, and only if, $3 \;|\; v_{\mathfrak{p}}(a)$, $|k (\mathfrak{p})|  \equiv 1 \mod \mathfrak{p}$ and $\left( \pi_\mathfrak{p}^{-v_{\mathfrak{p}}(a)}a\right)^{(|k (\mathfrak{p})| -1)/3} \equiv 1 \mod \mathfrak{p}$. 
The signature of $\mathfrak{p}$ is $(1,1; 1,1; 1,1)$. 
\item $\mathfrak{p}$ is inert if, and only if, $3 \;|\; v_{\mathfrak{p}}(a)$, $|k (\mathfrak{p})|  \equiv 1 \mod \mathfrak{p}$ and $\left( \pi_\mathfrak{p}^{-v_{\mathfrak{p}}(a)}a\right)^{(|k (\mathfrak{p})| -1)/3} \equiv \xi^i \mod \mathfrak{p}$, where $i=1,2$.  
The signature of $\mathfrak{p}$ is $(1,3)$. 
\item $\mathfrak{p}\mathcal{O}_{L,x}= \mathfrak{P}_1 \mathfrak{P}_2$ where $\mathfrak{P}_1$ and $\mathfrak{P}_2$ are two distinct places of $L$ if, and only if, $3 \;|\; v_{\mathfrak{p}}(a)$, $|k (\mathfrak{p})|  \equiv -1 \mod \mathfrak{p}$. In this case, up to relabelling, the place $\mathfrak{P}_1$ satisfies $f(\mathfrak{P}_1| \mathfrak{p})=1$ and is inert in $L(r)$, whereas $f(\mathfrak{P}_2 | \mathfrak{p})=2$ and $\mathfrak{P_2}$ splits in $L(r)$. The signature of $\mathfrak{p}$ is $(1,1; 1,2)$. 
\end{enumerate}
\end{enumerate}  
\end{corollaire} 
\begin{proof} 
As a consequence of Theorem \ref{galoisglobal}, we may write the minimal polynomial $$T(X)=X^3-3X-a,$$ where $a = \frac{2A^2 -2 AB- B^2}{ A^2 - A B+ B^2}$, for some $A,B \in \mathcal{O}_{K}$. Let $\mathfrak{p}$ be a place of $K$. Since $L/K$ is Galois, by \cite[Theorem 2.3]{Con}, any root $r$ of the quadratic resolvent $R(X)= X^2 +3aX + ( -27 + 9 a^2)$ of the cubic polynomial $X^3-3X-a$ in $\overline{K}$ lies in $K$. In particular, the discriminant of $T(X)$, $\Delta= -27(a^2-4)$ (when $p\neq 2$) is a square in $K$. 
\begin{enumerate} 
\item When $v_{\mathfrak{p}}(a)>0 $, then by Kummer's theorem \cite[Theorem 3.3.7]{Sti}, and the splitting of the polynomial $X(X^2 -3 ) \mod \mathfrak{p}$, $\mathfrak{p}$ completely splits. 
\item When $v_{\mathfrak{p}}(a)= 0$, 

\begin{enumerate} 
\item[(a), (b)] if $|k (\mathfrak{p})|  \equiv 1 \mod 3$, then by Corollary \ref{splittingextqneq1mod31} and \ref{splittingextqneq1mod32}, $\mathfrak{p}$ is inert in $L$ if, and only if, the root of the polynomial $S(X)= X^2 + aX +1$ are non cubes modulo $\mathfrak{p}$. When $p=2$, $|k (\mathfrak{p})| =2^m$, and $m$ odd, $\mathfrak{p}$ is inert in $L$ if, and only if, the root of the polynomial $S(X)= X^2 + aX +1$ are non cubes in $\mathbb{F}_{2^{2m}}$. Note that the roots of the polynomial $S(X)$ are $-\frac{ A+ \xi^2 B}{A+ \xi B}$ and $-\frac{ A+ \xi B}{A+ \xi^2 B}$, since $a =\frac{ 2A^2 -2 AB- B^2 }{A^2 - A B+ B^2}$. Moreover, $-\frac{ A+ \xi^2 B}{A+ \xi B}$ is a cube modulo $\mathfrak{p}$ if, and only if, $-\frac{ A+ \xi B}{A+ \xi^2 B}$ is a cube modulo $\mathfrak{p}$. We thus obtain $(a)$ and $(b)$
\item[(c)]  when $|k (\mathfrak{p})|   \equiv -1 \mod 3 $ and $p \neq 2$, by Corollary  \ref{splittingextqneq1mod31}, $\mathfrak{p}$ is inert if, and only if, $|k (\mathfrak{p})|   =5$ and $a\equiv \pm 1\mod \mathfrak{p}$.
Moreover, $a =\frac{ 2A^2 -2 AB- B^2 }{A^2 - A B+ B^2} \equiv 1 \mod \mathfrak p$ if, and only if, $ 2A^2 -2 AB- B^2 = A^2 - A B+ B^2 \mod \mathfrak p$. Equivalently, $A^2 -AB-2 B^2 = (A -2B)(A+B)\equiv 0 \mod \mathfrak{p}$. That is, $A\equiv 2 B \mod \mathfrak{p}$ or $A \equiv -B \mod \mathfrak{p}$. \\ 
Also, $a =\frac{ 2A^2 -2 AB- B^2 }{A^2 - A B+ B^2} \equiv -1 \mod \mathfrak p$ if, and only if, $ 2A^2 -2 AB- B^2 =- A^2 + A B- B^2 \equiv 0 \mod \mathfrak p$. Equivalently, $3A^2 -3AB = 3A(A -B)\equiv 0 \mod \mathfrak{p}$. That is, $A\equiv B \mod \mathfrak{p}$ and $A\equiv 0 \mod \mathfrak{p}$. 
\end{enumerate} 
\item  When $v_{\mathfrak{p}}(a)< 0$, the result is immediate from Theorem \ref{splittingextqneq1mod3}.
\end{enumerate}  
\end{proof} 

\subsubsection{$X^3+aX+a^2$, $a\in K$, $p = 3$} 
As for purely cubic extensions, there exist a local standard form which is useful for a study of splitting and ramification.
\begin{lemma}\label{char3localstandardform1}
Let $p=3$, and let $L/K$ be a cubic separable extension. Let $\mathfrak{p}$ be a place of $K$. Then there is a generator $y$ such that $y^3 +a y +a ^2 =0$ such that $v_\mathfrak{p}(a) \geq 0$, or 
$v_\mathfrak{p} (a) <0 $ and $( v_\mathfrak{p} (a), 3)=1$. Such a $y$ is said to be {\sf in local standard form} at $\mathfrak{p}$.
\end{lemma}
\begin{proof} 
Let $\mathfrak{p}$ be a place of $K$. Let $y_1$ be a generator of $L/K$ such that $y_1^3 +a_1 y_1 + a_1^2=0$ (this was shown to exist in \cite{MWcubic}). By Lemma \ref{char3extensionsthesame}, any other generator $y_2$ with a minimal equation of the same form $y_2^3 +a_2 y_2 +a_2^2=0$ is such that $y_2 =-\beta (\frac{j}{a_1}y_1+ \frac{1}{a_1} w )$, and we have $$a_2  = \frac{(ja_1^2 + (w^3 + a_1 w) )^2}{a_1^3}.$$ Suppose that $v_\mathfrak{p}(a_1)  < 0$, and that $3 \;|\; v_\mathfrak{p}(a_1)$. Using the weak approximation theorem, we choose $\alpha \in K$ such that $v_\mathfrak{p}(\alpha) = 2v_\mathfrak{p}(a_1)/3$, which exists as $3 \;|\; v_\mathfrak{p}(a_1)$. Then $$v_\mathfrak{p}(\alpha^{-3} ja_1^2) = 0.$$ Let $w_0 \in K$ be chosen so that $w_0 \neq - \alpha^{-3} ja_1^2$ and $$v_\mathfrak{p}(\alpha^{-3} ja_1^2 + w_0) > 0.$$ 
This may be done via the following simple argument: As $v_\mathfrak{p}(\alpha^{-3} ja_1^2) = 0$, then $\overline{\alpha^{-3} ja_1^2} \neq 0$ in $k(\mathfrak{p})$. 

We then choose some $w_0 \neq - \alpha^{-3} ja_1^2 \in K$ such that $\overline{ w_0} = -\overline{\alpha^{-3} ja_1^2}$ in $k(\mathfrak{p})$. Note that $v_{\mathfrak{p}}(w_0)=0$. Thus, $\overline{\alpha^{-3} ja_1^2 + w_0}=0$ in $k(\mathfrak{p})$ and $v_\mathfrak{p}(\alpha^{-3} ja_1^2 + w_0) > 0$. As $p = 3$, it follows that the map $X \rightarrow X^3$ is an isomorphism of $k(\mathfrak{p})$, so we may find an element $w_1 \in K$ such that $w_1^3 = w_0 \mod \mathfrak{p}$. Hence $$v_\mathfrak{p}(\alpha^{-3} ja_1^2 + w_1^3) > 0.$$ We then let $w_2 = \alpha w_1$, so that $$v_\mathfrak{p}(j a_1^2 + w_2^3) = v_\mathfrak{p}(j a_1^2 + \alpha^3 w_1^3) > v_\mathfrak{p}(j a_1^2) .$$ Thus, as $v_\mathfrak{p}(a_1) < 0$, we obtain \begin{align*} v_\mathfrak{p}(ja_1^2 + (w_2^3 + a_1 w_2)) &\geq \min\{v_\mathfrak{p}(ja_1^2 + w_2^3  ) ,v_\mathfrak{p}(a_1 w_2) \} \\&> \min\{v_\mathfrak{p}(j a_1^2),v_\mathfrak{p}(a_1 w_2)\}\\& = \min\{v_\mathfrak{p}(j a_1^2),v_\mathfrak{p}(a_1) + 2v_\mathfrak{p}(a_1)/3\} \\& =  \min\{2v_\mathfrak{p}(a_1),5v_\mathfrak{p}(a_1)/3\} \\& = 2v_\mathfrak{p}(a_1). \end{align*} Hence $$v_\mathfrak{p}(a_2)  = v_\mathfrak{p}\left(\frac{(ja_1^2 + (w^3 + a_1 w) )^2}{a_1^3}\right) > 4v_\mathfrak{p}(a_1) - v_\mathfrak{p}(a_1^3) = v_\mathfrak{p}(a_1).$$ We can thus ensure (after possibly repeating this process if needed) that we terminate at an element $a_2 \in K$ for which $v_\mathfrak{p}(a_2) \geq 0$ or for which $v_\mathfrak{p}(a_2) < 0$ and $(v_\mathfrak{p}(a_2),3) = 1$. 
\end{proof} 

\begin{theorem} \label{char3localstandardform}
Suppose that $p = 3$. Let $L/K$ be a separable cubic extension and $y$ a primitive element with minimal polynomial $X^3 + aX + a^2$. Let $\mathfrak{p}$ be a place of $K$ and $\mathfrak{P}$ a place of $L$ above $\mathfrak{p}$. Then $\mathfrak{p}$ is fully ramified if, and only if, there is $w \in K$, $v_\mathfrak{p}(\alpha ) < 0$ and $(v_\mathfrak{p}(\alpha ),3) = 1$ with $$\alpha = \frac{(ja^2 + (w^3 + a w) )^2}{a^3}.$$ Equivalently, there is a generator $z$ of $L$ whose minimal polynomial is of the form $X^3 + \alpha X + \alpha^2$, where $v_\mathfrak{p}(\alpha ) < 0$ and $(v_\mathfrak{p}(\alpha ),3) = 1$. 
\end{theorem} 
\begin{proof} 
Let $\mathfrak{p}$ be a place of $K$, and denote by $\mathfrak{P}$ a place of $L$ above $\mathfrak{p}$. When $L/F$ is Galois, this theorem is simply the usual Artin-Schreier theory (see \cite[Proposition 3.7.8]{Sti}). Otherwise, since the discriminant of the polynomial $X^3+a X+a^2$ is equal to $\Delta=-4a^3=-a^3$, by \cite[Theorem 2.3]{Con}, we know that the Galois closure of $L/F$ is equal to $L(\Delta )= L(b)$, where $b^2 = -a$. Let $\mathfrak{p}_b$ a place of $K(b)$ above $\mathfrak{p}$. The extension $L(b)/K(b)$ is an Artin-Schreier extension with Artin-Schreier generator $y/b$ possessing minimal polynomial $X^3-X+b$. 
As $L(b)/K(b)$ is Galois, if $\mathfrak{p}_b$ is ramified in $L(b)$, then it must be fully ramified. Furthermore, as the degree $K(b)/K$ is equal to $2$, which is coprime with $3$, and the index of ramification is multiplicative in towers, it follows that the place $\mathfrak{p}$ is fully ramified in $L$ if, and only if, $\mathfrak{p}_b$ is fully ramified in $L(b)$. 
By \cite[Proposition 3.7.8]{Sti}, \begin{enumerate} \item $\mathfrak{p}_b$ is fully ramified in $L(b)$ if, and only if, there is an Artin-Schreier generator $z$ such that $z^3-z -c$ with $v_{\mathfrak{p}_b}(c) < 0$ and $(v_{\mathfrak{p}_b}(c),3)=1$, and \item $\mathfrak{p}_b$ is unramified in $L(b)$ if, and only if, there is an Artin-Schreier generator $z$ such that $z^3-z -c$ with $v_{\mathfrak{p}_b}(c) \geq 0$. \end{enumerate}
Suppose that there is a generator $w$ such that $w^3 +a_1 w +a_1 ^2 =0$, $v_\mathfrak{p} (a_1) <0 $ and $( v_\mathfrak{p} (a_1), 3)=1$. Then over $K(b_1)$, where $b_1^2 =-a_1$, we have an Artin-Schreier generator $z$ of $L(b_1)$ such that  $z^3-z +b_1$. Moreover, $$v_{\mathfrak{p}_{b_1}}(b_1)= \frac{ v_{\mathfrak{p}_{b_1}}(a_1)}{2}= \frac{e(\mathfrak{p}_{b_1} |\mathfrak{p}) v_{\mathfrak{p}}(a_1)}{2},$$ 
where $e(\mathfrak{p}_{b_1} |\mathfrak{p})$ is the index of ramification of $\mathfrak{p}_{b_1}$ over $K(b_1)$, whence $e(\mathfrak{p}_{b_1} |\mathfrak{p})=1$ or $2$. 
As a consequence, $$(v_{\mathfrak{p}_{b_1}}(b_1), 3)=( v_\mathfrak{p} (a_1), 3)=1,$$ and $\mathfrak{p}_{b_1}$ is fully ramified in $L(b_1)$, so that $\mathfrak{p}$ too must be fully ramified in $L$. 

Suppose that there exists a generator $w$ such that $w^3 +a_1 w +a_1 ^2 =0$, $v_\mathfrak{p} (a_1) \geq 0$. Then over $K(b_1)$, where $b_1^2 =-a_1$, we have a generator $z$ of $L(b_1)$ such that  $z^3-z +b_1$ and $$v_{\mathfrak{p}_{b_1}}(b_1)= \frac{e(\mathfrak{p}_{b_1} |\mathfrak{p}) v_{\mathfrak{p}}(a_1)}{2}\geq 0.$$ 
Thus $\mathfrak{p}_b$ is unramified in $L(b)$, so that $\mathfrak{p}$ cannot be fully ramified in $L$, since the ramification index is multiplicative in towers. The theorem then follows by Lemma \ref{char3localstandardform1}.
\end{proof} 
We finish the section describing the splitting of places in characteristic $3$. 
\begin{theorem} \label{splittingchar3}
Suppose that $p = 3$. Let $L/K$ be a separable cubic extension and $y$ a primitive element with minimal polynomial $X^3 + aX + a^2$. Let $\mathfrak{p}$ be a place of $K$. One can find $x\in K$ such that $\mathfrak{p}$ is a finite place over $\mathbb{F}_q(x)$ and the integral closure of $\mathbb{F}_q[x]$ in $L$ is a Dedekind Domain (see Remark \ref{rkp}). For this place $\mathfrak{p}$ and choice of $x$, we have:
\begin{enumerate}
\item $\mathfrak{p}\mathcal{O}_{L,x} = \mathfrak{P}^3$ where $\mathfrak{P}$ is a place of $L$ above $\mathfrak{p}$ if, and only if, $w \in K$, $v_\mathfrak{p}(\alpha ) < 0$ and $(v_\mathfrak{p}(\alpha ),3) = 1$ with $$\alpha = \frac{(ja^2 + (w^3 + a w) )^2}{a^3}.$$ 
The signature of $\mathfrak{p}$ is $(3,1)$. 
\end{enumerate}
Otherwise, by Lemma \ref{char3localstandardform1}, there is generator $z$ such that $z^3 +c z +c ^2 =0$ such that $v_\mathfrak{p}(c ) \geq 0$. In the following, we choose $z$ to be such a generator and $b\in \overline{K}$ such that $b^2 =-c$. We let $\emph{Tr}$ denote the trace map from $k(\mathfrak{p})$ into the prime field $\mathbb{F}_p$ (\S 2).
\begin{enumerate}
\item[(2)] If $v_\mathfrak{p}(c ) = 0$. Then:
\begin{enumerate}
\item $\mathfrak{p}$ is inert if, and only if, $-c$ is a square modulo $\mathfrak{p}$ and $\text{\emph{Tr}}(b) \notequiv 0 \mod \mathfrak{p}$. Moreover, $f(\mathfrak{P}|\mathfrak{p})=3$ for the place $\mathfrak{P}$ of $L$ above $\mathfrak{p}$. The signature of $\mathfrak{p}$ is $(1,3)$.
\item $\mathfrak{p}$ is completely split if, and only if, $-c$ is a square modulo $\mathfrak{p}$ and $\text{\emph{Tr}}(b) \equiv 0 \mod \mathfrak{p}$. 
The signature of $\mathfrak{p}$ is $(1,1; 1,1; 1,1)$. 
\item $\mathfrak{p}\mathcal{O}_{L,x} = \mathfrak{P}_1 \mathfrak{P}_2$ where $\mathfrak{P}_i$, $i=1,2$ are places of $L$ above $\mathfrak{p}$ if, and only if, $-c$ is not a square modulo $\mathfrak{p}$.  In this case, up to relabelling, the place $\mathfrak{P}_1$ satisfies $f(\mathfrak{P}_1| \mathfrak{p})=1$ and is inert in $L(b)$, whereas $f(\mathfrak{P}_2 | \mathfrak{p})=2$ and $\mathfrak{P_2}$ splits in $L(b)$. The signature of $\mathfrak{p}$ is $(1,1; 1,2)$. 
\end{enumerate}
\item[(3)] If $v_\mathfrak{p}(c ) >0$. Then:
\begin{enumerate}
\item $\mathfrak{p}$ is completely split if, and only if, $2 \;|\; v_\mathfrak{p} ( c)$ and $-s^2c$ is square modulo $\mathfrak{p}$ for some $s\in K$ such that $v_\mathfrak{p} ( -s^2 c) =0$. 
The signature of $\mathfrak{p}$ is $(1,1; 1,1; 1,1)$. 
\item $\mathfrak{p}\mathcal{O}_{L,x} = \mathfrak{P}_1 \mathfrak{P}_2$ where $\mathfrak{P}_i$, $i=1,2$ if, and only if, $2 \;|\; v_\mathfrak{p} ( c)$ and $-s^2c$ is not square modulo $\mathfrak{p}$ for some $s\in K$ such that $v_\mathfrak{p} ( s^2 c) =0$. In this case, up to relabelling, the place $\mathfrak{P}_1$ satisfies $f(\mathfrak{P}_1| \mathfrak{p})=1$ and is inert in $L(b)$, whereas $f(\mathfrak{P}_2 | \mathfrak{p})=2$ and $\mathfrak{P_2}$ splits in $L(b)$. The signature of $\mathfrak{p}$ is $(1,1; 1,2)$. 
\item $\mathfrak{p}\mathcal{O}_{L,x} = \mathfrak{P}_1^2 \mathfrak{P}_2$ where $\mathfrak{P}_i$, $i=1,2$ are places of $L$ above $\mathfrak{p}$ if, and only if, $(v_\mathfrak{p}( c) , 2)=1$ and $c$ is not a square modulo $\mathfrak{p}$. Moreover $f(\mathfrak{P}_i |\mathfrak{p})=1$, for $i=1,2$, $\mathfrak{P}_1$ is splits in $L(b)$ and $\mathfrak{P}_2$ is ramified in $L(b)$. The signature of $\mathfrak{p}$ is $(2,1; 1,1)$. 
\end{enumerate}
\end{enumerate}
\end{theorem}

\begin{proof}
Let $\mathfrak{p}$ be a place of $K$. 
\begin{enumerate}
\item This is simply Theorem \ref{char3localstandardform}.
\end{enumerate} 
If $\mathfrak{p}$ does not satisfy the conditions in case $(1)$,  then it is not fully ramified, and by Lemma \ref{char3localstandardform1}, there is a generator $z$ such that $z^3 +c z +c ^2 =0$ and $v_\mathfrak{p}(c) \geq 0$. In the following, we choose $z$ to be such a generator and $b\in \overline{K}$ such that $b^2 =-c$. 
\begin{enumerate} 
\item[2.]  If $v_\mathfrak{p}(c) = 0$, then $\overline{c}:= c \mod \mathfrak{p} \neq 0$ in $k(\mathfrak{p})$.  By Lemma \ref{splittingpolychar3partdeux}, we have that $f(X)= X^3 + cX + c^2$ splits as follows modulo $\mathfrak{p}$:
\begin{enumerate}
\item $f(X)$ is irreducible if, and only if, $-c$ is a non-zero square in $k(\mathfrak{p})$ and $\text{Tr}(b) \neq 0$
\item $f(X) = (X-\alpha)Q(X)$ where $\alpha \in k(\mathfrak{p})$ and $Q(X)$ is an irreducible quadratic polynomial if, and only if, $-c$ is not a square in $ \in k(\mathfrak{p})$. 
\item $f(X) = (X-\alpha)(X - \beta)(X - \gamma)$ if, and only if,  $-c$ is a non-zero square in $k(\mathfrak{p})$ and $\text{Tr}(b ) =0$.
\end{enumerate}
According to Kummer's theorem \cite[Theorem 3.3.7]{Sti}, \begin{enumerate} \item $f(X)$ is irreducible modulo $\mathfrak{p}$ if, and only if, $\mathfrak{p}$ is inert; \item $f(X) = (X-\alpha)Q(X)$ modulo $\mathfrak{p}$ where $\alpha \in \mathbb{F}_{3^m}$ and $Q(X)$ is an irreducible quadratic polynomial if, and only if, $\mathfrak{p}\mathcal{O}_{L,x} = \mathfrak{P}_1 \mathfrak{P}_2$ where $\mathfrak{P}_i$, $i=1,2$ are places of $L$ above $\mathfrak{p}$; and \item $f(X) = (X-\alpha)(X - \beta)(X - \gamma)$ modulo $\mathfrak{p}$ with $\alpha , \beta, \gamma \in k(\mathfrak{p})$ all distinct  if, and only if, $\mathfrak{p}$ is completely split. \end{enumerate} 
Together, these equivalences imply the splitting of $\mathfrak{p}$ in this case.
\item[3.] If $v_\mathfrak{p}(c) > 0$, then $L(b)/K(b)$ is an Artin-Schreier extension by \cite[Theorem 2.3]{Con}, and there is an Artin Schreier generator $w=\frac{z}{b}$ such that $w^3 - w + b=0$ and $v_{\mathfrak{p}_b} ( b)>0$, where $\mathfrak{p}_b$ is a place of $K(d)$ above $\mathfrak{p}$. Thus $b \equiv 0 \mod \mathfrak{p}_b$, and the polynomial $$X^3 -X +b \equiv X^3 - X \mod \mathfrak{p}_b$$ factors as $X(X-1)(X+1)$ modulo $\mathfrak{p}_b$. By Kummer's theorem (\cite[Theorem 3.3.7]{Sti}), we then have that $\mathfrak{p}_b$ is completely split in $L(b)$. 

As $\mathfrak{p}_b$ is completely split in $L(b )$, we have that $\mathfrak{p}$ cannot be inert in $L$. Indeed, if $\mathfrak{p}$ were inert in $L$, then there are at most two places above $\mathfrak{p}$ in $L(b)$, in contradiction with the proven fact that $\mathfrak{p}_b$ is completely split in $L(b)$. 

By \cite[p.55]{Neu}, $\mathfrak{p}$ splits completely in $L$ if, and only if, $\mathfrak{p}$ is completely split in $K(b)$ and $\mathfrak{p}_b$ is completely split in $L(b)$.

Also, since by the previous argument $\mathfrak{p}$ cannot be inert in $L$, we have that either $$\mathfrak{p}\mathcal{O}_{L,x} = \mathfrak{P}_1 \mathfrak{P}_2\qquad\text{or}\qquad\mathfrak{p}\mathcal{O}_{L,x} = \mathfrak{P}_1 \mathfrak{P}_2^2,$$ where $\mathfrak{P}_i$, $i=1,2$ are places of $L$ above $\mathfrak{p}$. Let $\mathfrak{P}_b$ be a place of $L(b)$ above $\mathfrak{p}$. When $\mathfrak{p}$ is inert in $K(b)$, the index of ramification of $e( \mathfrak{P}_b|  \mathfrak{p}) $ of $\mathfrak{p}$ over $L(b)$ is equal to $1$, since $\mathfrak{p}_b$ is completely split in $L(b )$, whence $\mathfrak{p}\mathcal{O}_{L,x} = \mathfrak{P}_1 \mathfrak{P}_2$. When $\mathfrak{p}$ is ramified in $K(b)$, then the index of ramification at any place above $\mathfrak{p}$ in $L(b)$ is divisible by $2$, since $L(b)/K$ is Galois by \cite[Theorem 2.3]{Con}, whence $\mathfrak{p}\mathcal{O}_{L,x} = \mathfrak{P}_1 \mathfrak{P}_2^2$. 

By studying the decomposition of a polynomial of the form $X^2 -d$ modulo $\mathfrak{p}$ over a finite field of characteristic not equal to $2$, using again Kummer's theorem \cite[Theorem 3.3.7]{Sti} and general Kummer Theory \cite[Proposition 3.7.3]{Sti}, we therefore find that: 
\begin{enumerate} 
\item $\mathfrak{p}$ is ramified in $K(b)$ if, and only if, $(v_\mathfrak{p}( c) , 2)=1$,
\item $\mathfrak{p}$ is completely split in $K(b)$ if, and only if, $2 \;|\; v_\mathfrak{p} ( c)$ and $-s^2c$ is a square modulo $\mathfrak{p}$ for some $s\in K$ such that $v_\mathfrak{p} ( s^2 c) =0$, and
 \item $\mathfrak{p}$ is inert in $K(b)$ if, and only if, $2 \;|\; v_\mathfrak{p} ( c)$ and $-s^2c$ is not square modulo $\mathfrak{p}$ for some $s\in K$ such that $v_\mathfrak{p} ( s^2 c) =0$. \end{enumerate} 
\end{enumerate} 
And the Theorem follows.

The inertia degrees are obtained as a consequence to the fundamental equality \cite[Proposition]{Sti}. The splitting of $\mathfrak{P}_i$ in $L(r)$ is an easy consequence of the study of the possible splitting of the place $\mathfrak{p}$ in the tower $L(b)/K(b)/K$, where $L(b)/K(b)$ is Galois, by \cite[Theorem 2.3]{Con}.
\end{proof}

\subsection{Riemann-Hurwitz formulae}
Using the extension data, it is possible to give the Riemann-Hurwitz theorem for each of our forms in Corollary \ref{classification}. These depend only on information from a single parameter.
\subsubsection{$X^3 -a$, $a\in K$, $p \neq 3$} 
\begin{lemma} 
 \label{diffexpqneq1mod3}
Let $p \neq 3$. Let $L/K$ be a purely cubic extension and $y$ a primitive element of $L$ with minimal polynomial $f(X) = X^3-a$. Let $\mathfrak{p}$ be a place of $K$ and $\mathfrak{P}$ a place of $L$ over $\mathfrak{p}$. Then the following are true:
\begin{enumerate} 
\item $d(\mathfrak{P}|\mathfrak{p}) =0$ if, and only if, $e(\mathfrak{P}|\mathfrak{p})=1$. 
\item  $d(\mathfrak{P}|\mathfrak{p}) =2$, otherwise. That is, by Theorem \ref{whatisthisidontknow}, $e(\mathfrak{P}|\mathfrak{p})=3$, which by Theorem \ref{RPC} is equivalent to $(v_{\mathfrak{p}}(a), 3)=1$.
\end{enumerate} 
\end{lemma} 
\begin{proof}
By Theorem \ref{whatisthisidontknow}, either $e(\mathfrak{P}|\mathfrak{p})=1$ or $e(\mathfrak{P}|\mathfrak{p})=2$.
 \begin{enumerate} 
\item  As the constant field $\mathbb{F}_q$ of $K$ is perfect, all residue field extensions in $L/K$ are automatically separable. The result then follows from \cite[Theorem 5.6.3]{Vil}.
\item If $e(\mathfrak{P}|\mathfrak{p})=3$, then as $p \nmid 3$, it follows again from [Theorem 5.6.3, Ibid.] that $d(\mathfrak{P}|\mathfrak{p}) = e(\mathfrak{P}|\mathfrak{p}) - 1 =   2$.
\end{enumerate}
\end{proof} 
We thus find the Riemann-Hurwitz formula as follows for purely cubic extensions when the characteristic is not equal to 3, which resembles that of Kummer extensions, but no assumption is made that the extension is Galois.
\begin{theoreme}[Riemann-Hurwitz I]  \label{RHPC}  Let $p \neq 3$. Let $L/K$ be a purely cubic  geometric extension, and $y$ a primitive element of $L$ with minimal polynomial $f(X) = X^3-a$. Then the genus $g_L$ of $L$ is given according to the formula $$g_L =  3g_K - 2 +  \sum_{\substack{  ( v_\mathfrak{p}(a),3) = 1}} d_{K}(\mathfrak{p}).$$
\end{theoreme}
\begin{proof}
This follows from Lemma \ref{diffexpqneq1mod3}, \cite[Theorem 9.4.2]{Vil}, and the fundamental identity $\sum e_i f_i = [L:K] = 3$. 
\end{proof} 
\subsubsection{$X^3 -3X -a$, $a\in K$, $p\neq 3$} 
\begin{lemma} 
 \label{diffexpqneq1mod3}
 Let $p \neq 3$. Let $L/K$ be an impurely cubic extension and $y$ a primitive element of $L$ with minimal polynomial $f(X) = X^3-3X-a$. Let $\mathfrak{p}$ be a place of $K$ and $\mathfrak{P}$ a place of $L$ over $\mathfrak{p}$. Let $\Delta=-27(a^2-4)$ be the discriminant of $f(X)$ and $r \in \overline{K}$ a root of the quadratic resolvent $R(X)= X^2 +3aX + ( -27 + 9 a^2)$ of $f(X)$. Then the following are true:
\begin{enumerate} 
\item $d(\mathfrak{P}|\mathfrak{p}) =0$ if, and only if, $e(\mathfrak{P}|\mathfrak{p})=1$. 
\item If $e(\mathfrak{P}|\mathfrak{p})=3$, which by Theorem \ref{ramification} is equivalent to $v_\mathfrak{p}(a) < 0$ and $(v_{\mathfrak{p}}(a), 3)=1$, then $d(\mathfrak{P}|\mathfrak{p}) =2$. 
\item If $e(\mathfrak{P}|\mathfrak{p})=2$, 
\begin{enumerate} 
\item If $p \neq 2$, by Corollary \ref{splittingextqneq1mod31}, this occurs precisely when $\Delta$ is not a square in $K$, $a\equiv \pm 2 \mod \mathfrak{p}$,  $( v_{\mathfrak{p}}(\Delta ), 2)=1$, and $2 \;| \;v_{\mathfrak{P}}(\Delta )$. In this case, $d(\mathfrak{P}|\mathfrak{p})=1$.
\item If $p = 2$, by Corollary \ref{splittingextqneq1mod32}, this occurs when $r\notin K$, $a\equiv 0 \mod \mathfrak{p}$, there is $w_\mathfrak{p} \in K$ such that $v_{\mathfrak{p}}(\left( \frac{1}{a^2}+1 -w_\mathfrak{p}^2+w_\mathfrak{p} \right) , 2)=1\quad\text{and}\quad v_{\mathfrak{p}}\left( \frac{1}{a^2}+1 -w_\mathfrak{p}^2+w_\mathfrak{p} \right)<0$. Also, in this case, there exists $\eta_\mathfrak{P} \in L$ such that $v_{\mathfrak{P}}\left( \frac{1}{a^2}+1 -\eta_\mathfrak{P} ^2+\eta_\mathfrak{P} \right)\geq 0,$ and we have for this $\mathfrak{P}$ that $$d(\mathfrak{P}|\mathfrak{p})=-v_{\mathfrak{p}}\left( \frac{1}{a^2} +1 -w_\mathfrak{p}^2+w_\mathfrak{p} \right)+1 .$$
\end{enumerate} 
\end{enumerate} 
\end{lemma} 
\begin{proof}
Let $\mathfrak{p}$ be a place of $K$, $\mathfrak{P}_r$ a place of $L(r)$ above $\mathfrak{p}$,  $\mathfrak{P}=\mathfrak{P}_r\cap L$, and $\mathfrak{p}_r=\mathfrak{P}_r\cap K(r)$.
 \begin{enumerate} 
\item  As the constant field $\mathbb{F}_q$ of $K$ is perfect, all residue field extensions in $L/F$ are automatically separable. The result then follows from \cite[Theorem 5.6.3]{Vil}.
\item If $e(\mathfrak{P}|\mathfrak{p})=3$, then as $p \nmid 3$, it follows again from [Theorem 5.6.3, Ibid.] that $d(\mathfrak{P}|\mathfrak{p}) = e(\mathfrak{P}|\mathfrak{p}) - 1 =   2$.
\item When $e(\mathfrak{P}|\mathfrak{p})=2$, 
\begin{enumerate} 
\item  if $p\neq 2$, then by [Theorem 5.6.3, Ibid.], $d(\mathfrak{P}|\mathfrak{p}) = e(\mathfrak{P}|\mathfrak{p}) - 1 =   1$.
\item if $p= 2$, then we work on the tower $L(r)/K(r)/ K$. 

If $e(\mathfrak{P}|\mathfrak{p})=2$, then $e(\mathfrak{p}_r|\mathfrak{p})=2$, $e(\mathfrak{P}_r|\mathfrak{p}_r)=1$ and $e(\mathfrak{P}_r|\mathfrak{P})=1$, by Theorem \ref{splittingextqneq1mod32}. As $p=2$, the extension $K(r)/K$ is Artin-Schreier and is generated by an element $\alpha$ such that $\alpha^2 - \alpha = \frac{1}{a^2}+1$. By Artin-Schreier theory (see \cite[Theorem 3.7.8]{Sti}), as $ e(\mathfrak{p}_r|\mathfrak{p})=2$, there exists an element $w_\mathfrak{p} \in K$ such that $$(v_{\mathfrak{p}}\left( \frac{1}{a^2}+1 -w_\mathfrak{p}^2+w_\mathfrak{p} \right) , 2)=1\qquad\text{and}\qquad v_{\mathfrak{p}}\left( \frac{1}{a^2}+1 -w_\mathfrak{p}^2+w_\mathfrak{p} \right)<0.$$ 
In addition, since $e(\mathfrak{P}_r|\mathfrak{P})=1$, there exists $\eta_\mathfrak{P} \in L$ such that $$v_{\mathfrak{P}}\left( \frac{1}{a^2}+1 -\eta_\mathfrak{P} ^2+\eta_\mathfrak{P} \right)\geq 0.$$ By Artin-Schreier theory (see \cite[Theorem 3.7.8]{Sti}), we obtain
 $$d(\mathfrak{p}_r | \mathfrak{p})=  -v_{\mathfrak{p}}\left( \frac{1}{a^2}+1 -w_\mathfrak{p}^2+w_\mathfrak{p} \right)+1.$$ By \cite[Theorem 5.7.15]{Vil}, we then find by equating differential exponents in the towers $L(r)/K(r)/K$ and $L(r)/L/K$ that
 $$ d( \mathfrak{P}_r| \mathfrak{p})= d( \mathfrak{P}_r| \mathfrak{P})+ e( \mathfrak{P}_r| \mathfrak{P}) d( \mathfrak{P}| \mathfrak{p})= d( \mathfrak{P}_r| \mathfrak{p}_r)+ e( \mathfrak{P}_r| \mathfrak{p}_r) d( \mathfrak{p}_r| \mathfrak{p}).$$
This implies that
 $$  d( \mathfrak{P}| \mathfrak{p})=  d( \mathfrak{p}_r| \mathfrak{p})= -v_{\mathfrak{p}}\left( \frac{1}{a^2}+1 -w_\mathfrak{p}^2+w_\mathfrak{p} \right)+1,$$
as $e(\mathfrak{P}_r|\mathfrak{P})=e(\mathfrak{P}_r|\mathfrak{p}_r)=1$ implies $d( \mathfrak{P}_r| \mathfrak{P})=d(\mathfrak{P}_r|\mathfrak{p}_r)=0$.
\end{enumerate} 
\end{enumerate} 
\end{proof}
We are now able to state and prove the Riemann-Hurwitz formula for this cubic form. 
\begin{theoreme}[Riemann-Hurwitz II] \label{RH} Let $p \neq 3$. Let $L/K$ be a cubic  geometric extension and $y$ a primitive element of $L$ with minimal polynomial $f(X) = X^3-3X-a$. Let $\Delta=-27(a^2-4)$ be the discriminant of $f(X)$ and $r$ a root of the quadratic resolvent $R(X)= X^2 +3aX + ( -27 + 9 a^2)$ of the cubic polynomial $X^3-3X-a$ in $\overline{K}$.
Then the genus $g_L$ of $L$ is given according to the formula
\begin{enumerate} 
\item If $p \neq 2$, then $$g_L =  3g_K - 2 +\frac{1}{2}\sum_{\mathfrak{p}\in \mathcal{S}} d_{K}(\mathfrak{p})+  \sum_{\substack{ v_\mathfrak{p}(a)<0\\ ( v_\mathfrak{p}(a),3) = 1}} d_{K}(\mathfrak{p}).$$
where $\mathcal{S}$ is the set of places of $K$ such that both $a\equiv \pm 2 \mod \mathfrak{p}$ and $  v_{\mathfrak{p}}(\Delta, 2)=1$. Moreover, the set $\mathcal{S}$ is empty when $\Delta$ is a square in $K$.
\item If $p = 2$, then $$g_L =  3g_K - 2 +  \frac{1}{2} \sum_{\mathfrak{p}\in \mathcal{S}} [ -v_{\mathfrak{p}}\left( \frac{1}{a^2}+1 -w_\mathfrak{p}^2+w_\mathfrak{p} \right)+1] d_{K}(\mathfrak{p})+  \sum_{\substack{ v_\mathfrak{p}(a)<0\\  ( v_\mathfrak{p}(a),3) = 1}} d_{K}(\mathfrak{p}),$$
where $\mathcal{S}$ is the set of places of $K$ such that both $a \equiv 0 \mod \mathfrak{p}$ and there exists $w_\mathfrak{p} \in K$ such that $v_{\mathfrak{p}}\left( \frac{1}{a^2}+1 -w_\mathfrak{p}^2+w_\mathfrak{p} \right)<0$ and $(v_{\mathfrak{p}}\left( \frac{1}{a^2}+1 -w_\mathfrak{p}^2+w_\mathfrak{p} \right) , 2)=1$. Moreover, the set $\mathcal{S}$ is empty when $r \in K$.
\end{enumerate}
\end{theoreme}
\begin{proof} \begin{enumerate} \item By \cite[Theorem 9.4.2]{Vil}, the term associated with a place $\mathfrak{P}$ of $L$ in the different $\mathfrak{D}_{L/F}$ contributes $\frac{1}{2} d_L(\mathfrak{P})^{d(\mathfrak{P}|\mathfrak{p})}$ to the genus of $L$, where $\mathfrak{p}$ is the place of $K$ below $\mathfrak{P}$, $d_L(\mathfrak{P})$ is the degree of the place $\mathfrak{P}$, and $d(\mathfrak{P}|\mathfrak{p})$ is the differential exponent of $\mathfrak{P}|\mathfrak{p}$. By the fundamental identity $\sum_i e_i f_i = [L:K] = 3$ for ramification indices $e_i$ and inertia degrees $f_i$ of all places of $L$ above $\mathfrak{p}$, we always have that $f_i=1$ whenever $\mathfrak{p}$ ramifies in $L$ (fully or partially). Thus from Lemma  \ref{diffexpqneq1mod3}, it follows that $d(\mathfrak{P}|\mathfrak{p}) = 2$ if $\mathfrak{p}$ is fully ramified, whereas $d(\mathfrak{P}|\mathfrak{p}) = 1$ if $\mathfrak{p}$ is partially ramified. The result then follows by reading off [Theorem 9.4.2, Ibid.] and using the conditions of Lemma \ref{diffexpqneq1mod3}. \item This follows in a manner similar to part (1) of this theorem, via Lemma \ref{diffexpqneq1mod3} for $p=2$. \end{enumerate}
\end{proof}
We obtain directly the following corollary when the extension $L/K$ is Galois.
\begin{corollaire} \label{RHGalois} Let $p \neq 3$. Let $L/K$ be a Galois cubic  geometric extension and $y$ a primitive element of $L$ with minimal polynomial $f(X) = X^3-3X-a$. Then the genus $g_L$ of $L$ is given according to the formula $$g_L =  3g_K - 2 +  \sum_{\substack{v_\mathfrak{p}(a)<0\\  ( v_\mathfrak{p}(a),3) = 1}} d_{K}(\mathfrak{p}).$$
\end{corollaire}
\subsubsection{$X^3 +aX +a^2$, $a\in K$, $p=3$} 

\begin{lemma}\label{char3diffexp}
Suppose that $p = 3$. Let $L/K$ be a separable cubic extension and $y$ a primitive element with minimal polynomial $X^3 + aX +  a^2$. Let $\mathfrak{p}$ be a place of $K$ and $\mathfrak{P}$ a place of $L$ above $\mathfrak{p}$. 
\begin{enumerate}
\item  $d(\mathfrak{P}|\mathfrak{p}) =0$ if, and only if, $e(\mathfrak{P}|\mathfrak{p})=1$. 

\item when $e(\mathfrak{P}|\mathfrak{p})=3$, by Theorem \ref{char3localstandardform}, there is $w_\mathfrak{p} \in K$ such that $v_\mathfrak{p}(\alpha_\mathfrak{p}  ) < 0$ and $(v_\mathfrak{p}(\alpha_\mathfrak{p}  ),3)=1$ with $$\alpha_\mathfrak{p}  = \frac{(ja^2 + (w_\mathfrak{p} ^3 + a w_\mathfrak{p} ) )^2}{a^3}.$$ Then $d ( \mathfrak{P}| \mathfrak{p}) = -v_{\mathfrak{p}}(\alpha_\mathfrak{p} )+2$.
\item $d(\mathfrak{P}|\mathfrak{p}) = 1$ whenever $e(\mathfrak{P}|\mathfrak{p})=2$. Moreover, by Lemma \ref{char3localstandardform1}, when $e(\mathfrak{P}|\mathfrak{p})=2$, there is generator $z_\mathfrak{p} $ such that $z_\mathfrak{p} ^3 +c_\mathfrak{p}  z_\mathfrak{p}  +c _\mathfrak{p} ^2 =0$ and $v_\mathfrak{p}(c_\mathfrak{p}  ) \geq 0$, $(v_\mathfrak{p} (c_\mathfrak{p}),2)=1$ and $2|v_\mathfrak{P} (c_\mathfrak{p})$. 
\end{enumerate}
\end{lemma} 
\begin{proof}
Let $b \in \overline{K}$ such that $b^2 = -a$, $\mathfrak{p}$ be a place of $K$, $\mathfrak{P}_b$ be a place of $L(b)$ above $\mathfrak{p}$, $\mathfrak{p}_b=\mathfrak{P}_b\cap K(b)$, $\mathfrak{P}= \mathfrak{P}_b\cap L$.
\begin{enumerate}
\item This is an immediate consequence of \cite[Theorem 5.6.3]{Vil}.
\item Suppose that $\mathfrak{p}$ is ramified in $L$, whence $\mathfrak{p}_b$ is ramified in $L(b)$. 
Moreover, by Theorem \ref{char3localstandardform}, there exists $w_\mathfrak{p} \in K$ such that $v_\mathfrak{p} (\alpha_\mathfrak{p}) <0$ and $(v_\mathfrak{p} (\alpha_\mathfrak{p}),3)=1$, 
where 
$$\alpha_\mathfrak{p} =  \frac{(ja^2 + (w_\mathfrak{p}^3 + a w_\mathfrak{p}) )^2}{a^3},$$
and furthermore, there exists a generator $z_\mathfrak{p}$ of $L$ such that $z_\mathfrak{p}^3 + \alpha_\mathfrak{p} z_\mathfrak{p} + \alpha_\mathfrak{p}^2=0$. Again by [Theorem 5.6.3, Ibid.], the differential exponent $d(\mathfrak{p}_b | \mathfrak{p})=d(\mathfrak{P}_b | \mathfrak{P})$ of $\mathfrak{p}$ over $K(b)$ (resp. $\mathfrak{P}$ over $L(b)$) is equal to
\begin{enumerate} 
\item $1$ if $\mathfrak{p}$ is ramified in $K(b)$, whence $e(\mathfrak{p}_b | \mathfrak{p})=e(\mathfrak{P}_b | \mathfrak{P})=2$, and 
\item $0$ if $\mathfrak{p}$ is unramified in $K(b)$, whence $e(\mathfrak{p}_b | \mathfrak{p})=e(\mathfrak{P}_b | \mathfrak{P})=1$.
\end{enumerate}
 By \cite[Theorem 2.3]{Con}, $L(b)/K(b)$ is Galois and $-\alpha_\mathfrak{p}$ is a square in $K(b)$. We write $-\alpha_\mathfrak{p} = \beta_\mathfrak{p}^2$. Moreover, $w_\mathfrak{p}= \frac{z_\mathfrak{p}}{\beta_\mathfrak{p}}$ and $w_\mathfrak{p}^3 - w_\mathfrak{p}- \beta_\mathfrak{p}=0$. Moreover, 
 $$v_{\mathfrak{p}_b}(\beta_\mathfrak{p})= \frac{ v_{\mathfrak{p}_b}(\alpha_\mathfrak{p})}{2}=\frac{ e(\mathfrak{p}_b | \mathfrak{p})v_{\mathfrak{p}}(\alpha_\mathfrak{p})}{2} $$
 with $e(\mathfrak{p}_b | \mathfrak{p})=2$ or $1$, depending on whether $\mathfrak{p}$ is ramified or not in $K(b)$. Also, $v_{\mathfrak{p}_b}(\beta_\mathfrak{p})=v_{\mathfrak{p}}(\alpha_\mathfrak{p})$ when $\mathfrak{p}$ is ramified in $K(b)$, whereas $v_{\mathfrak{p}_b}(\beta_\mathfrak{p})= \frac{ v_{\mathfrak{p}}(\alpha_\mathfrak{p})}{2} $ when $\mathfrak{p}$ is unramified in $K(b)$ (note that in this case $2|v_{\mathfrak{p}}(\alpha_\mathfrak{p})$).  Thus $v_{\mathfrak{p}_b}(\beta_\mathfrak{p})<0$ and $(v_{\mathfrak{p}_b}(\beta_\mathfrak{p}),3)=1$ and by \cite[Theorem 3.7.8]{Sti}, we also have that the differential exponent $d(\mathfrak{P}_b | \mathfrak{p}_b)$ of $\mathfrak{p}_b$ in $L(b)$ satisfies 
 $$d(\mathfrak{P}_b | \mathfrak{p}_b)=2 (-v_{\mathfrak{p}}(\beta_\mathfrak{p})+1).$$ 
By \cite[Theorem 5.7.15]{Vil}, the differential exponent of $\mathfrak{p}$ in $L(b)$ satisfies
 $$d(\mathfrak{P}_b | \mathfrak{p})= d(\mathfrak{P}_b | \mathfrak{p}_b) + e(\mathfrak{P}_b | \mathfrak{p}_b) d(\mathfrak{p}_b | \mathfrak{p})= d(\mathfrak{P}_b | \mathfrak{P}) + e(\mathfrak{P}_b | \mathfrak{P}) d(\mathfrak{P} | \mathfrak{p}).$$
 Thus,
 \begin{enumerate} 
 \item if $\mathfrak{p}$ is ramified in $K(b)=K(\beta_\mathfrak{p})$, that is, $(v_\mathfrak{p}(\alpha_\mathfrak{p}),2)=1$ by \cite[Proposition 3.7.3]{Sti}, then $ 2 (-v_{\mathfrak{p}}(\alpha_\mathfrak{p})+1) + 3= 1 +  2d(\mathfrak{P} | \mathfrak{p})$ and 
 $$d(\mathfrak{P} | \mathfrak{p})= -v_{\mathfrak{p}}(\alpha_\mathfrak{p})+2,$$ 
 whereas
  \item if $\mathfrak{p}$ is unramified in $K(b)$, that is, $2|v_\mathfrak{p}(\alpha_\mathfrak{p})$ again by \cite[Proposition 3.7.3]{Sti}, then also
  $$   d(\mathfrak{P} | \mathfrak{p})= 2 \left( -\frac{v_{\mathfrak{p}}(\alpha_\mathfrak{p})}{2}+1\right)=  -v_{\mathfrak{p}}(\alpha_\mathfrak{p})+2.$$
 \end{enumerate}
 \item This is immediate from Theorem \ref{splittingchar3} and \cite[Theorem 5.6.3]{Vil}, via application of the same method as in  Lemma \ref{diffexpqneq1mod3} $(3)$.
\end{enumerate} 
\end{proof} 
Finally, we use this to conclude the Riemann-Hurwitz formula for cubic extensions in characteristic 3.
\begin{theorem}[Riemann-Hurwitz III] \label{char3RH}
Suppose that $p = 3$. Let $L/K$ be a separable cubic extension and $y$ a primitive element with minimal polynomial $X^3 + aX + a^2$. Then the genus $g_L$ of $L$ is given according to the formula $$g_L =  3g_K - 2 + \frac{1}{2} \sum_{\mathfrak{p} \in S}\left(-v_{\mathfrak{p}}(\alpha_\mathfrak{p} )+2 \right) d_{K}(\mathfrak{p}) + \frac{1}{2} \sum_{\mathfrak{p} \in T} d_{K}(\mathfrak{p}),$$
where \begin{enumerate} \item $S$ is the set of places of $K$ for which there exists $w_\mathfrak{p}  \in K$ such that $v_\mathfrak{p}(\alpha_\mathfrak{p}  ) < 0$, $(v_\mathfrak{p}(\alpha_\mathfrak{p}  ),3)=1$ with $$\alpha_\mathfrak{p}  = \frac{(ja^2 + (w_\mathfrak{p} ^3 + a w_\mathfrak{p} ) )^2}{a^3},$$ 
and 
\item $T$ is the set of places of $K$ for which there is generator $z_\mathfrak{p} $ such that $z_\mathfrak{p} ^3 +c_\mathfrak{p}  z_\mathfrak{p}  +c_\mathfrak{p}  ^2 =0$,  $v_\mathfrak{p}(c_\mathfrak{p}  ) \geq 0$ and $(v_\mathfrak{p} (c_\mathfrak{p}), 2)=1$. \end{enumerate}
\end{theorem}
\begin{proof} This follows from Lemma \ref{char3diffexp}, \cite[Theorem 9.4.2]{Vil}, and the fundamental identity $\sum e_i f_i = [L:K] = 3$. 
\end{proof}

\bibliographystyle{plain}
\raggedright
\bibliography{references}

\begin{thebibliography}{10}

\bibitem{Con}
K.~Conrad.
\newblock Galois groups of cubics and quartics in all characteristics.
\newblock {\em Unpublished note}.

\bibitem{Dickson}
L.E. Dickson.
\newblock Criteria for the irreducibility of functions in a finite field.
\newblock {\em Bull. Amer. Math. Soc.}, pages 1--8, 1906.

\bibitem{LiNi}
R.~Lidl and H.~Niederreiter.
\newblock {\em Introduction to finite fields and their applications}.
\newblock Cambridge, 1986.

\bibitem{MWcubic}
S.~Marques and K.~Ward.
\newblock A complete classification of cubic function fields over any finite
  field.
\newblock {\em Preprint}, 2017.

\bibitem{Neu}
J.~Neukirch.
\newblock {\em Algebraic Number Theory}.
\newblock Springer, 1999.

\bibitem{Pomm}
K.~Pommerening.
\newblock Quadratic equations in finite fields of characteristic 2.
\newblock {\em Unpublished}, 2012.

\bibitem{RoWe}
P.~Rozenhart and J.~Webster.
\newblock Simple cubic function fields and class number computations.
\newblock {\em Preprint}, 2012.

\bibitem{Shanks}
D.~Shanks.
\newblock Class numbers of the simplest cubic fields.
\newblock {\em Math. Comp.}, 1:1137--1152, 1974.

\bibitem{Sti}
H.~Stichtenoth.
\newblock {\em Algebraic Function Fields and Codes}.
\newblock Springer, 2009.

\bibitem{Vil}
G.D. Villa-Salvador.
\newblock {\em Topics in the Theory of Algebraic Function Fields}.
\newblock Birkh\"{a}user, 2006.

\bibitem{Wash}
L.~Washington.
\newblock {\em Introduction to Cyclotomic Fields}.
\newblock Springer, 1982.

\bibitem{Williams}
K.S. Williams.
\newblock A note on cubics over {GF}$(2^n)$ and {GF}$(3^n)$.
\newblock {\em J. Num. Th.}, 7:361--365, 1975.

\end{thebibliography}
\vspace{.5cm}

\end{document}